\documentclass[reqno, 10pt]{amsart}

\usepackage{fullpage} 
\usepackage{diagbox}
\usepackage{hyperref}
\usepackage{etex}
\usepackage[shortlabels]{enumitem}
\usepackage{amsmath}
\usepackage{amsxtra}
\usepackage{amscd}
\usepackage{amsthm}
\usepackage{amsfonts}
\usepackage{amssymb}
\usepackage{eucal}
\usepackage[all]{xy}
\usepackage{graphicx}
\usepackage{tikz-cd}
\usepackage{mathrsfs}
\usepackage{subfiles}
\usepackage[colorinlistoftodos, textsize=tiny]{todonotes}
\usepackage{morefloats}
\usepackage{pdfpages}
\usepackage{thm-restate}
\usepackage[utf8]{inputenc}
\usepackage[T1]{fontenc}
\usepackage{epigraph}
\usepackage{csquotes}
\usepackage{moreverb}
\usepackage[margin=1in]{geometry}
\usepackage{tikz}
\graphicspath{ {images/} }

\usepackage{adjustbox}

\numberwithin{equation}{section}

\DeclareFontFamily{U}{wncy}{}
\DeclareFontShape{U}{wncy}{m}{n}{<->wncyr10}{}
\DeclareSymbolFont{mcy}{U}{wncy}{m}{n}
\DeclareMathSymbol{\Sha}{\mathord}{mcy}{"58} 

\usepackage{color}

\newcommand*{\Scale}[2][4]{\scalebox{#1}{$#2$}}%

\renewcommand\o{{\rm{O}}}
\newcommand\so{{\rm{SO}}}
\newcommand{\F}{\mathbb{F}}

\newcommand{\wt}[1]{\widetilde{#1}}
\newcommand{\frob}[0]{\operatorname{Frob}}

\newcommand{\Q}{\mathbb{Q}}
\newcommand{\Z}{\mathbb{Z}}
\newcommand{\bz}{\mathbb{Z}}
\newcommand{\bc}{\mathbb{C}}
\newcommand{\br}{\mathbb{R}}
\newcommand{\bq}{\mathbb{Q}}

\newcommand{\bp}{\mathbb{P}}
\newcommand{\ba}{\mathbb{A}}

\newcommand{\scf}{\mathscr{F}}

\newcommand{\mf}[1]{\mathfrak{#1}}

\newcommand{\ol}[1]{\overline{#1}}
\newcommand{\wh}[1]{\widehat{#1}}

\newcommand{\Cal}[1]{\mathcal{#1}}

\DeclareMathOperator{\et}{\text{\'{e}t}}

\newcommand{\co}{\colon}
\newcommand{\mrm}[1]{\mathrm{#1}}

\newcommand{\scr}[1]{\mathscr{#1}}
\newcommand{\EE}{\mathbb{E}}

\newcommand\scz{\mathscr Z}
\newcommand{\surj}{\twoheadrightarrow}

\newcommand{\dTV}{d_{\mrm{TV}}}

\DeclareMathOperator{\chr}{char}

\DeclareMathOperator{\Frob}{Frob}
\DeclareMathOperator{\coker}{coker}
\DeclareMathOperator{\cyc}{cyc}
\DeclareMathOperator{\N}{\mathbb{N}}

\DeclareMathOperator{\SO}{SO}

\DeclareMathOperator{\shom}{\mathscr{H}\mathrm{om}}

\DeclareMathOperator{\rank}{rk}
\DeclareMathOperator{\im}{im}

\DeclareMathOperator{\ord}{ord}

\DeclareMathOperator{\Lie}{Lie}

\DeclareMathOperator{\Isom}{Isom}

\DeclareMathOperator{\Frac}{Frac}

\DeclareMathOperator{\pic}{Pic}

\DeclareMathOperator{\Id}{Id}
\DeclareMathOperator{\Ad}{Ad}

\DeclareMathOperator{\Span}{Span}

\DeclareMathOperator{\spn}{Span}

\DeclareMathOperator{\tors}{tors}

\DeclareMathOperator{\rad}{rad}

\DeclareMathOperator{\gal}{Gal}
\DeclareMathOperator{\ogr}{OGr}

\DeclareMathOperator{\Sel}{Sel}
\DeclareMathOperator{\sel}{Sel}

\DeclareMathOperator{\prob}{Prob}

\DeclareMathOperator\sym{Sym}
\newcommand \xra{\xrightarrow}
\newcommand \ra{\rightarrow}
\DeclareMathOperator\spec{\text{Spec}}
\DeclareMathOperator\proj{\text{Proj}}
\newcommand\bg{{\mathbb G}}
\DeclareMathOperator\gl{GL}

\DeclareMathOperator\rk{rk}
\DeclareMathOperator\anrk{rk^{\text{an}}}
\DeclareMathOperator\id{id}
\DeclareMathOperator\mult{mult}

\makeatletter
\newcommand{\customlabel}[2]{\protected@write \@auxout {}{\string \newlabel {#1}{{#2}{\thepage}{#2}{#1}{}} }\hypertarget{#1}{#2}}

\newcommand\aff[2]{\mathbb A_{#2}^{12{#1}+3}} 
\newcommand\espace[2]{ {\mathscr W'}_{#2}^{#1}} 
\newcommand\estack[2]{\underline{\mathscr W'}_{#2}^{#1}} 
\newcommand\smestack[2]{{\underline{\mathscr W}^{\circ}}_{#2}^{#1}} 
\newcommand\sfestack[2]{{\underline{\mathscr W}^{\not\square}}^{#2}_{#1}} 
\newcommand\smespace[2]{\mathrm{\mathscr W^{\circ}}_{#2}^{#1}} 
\newcommand\sfespace[2]{\mathrm{\mathscr W^{\not\square}}_{#2}^{#1}} 

\newcommand\uespace[2]{{\mathscr U\!\!\mathscr W'}_{#2}^{#1}} 
\newcommand\usmespace[2]{\mathrm{\mathscr U\!\!\mathscr W^{\circ}}_{#2}^{#1}} 
\newcommand\usfespace[2]{\mathrm{\mathscr U\!\!\mathscr W^{\not\square}}_{#2}^{#1}} 
\newcommand\sspace[3]{{\mathrm{Sel}'}_{#1, #3}^{#2}} 
\newcommand\sstack[3]{\underline{\mathrm{Sel}'}_{#1, #3}^{#2}} 
\newcommand\smsstack[3]{{\underline{\mathrm{Sel}}^{\circ}}_{#1, #3}^{#2}} 
\newcommand\sfsstack[3]{{\underline{\mathrm{Sel}}^{\not\square}}_{#1, #3}^{#2}} 
\newcommand\smsspace[3]{{\mathrm{Sel}^{\circ}}_{#1, #3}^{#2}} 
\newcommand\sfsspace[3]{{\mathrm{Sel}^{\not\square}}_{#1, #3}^{#2}} 
\newcommand\ssheaf[3]{ {\mathcal{S}e\ell'}_{#1, #3}^{#2}} 
\newcommand\smssheaf[3]{{\mathcal{S}e\ell^{\circ}}_{#1, #3}^{#2}} 
\newcommand\psheaf[3]{\mathcal{S^{\circ}}_{#1, #3}^{#2}} 
\newcommand\esheaf[2]{\mathcal{E}[#1]_{#2}} 

\newcommand\vsel[3]{V_{#1}^{#2}} 
\newcommand\qsel[3]{Q_{#1}^{#2}} 

\newcommand\mono[3]{\rho_{#1,#3}^{#2}} 
\newcommand\osp{{\rm{O^*_{-}}}}
\newcommand\rvker[3]{\mathrm{RSel}_{#1, #3}^{#2}} 
\newcommand\rvsel[3]{\Sel_{#1}^{ #2}/#3(t)} 
\newcommand\rvrk[3]{\rk^{#2}/#3(t)} 
\newcommand\rvanrk[3]{\rk^{\mathrm{an,}#2}/#3(t)} 
\newcommand\rvo[2]{\mathrm{RSel}_{\vsel{#1}{#2}{}}^{\o(12{#2}-4, \mathbb F_{#1})}} 
\newcommand\rvcoset[2]{\mathrm{RSel}_{#1}^{#2}} 
\newcommand\zprime[1]{\widehat{\mathbb Z}^{(#1)}} 
\newcommand\dickson[1]{\mathrm{D}_{#1}} 
\renewcommand\sp[1]{\mathrm{sp}^{-}_{#1}} 
\newcommand\ansel[3]{(\mathrm{rk}^{\mathrm{an}},\mathrm{Sel}_{#1})^{#2}_{#3}} 
\newcommand\algsel[3]{(\mathrm{rk},\mathrm{Sel}_{#1})^{#2}_{#3}} 
\newcommand\flr[3]{(\mathrm{Rrk}, \mathrm{RSel}_{#1})^{#2}_{#3}} 
\newcommand\bklpr[1]{(\mathrm{rk}^{\mathrm{BKLPR}}, \mathrm{Sel}_{#1}^{\mathrm{BKLPR}})} 
\newcommand\ab[1]{\mathrm{Ab}_{#1}} 

\newcommand{\PP}{\mathbb{P}}

\newcommand\error[1]{{ \frac{-1}{216{#1}^2 - 162 {#1} + 31}}} 
\newcommand\smallrank[3]{\mathcal W^{#2, \anrk \leq 1}_{#1,#3}} 

\theoremstyle{plain}
\newtheorem{theorem}{Theorem}[section]
\newtheorem{proposition}[theorem]{Proposition}
\newtheorem{lemma}[theorem]{Lemma}

\newtheorem{corollary}[theorem]{Corollary}
\theoremstyle{definition}
\newtheorem{definition}[theorem]{Definition}
\newtheorem{remark}[theorem]{Remark}
\newtheorem{example}[theorem]{Example}

\newtheorem{warn}[theorem]{Warning}

\theoremstyle{remark}
\newtheorem{notation}[theorem]{Notation}
\numberwithin{equation}{section}

\makeatletter
\def\th@remark{%
  \thm@headfont{\bfseries}%
  \normalfont 
  \thm@preskip \thm@preskip 
  \thm@postskip\thm@preskip
}
\def\imod#1{\allowbreak\mkern5mu({\operator@font mod}\,\,#1)}
\makeatother

\numberwithin{equation}{section}

\def\listtodoname{List of Todos}
\def\listoftodos{\@starttoc{tdo}\listtodoname}

\title{The geometric distribution of Selmer groups of elliptic curves over function fields}

\author{Tony Feng, Aaron Landesman, Eric M. Rains}

\begin{document}

\begin{abstract}
	Fix a positive integer $n$ and a finite field $\mathbb F_q$. We study the joint distribution of the rank $\rk(E)$, the $n$-Selmer group $\mathrm{Sel}_n(E)$, and the $n$-torsion in the Tate-Shafarevich group $\Sha(E)[n]$ as $E$ varies over elliptic curves of fixed height $d \geq 2$ over $\mathbb F_q(t)$. 
We compute this joint distribution in the large $q$ limit.
We also show that the ``large $q$, then large height'' limit of this distribution agrees with the one predicted by Bhargava-Kane-Lenstra-Poonen-Rains. 
\end{abstract}

\maketitle
\tableofcontents

\section{Introduction}

\subsection{Arithmetic statistics of Selmer groups} The statistical behavior of Selmer groups has recently been the focus of much study. In \cite{bhargavaKLPR:modeling-the-distribution-of-ranks-selmer-groups}, remarkable probability distributions are introduced to model the distribution of the $n$-Selmer group $\Sel_n(E)$, for $E$ varying through isomorphism classes of elliptic
curves over a fixed global field. We refer to the these distributions, and the models which generate them, as the ``BKLPR heuristic.'' The BKLPR heuristic is consistent with all known results on the statistics of Selmer groups.

One can also consider the analogous question for elliptic curves over a global function field. The heuristics make sense in that case as well, and it is generally believed that in the ``large height, then large $q$'' limit, $\lim_{q \rightarrow \infty} \lim_{d \rightarrow \infty}$, the statistics of Selmer groups over global function fields should behave the same as in the case of number fields. For example, \cite{dJ02} computes the average size of 3-Selmer groups in this limit, and \cite{HLN14} computes the average size of 2-Selmer groups in this limit. Most notably, breakthrough work of Bhargava-Shankar \cite{bhargava-shankar:binary-quartic-forms-having-bounded-invariants, bhargavaS:ternary, bhargavaS:average-4-selmer, bhargavaS:average-5-selmer} computes the average size of $n$-Selmer groups for elliptic curves over number fields for $n=2,3,4,5$; the methods are expected to extend to global function fields with the same answers (and without taking a large $q$ limit!). The proofs of all these results rely on special features of small $n$, and confirming the BKLPR heuristic for the average size of $\Sel_n$ seems out of reach at present when $n>5$. Our goal is to nevertheless provide some partial evidence for the full BKLPR heuristic, by studying an easier version of the problem. 



To this end, we study the limiting process in the reversed order, $\lim_{d \rightarrow \infty} \lim_{q \rightarrow \infty}$ for elliptic curves over a rational function field $\F_q(t)$. This problem is significantly more accessible by algebraic geometry, which allows us to identify the distribution completely. Informally speaking, we show that in the ``large $q$, then large height'' limit, the distribution of $\Sel_n(E)$ is exactly as predicted by the BKLPR heuristic. A novel difficulty of this result is that it cannot be proved simply by computing and comparing the moments of the two distributions, because these distributions are not determined by their moments. 
Conversely, because the distribution is unbounded, convergence in distribution in the ``large $q$, then large height'' limit does not automatically imply convergence of the moments in these limits, though we do show the moments converge to the BKLPR moments as well.


\subsection{Statement of results}

\subsubsection{Some notation} We now introduce notation in order to state our main results precisely.
Let $p = \chr(\F_q)$.  For $p>2$, an elliptic curve $E$ over $\F_q(t)$ has a minimal Weierstrass model of the form 
\[
	y^2 = x^3 + a_2(t)x^2 + a_4(t) x + a_6(t),
\]
where $a_i(t)$ is a polynomial of degree $2id$ for $i \in \left\{ 1,2,3 \right\}$ (cf. \cite[\S4.2-4.8]{dJ02}). 
This value of $d$ is uniquely determined by $E$,
and we define $d =: h(E)$ to be the \emph{height} of $E$.
Let $\algsel n d {\mathbb F_q}$ denote the probability distribution assigning to a pair $(r,G)$,
for $r \in \bz$ and $G$ a finite abelian group, the proportion of isomorphism classes of
height $d$ elliptic curves
over $\mathbb F_q(t)$ with algebraic rank $r$ and $n$-Selmer group isomorphic to $G$ (see \autoref{definition:actual-distribution}).

\subsubsection{The BKLPR heuristic}

We summarize the BKLPR heuristic in \autoref{ssec: BKLPR}. Briefly put, it models the distribution of the $\ell^{\infty}$-Selmer group in terms of the intersection in $(\Q_\ell/\Z_\ell)^m$ induced by two maximal isotropic subspaces of $\Z_\ell^m$ (with the standard split quadratic form) as $m \rightarrow \infty$. Conditioned on the rank, the $\ell$-primary parts of the Selmer group are predicted to behave independently. This gives, in particular, a conjectural joint distribution $\bklpr n$ for the rank and $n$-Selmer group
of elliptic curves, described in \autoref{definition:bklpr-rank-selmer}.

\subsubsection{Main result} 

We consider the distribution $\algsel n d {\mathbb F_q}$ as a function on pairs $(r,G)$, where $r \in \Z$ and $G$ is an isomorphism class of finite abelian groups. Then we form 
\[
\limsup_{\substack{q \rightarrow \infty\\ \gcd(q,2n)=1}} \algsel n d {\mathbb F_q}\quad \text{and} \quad \liminf_{\substack{q \rightarrow \infty\\ \gcd(q,2n)=1}} \algsel n d {\mathbb F_q}
\]
as functions on $\{(r,G)\}$.\footnote{To spell this out: the $\liminf$ (resp. $\limsup$) of a distribution on the discrete set of $\{(r,G)\}$ is, by definition, the measure assigning to $(r,G)$ the $\liminf$ (resp. $\limsup$) of the probability that $(r,G)$ occurs.} (Note that because we are taking a pointwise $\liminf$ or $\limsup$, the resulting function may no longer be a probability distribution, i.e., its sum over all $(r,G)$ may not be $1$.) Our main result is the following,
which we deduce as a consequence of \autoref{theorem:large-q-distribution} and \autoref{theorem:kernel vs bklpr}: 
\begin{theorem}
	\label{theorem:main}
	For fixed integers $d \geq 2$ and $n \geq 1$, and $q$ ranging over prime powers, the limits 
	\[
	\lim_{d \rightarrow \infty} \limsup_{\substack{q \rightarrow \infty\\ \gcd(q,2n)=1}} \algsel n d {\mathbb F_q} \quad \text{and} \quad \lim_{d \rightarrow \infty} \liminf_{\substack{q \rightarrow \infty\\ \gcd(q,2n)=1}} \algsel n d {\mathbb F_q}
	\]
	exist, are equal to each other, and coincide with the distribution
	predicted by the BKLPR heuristic.
\end{theorem}

As far as we are aware, our results give the first \emph{direct} connection between the heuristics of \cite{bhargavaKLPR:modeling-the-distribution-of-ranks-selmer-groups} for general $n$ and the arithmetic of elliptic curves.
Further, our results suggest a potential approach to proving 
the conjectures of 
\cite{bhargavaKLPR:modeling-the-distribution-of-ranks-selmer-groups}
in the function field 
setting via homological stability techniques as used in \cite{EllenbergVW:cohenLenstra} to
prove a version of the Cohen-Lenstra heuristics over function fields.

\begin{remark}
	\label{remark:}
	One can deduce a more precise version of \autoref{theorem:main} with estimates on the error terms in the above limits directly from 
	\autoref{theorem:large-q-distribution} and \autoref{theorem:kernel vs bklpr}.
	One may also deduce the same result holds with algebraic rank replaced by analytic rank. Further, one may
	include the joint distribution of Tate-Shafarevich groups -- see \autoref{remark:tate-shafarevich-group}.
\end{remark}

\subsubsection{Summary of the main difficulties}\label{sssec: intro summary of content}

Experts will recognize that the distribution in this ``large $q$ limit'' is completely determined by certain monodromy representations. Letting $\smespace d B$ be the ``moduli space of smooth height $d$ elliptic surfaces'' (described more precisely in \autoref{subsection:notation-monodromy}) the relevant monodromy representations take the form 
$\mono n d B : \pi_1(\smespace d B) \ra \gl(\vsel n d B)$. Their significance lies in the fact that they control the number of connected components of moduli spaces parameterizing Selmer elements. Let us call the image of $\mono n d B $ the \emph{(arithmetic) monodromy group}, and the image of $\mono n d B |_{\pi_1((\smespace d B)_{\ol{\F}_q})}$ the \emph{geometric monodromy group}.


Let us talk through some of the difficulties in proving \autoref{theorem:main} in order to orient the reader where the content of the paper lies. First, it is important that we determine the monodromy group
precisely. If we had just wanted to compute the moments of $\Sel_n$, then it would have been enough to know that the geometry monodromy group is ``large enough.'' However, the behavior of the distribution depends more subtly on the arithmetic monodromy group. For example, it turns out that sometimes the Selmer distribution does not have a limit as $q \rightarrow \infty$, and this can happen even when $q$ is taken only over powers of a fixed odd prime $p$. Nevertheless, both the ``$\limsup_{q \rightarrow \infty}$'' and the ``$\liminf_{q \rightarrow \infty}$'' exist, and tend towards each other as the height tends to $\infty$. 

In a bit more detail, it is possible that for fixed height $d$, the Selmer distribution does not have a well defined limit as $q \to \infty$.
Specifically, the 
$\limsup_{q \rightarrow \infty}$ and $\liminf_{q \rightarrow \infty}$
do not agree when, for an infinite sequence of $q$'s over which the limits run, the \emph{arithmetic} monodromy group contains an element of non-trivial spinor norm (see \autoref{subsubsection:spinor}) but the \emph{geometric} monodromy group does not. In this case, the arithmetic monodromy group fluctuates between two possibilities, which ends up creating a discrepancy between $\limsup_{q \rightarrow \infty} $ and $\liminf_{q \rightarrow \infty}$. 

A second substantial issue is that even after having determined the monodromy representations that control the Selmer groups, it is not straightforward to identify the resulting distribution with the BKLPR heuristic. (To be clear, this is a purely combinatorial question, although it turns out to require techniques from algebraic geometry, number theory, etc. to address.) 
The reason for this difficulty is that the BKLPR heuristic is not described in terms of explicit closed formulas, but in terms of a random algebraic model. For example, it is not determined by its moments, as illustrated in \autoref{example:moments-and-distribution} below. 
In order to compare the BKLPR distribution to the distribution coming from a monodromy representation, we introduce a ``random kernel model'' that mediates between the two distributions. We observe that both the BKLPR heuristic and the random kernel model enjoy Markov properties which reduce their comparison to simpler cases that can be computed explicitly, by matching enough moments. (Even this is a little oversimplified: what we need is to establish enough control on the moments already at a ``finite height'' level-- see \autoref{sec: prime distribution}.)

\subsubsection{Defining the random variables} 
\label{subsection:random-variables} 

In order to state the next results, we will need to introduce some more notation. 

Let $\ab n$ denote the set of isomorphism classes of finite $\bz/n\bz$-modules. We will next define several distributions on $\bz_{\geq 0} \times \ab n$ modeling the joint distribution of the rank and $n$-Selmer group of an elliptic curves. For $E$ an elliptic curve, we use $\rk(E)$ to denote the algebraic rank of $E$ and $\anrk(E)$ to denote the analytic rank of $E$.
In what follows, we use $E$ to denote an isomorphism class of elliptic curves.

\begin{definition}
	\label{definition:actual-distribution}
	For $n, d \in \bz_{\geq 1}$ and $k$ a finite field, let $\algsel n d k$ and $\ansel n d k$ be the distributions on 
	$\bz_{\geq 0} \times \ab n$
	given by 
\begin{align*}
	\prob(\algsel n d k = (r,G)) &= \frac{\#\{E/k(t) \colon h(E) = d, \rk(E) = r, \Sel_n(E) \simeq G\}}{\# \{E/k(t) \colon h(E) = d\}}  \\
	\prob(\ansel n d k = (r,G)) &= \frac{\#\{E/k(t) \colon h(E) = d, \anrk(E) = r, \Sel_n(E) \simeq G\}}{\# \{E/k(t) \colon h(E) = d\}},
\end{align*}
where $E$ varies over isomorphism classes of elliptic curves over $k(t)$.
Also, define the distribution $\rvsel n d k$ on $\ab n$ by
\begin{align*}
	\prob(\rvsel n d k = G) = \frac{\#\{E/k(t) \colon h(E) = d, \Sel_n(E) \simeq G\}}{\# \{E/k(t) \colon h(E) = d\}}
\end{align*}
and define the distributions $\rvrk n d k$, $\rvanrk n d k$ on $\bz_{\geq 0}$ by
\begin{align*}
	\prob(\rvrk n d k = r) &= \frac{\#\{E/k(t) \colon h(E) = d, \rk(E) = r\}}{\# \{E/k(t) \colon h(E) = d\}}  \\
	\prob(\rvanrk n d k = r) &= \frac{\#\{E/k(t) \colon h(E) = d, \anrk(E) = r\}}{\# \{E/k(t) \colon h(E) = d\}}.
\end{align*}

For a random variable $X$, we let $\EE[X]$ be denote the expected value of $X$ (if it exists). 

\end{definition}
\begin{remark}
	\label{remark:distribution-definition-variations}
	In \autoref{definition:actual-distribution}, for the purposes of computing these distributions in the limit $q \ra \infty$, we could equally well
	replace the condition $h(E) = d$ by the condition $h(E) \leq d$.
	The reason for this is that isomorphism classes of curves with $h(E) < d$ are parameterized by $k$ points of the stack $\estack i k$ (defined below in \autoref{subsubsection:elliptic-stack})
	for $i < d$, which is a finite type global quotient stack of strictly smaller
	dimension than $\estack d k$. 
	Hence, 
	$\cup_{i \leq d} \estack i k$
	will 
	only contributes at most $O_{n,d}(q^{-1/2})$ to the probability distributions in question, as can be deduced from the Lang-Weil estimate and \cite[Lemma 5.3]{landesman:geometric-average-selmer}.

	For analogous reasons, one can equally well weight the above counts by automorphisms (which would be the correct ``stacky way'' to count points)
	and the distribution in the $q \ra \infty$ limit will remain the same. Note that after excising the locus of elliptic curves with more than $2$ automorphisms,
	there will be a factor of one half in both the numerator and denominator in the definition
	of the distributions in \autoref{definition:actual-distribution}, which cancel out.
\end{remark}

\subsubsection{Some consequences}

The following result (which is part of	 \autoref{corollary: minimalist-precise}) is a variant of the Katz-Sarnak minimalist conjecture, stating that for fixed height, in the large $q$ limit, the average rank is $1/2$. Moreover, in the large $q$ limit, the rank takes value $1$ and $0$ with probability $1/2$, and takes value $\geq 2$ with probability $0$. It can also be deduced from \cite[Theorem 13.3.3]{katz:moments-monodromy-and-perversity},
though the more precise error terms given in \autoref{corollary: minimalist-precise} do not directly follow from 
\cite[Theorem 13.3.3]{katz:moments-monodromy-and-perversity}. We note that the fact that elliptic curves in the large $q$ limit have rank $0$ with probability $1/2$ is not a direct consequence of \autoref{theorem:main}, but it comes out of the more refined analysis used to prove \autoref{theorem:main} for $n = \ell$ a prime.\footnote{However, the statement that elliptic curves in the large $q$ limit have rank at least $2$ with probability $0$ does follow from just the computation of the
average size of $\# \sel_n$, see \cite[Corollary 1.3]{landesman:geometric-average-selmer}.}

\begin{proposition}[Large $q$ analog of \protect{\cite[Conjecture 1.2]{poonenR:random-maximal-isotropic-subspaces-and-selmer-groups}}]
	\label{corollary: minimalist}
	For fixed integers $d \geq 2$ and $n \geq 1$, 
	we have
		\begin{align}
	\lim\limits_{\substack{q \ra \infty \\ \gcd(q,2n)=1}} 
	\prob(\rvrk n d {\mathbb F_q} = r) &=
		\begin{cases}
			1/2 & \text{ if } r \leq 1, \\
			0  & \text{ if } r \geq 2.
		\end{cases}  \\
	\end{align}
		Furthermore, 
		\[
\lim\limits_{\substack{q \ra \infty \\ \gcd(q,2n)=1}} \EE[\rvrk n d {\mathbb F_q}] = 1/2 .
\]
\end{proposition}

The following calculation of the geometric moments of Selmer groups is a consequence of \autoref{corollary:squarefree-moments-precise}, which includes more precise error terms.

\begin{theorem}[Large $q$ analog of \protect{\cite[Conjecture 1.4]{poonenR:random-maximal-isotropic-subspaces-and-selmer-groups}}]
	\label{corollary:squarefree-moments}
	Let $n$ be a squarefree positive integer, $d \geq 2$, and $\omega(n)$ be the number of prime factors of $n$.
	\begin{enumerate}
		\item Fix $c_\ell \in \bz_{\geq 0}$ for each prime $\ell \mid n$. Then
 \begin{equation}
	\label{equation:poonen-rains-squarefree-prediction}
\Scale[0.8]{\begin{aligned}
	&\lim_{d \ra \infty}\limsup_{\substack{q \ra \infty \\ \gcd(q,2n)=1}} \prob\left( \rvsel n d {\mathbb F_q} \simeq \prod_{\ell \mid n} \left( \bz/\ell \bz \right)^{c_\ell} \right) =\lim_{d \ra \infty}\liminf_{\substack{q \ra \infty \\ \gcd(q,2n)=1}} \prob\left( \rvsel n d {\mathbb F_q} \simeq \prod_{\ell \mid n} \left( \bz/\ell \bz \right)^{c_\ell} \right) \\
	&= 
	\begin{cases}
	2^{\omega(n)-1} \prod_{\ell \mid n} \left( \left( \prod_{j \geq 0} \left( 1-\ell^{-j} \right)^{-1} \right) \left( \prod_{j=1}^{c_\ell} \frac{\ell}{\ell^j-1} \right) \right)
		& \text{ if all $c_\ell$ have the same parity},  \\
		0 & \text{otherwise.}  \\
	\end{cases}
\end{aligned}}
\end{equation}
		\item We have 
		\[
		\lim\limits_{\substack{q \ra \infty \\ \gcd(q,2n)=1}}  \EE[ \# \rvsel n d {\mathbb F_q}] = \sigma(n) := \sum_{s \mid n} s.
		\]
		\item For $m \leq 6d-3$, we have
		\[
		\lim\limits_{\substack{q \ra \infty \\ \gcd(q,2n)=1}}  \EE[(\# \rvsel n d {\mathbb F_q})^m] = \prod_{\text{prime }\ell \mid n} \prod_{i=1}^m \left( \ell^i +1 \right).
		\]
	\end{enumerate}
\end{theorem}

The following corollary is the more familiar case of \autoref{corollary:squarefree-moments} when $n$ is taken to be a prime $\ell$.
One can also deduce a version with explicit error terms in $q$, as in \autoref{corollary:squarefree-moments-precise}.

\begin{corollary}[Large $q$ analogue of \protect{\cite[Conjecture 1.1]{poonenR:random-maximal-isotropic-subspaces-and-selmer-groups}}]
	\label{corollary:prime-moments}
	Let $\ell$ be a prime, and $d \geq 2$.
	\begin{enumerate}
		\item We have
\[\Scale[0.8]{		\begin{aligned}
&			\lim_{d \ra \infty} \limsup_{\substack{q \ra \infty \\ \gcd(q,2\ell)=1}} \prob\left( \rvsel \ell d {\mathbb F_q} = (\bz/\ell \bz)^c \right) = 
			\lim_{d \ra \infty} \liminf_{\substack{q \ra \infty \\ \gcd(q,2\ell)=1}} \prob\left( \rvsel \ell d {\mathbb F_q} = (\bz/\ell \bz)^c \right) \\
			  & \hspace{1cm}= 
		\left( \prod_{j \geq 0} \left( 1-\ell^{-j} \right)^{-1} \right) \left( \prod_{j=1}^{c} \frac{\ell}{\ell^j-1} \right).
\end{aligned}}
\]
		\item We have
		\[
			\lim\limits_{\substack{q \ra \infty \\ \gcd(q,2\ell)=1}}  \EE[ \# \rvsel \ell d {\mathbb F_q}] = \sigma(\ell) := \ell+1.
		\]
		\item For $m \leq 6d-3$ the $m$th moment of $\rvsel \ell d {\mathbb F_q}$ is 
			\[
		\lim\limits_{\substack{q \ra \infty \\ \gcd(q,2n)=1}}  \EE[(\# \rvsel \ell d {\mathbb F_q})^m] = \prod_{i=1}^m \left( \ell^i +1 \right).
			\]
	\end{enumerate}
\end{corollary}

\begin{remark}[Distributions of Tate-Shafarevich groups]
	\label{remark:tate-shafarevich-group}
	Throughout this paper, we mostly work with the joint distribution of ranks and $n$-Selmer groups of elliptic curves,
	while \cite{bhargavaKLPR:modeling-the-distribution-of-ranks-selmer-groups} also makes predictions for Tate-Shafarevich
	groups of elliptic curves.
	Indeed, as an easy consequence of our results, we obtain analogous predictions for Tate-Shafarevich groups,
	as we now explain.
	For $E$ a torsion free elliptic curve over $\mathbb F_q(t)$, we have an exact sequence
\begin{equation}
	\label{equation:}
	\begin{tikzcd}
		0 \ar {r} & \left( \bz/n\bz \right)^{\rk E} \ar {r} & \Sel_n(E) \ar {r} & \Sha(E)[n] \ar {r} & 0.
\end{tikzcd}\end{equation}
Note that the torsion freeness condition is satisfied $100\%$ of the time
\cite[Lemma 5.7]{bhargavaKLPR:modeling-the-distribution-of-ranks-selmer-groups}.
Therefore, the algebraic rank and $n$-Selmer group of $E$ determines $\Sha(E)[n]$, and hence the joint distribution of algebraic ranks, and $n$-Selmer groups
determines the joint distribution of algebraic ranks, $n$-Selmer groups, and $n$-torsion in Tate-Shafarevich groups.
Let $(\mathrm{rk}^{\mathrm{BKLPR}}, \mathrm{Sel}_{n}^{\mathrm{BKLPR}}, \Sha[n]^{\mathrm{BKLPR}})$ denote the conjectural joint distribution
for ranks, $n$-Selmer groups, and $n$-torsion in Tate-Shafarevich groups described in \cite[\S5.7]{bhargavaKLPR:modeling-the-distribution-of-ranks-selmer-groups}
and let 
$(\mathrm{rk}, \mathrm{Sel}_{n}, \Sha[n])_{\mathbb F_q}^d )$
denote the joint distribution of algebraic ranks, $n$-Selmer groups, and $n$-torsion in Tate-Shafarevich groups
of height $d$ elliptic curves over $\mathbb F_q$.
Then, it follows from \autoref{theorem:main} and the above remarks that
\begin{align*}
	(\mathrm{rk}^{\mathrm{BKLPR}}, \mathrm{Sel}_{n}^{\mathrm{BKLPR}}, \Sha[n]^{\mathrm{BKLPR}})
	&= 
	\lim_{d \ra \infty} \left( \limsup\limits_{\substack{q \ra \infty\\ \gcd(q,2n)=1}} 
		(\mathrm{rk}, \mathrm{Sel}_{n}, \Sha[n])_{\mathbb F_q}^d  \right)  \\
	& = 
	\lim_{d \ra \infty} \left(\liminf\limits_{\substack{q \ra \infty\\ \gcd(q,2n)=1}} 
		(\mathrm{rk}, \mathrm{Sel}_{n}, \Sha[n])_{\mathbb F_q}^d
 \right).
\end{align*}

One can also bound the error in these limits using \autoref{theorem:large-q-distribution} and \autoref{theorem:kernel vs bklpr}.
We note that for fixed height $d \geq 2$, the proportion of elliptic curves of height up to $d$ over $\mathbb F_q$ with analytic rank equal to algebraic rank
tends to $1$ as $q \ra \infty$ over prime powers $q$ with $\gcd(q,2) = 1$.
	This follows from \autoref{theorem:main} and \autoref{proposition:component-and-rank}.
	Therefore, the Birch and Swinnerton-Dyer Conjecture holds for all such curves, implying the Tate-Shafarevich group is finite for all such curves.
\end{remark}

\begin{remark}[Families of quadratic twists]
In other families of elliptic curves, such as quadratic twist families, the ``geometric distribution'' will similarly be controlled by the analogous monodromy representations to those described in \S \ref{sssec: intro summary of content}. Adapting our arguments will yield similar results for such families whenever the geometric monodromy group is large enough. However, the precise distribution that results depends rather delicately on the precise monodromy group, for the same reasons as described in \S \ref{sssec: intro summary of content}.

For example, 
in forthcoming work
\cite{parkW:average-selmer-rank-in-quadratic-twist-families},
Park and Wang carry out an analog of the results of \cite{landesman:geometric-average-selmer} for quadratic twist families of elliptic curves, at least
in the case of $n$-Selmer groups for $n$ prime.
We note this should often be extendable to composite $n$, see \cite[Remark 1.7]{landesman:geometric-average-selmer}.
Suppose one chooses a quadratic twist family 
such that the sheaf on that family constructed analogously to $\psheaf n d B$ on the universal family
has geometric monodromy containing the commutator of the relevant orthogonal group,
but with nontrivial Dickson invariant (see \autoref{subsubsection:dickson}).
Given such a family, via similar arguments to those in this paper,
if one first takes $\liminf_{q \rightarrow \infty}$ or $\limsup_{q \rightarrow \infty}$, and then a large height limit, the joint distribution of the rank and $n$-Selmer
group will agree with $\bklpr n$.
We note that triviality or nontriviality of the Dickson invariant can often be verified for explicit examples, as in the proof of
\cite[Theorem 4.1]{zywina:inverse-orthogonal}. 

On the other hand, it is possible for the Dickson invariant to be trivial in quadratic twist families;
explicit such examples are constructed in \cite[\S5 and \S6]{zywina:inverse-orthogonal}.
In these cases, the distribution of ranks and Selmer groups in the quadratic twist family will differ from those predicted in \cite{bhargavaKLPR:modeling-the-distribution-of-ranks-selmer-groups}.
E.g., the minimalist conjecture will fail as 100\% of elliptic curves in such families will have rank $0$.
Nevertheless, for sufficiently high degree twists, the large $q$ limit $m$th moments in these quadratic twist families
will agree with those predicted in \cite{bhargavaKLPR:modeling-the-distribution-of-ranks-selmer-groups}.
Additionally, it is possible to choose quadratic twist families where the relevant geometric monodromy does not contain the commutator of the relevant orthogonal group,
in which case the large $q$ limit statistics of ranks and Selmer groups may differ drastically from those predicted in \cite{bhargavaKLPR:modeling-the-distribution-of-ranks-selmer-groups}.
\end{remark}

\begin{remark}[The inverse Galois problem]
	\label{remark:inverse-galois}
	For $\ell$ a prime, let $\qsel \ell d k$ denote the quadratic form defined in \autoref{definition:selmer-space-quadratic-form}, which we note has discriminant $1$ and hence
	is equivalent to the standard quadratic form $x_1x_2 + x_3 x_4 + \cdots + x_{12d-5}x_{12d-4}$.
	In order to prove \autoref{theorem:main}, we perform a certain monodromy computation in \autoref{theorem:monodromy}, which shows that
	for even $d \geq 2$, and $\ell \nmid d-1$,
	$\o(\qsel \ell d k)$ occurs as a Galois group over $\mathbb Q(t_1, \ldots, t_{10d+2})$, and hence also as a Galois group over $\bq$ by Hilbert
	irreducibility (\cite[\S 9.2, Proposition 2]{serre1989lectures} in conjunction with \cite[\S 13.1, Theorem 3]{serre1989lectures}).
	To our knowledge, it was not previously known that these groups all appear as Galois groups over $\mathbb Q$.

	Closely related constructions to ours are given in \cite[Theorem 1.1]{zywina:inverse-orthogonal}, and the techniques of \cite{zywina:inverse-orthogonal}
	can likely be adapted to construct the Galois groups $\o(\qsel \ell d k)$ when $\ell \geq 5$. 
	However, our results also apply in the cases $\ell = 2$ and $\ell = 3$, to which the techniques of \cite{zywina:inverse-orthogonal} seem not to apply.
\end{remark}

\begin{remark}
	\label{remark:}
	An interesting byproduct of the proof of \autoref{theorem:main} is that the analytic rank of an elliptic curve over $\mathbb F_q(t)$
	with smooth minimal proper regular model
	is realized as the dimension of the generalized $1$-eigenspace of a certain matrix associated to an action of Frobenius (see \autoref{lemma:generalized-eigenspace-is-rank})
	while the $\ell^\infty$-Selmer rank is the dimension of the $1$-eigenspace of that same matrix (see \autoref{lemma:selmer-as-frob-fixed}).
	These dimensions agree for $100\%$ of elliptic curves of fixed height $d$ over $\mathbb F_q(t)$ in the large $q$ limit
and also agree with the rank of the elliptic curve (see \autoref{proposition:component-and-rank}).
Hence, at least in the function field setting, this gives an answer to the question raised in \cite[Remark 1.1.4]{parkPVW:heuristic-for-boundedness-of-ranks-of-elliptic-curves}
as to whether there exists a natural matrix coming from the arithmetic of elliptic curves giving rise to the rank and Selmer group of an elliptic curve.
\end{remark}

\begin{example}[A distribution not determined by its moments]
	\label{example:moments-and-distribution}
	Consider the three distributions 
	\begin{align*}
	&	\bklpr n, \\
	&(\bklpr n | \rk^{\mrm{BKLPR}} \equiv 0 \mod 2), \\
	&(\bklpr n | \rk^{\mrm{BKLPR}} \equiv 1 \mod 2),
	\end{align*}
with the latter two the distributions conditioning upon whether the rank is even or odd.
	These give examples of three distinct distributions which we claim have the same $m$th moments for all $m \geq 0$.

	We now justify why the moments of these three distributions agree. For simplicity, we assume $n$ is prime, though the same claim holds true for general composite $n$,
	as can be deduced from the Markov properties verified in \autoref{section:markov}.
	By \autoref{theorem:kernel vs bklpr},
	the above three distributions 
	agree with the three distributions 
\begin{align*}
	&\lim_{d \ra \infty} \liminf\limits_{\substack{q \ra \infty}} \flr n d {\mathbb F_q}, \\
	&\lim_{d \ra \infty} \liminf\limits_{\substack{q \ra \infty}}(\flr n d {\mathbb F_q}| \rk \equiv 0 \mod 2), \\
	&\lim_{d \ra \infty} \liminf\limits_{\substack{q \ra \infty}} (\flr n d {\mathbb F_q}| \rk \equiv 1 \mod 2)
\end{align*}
respectively.
	By \autoref{definition:kernel-distribution}, these distributions are all given by the limit as $d \ra \infty$ of the the dimension
	of the kernel of a random matrix drawn from certain cosets of the orthogonal group
	of rank $12d-4$.
	The distribution conditioned on even rank corresponds to the cosets with Dickson invariant $0$ while that conditioned on odd rank corresponds
	to cosets with Dickson invariant $1$.
	Therefore, by \autoref{thm: orbit count},
	the moments of these distributions all stabilize in $d$ (in fact once $6d-3 \geq m$),
	and are equal to
	$\prod_{i=1}^m \left( \ell^i + 1 \right)$.
\end{example}

\subsection{Overview of the proof}\label{subsec: overview}

We next indicate the idea of the proof of \autoref{theorem:main}. There is a moduli stack $\estack d {\F_q}$ parameterizing Weierstrass equations for elliptic curves over $\F_q(t)$ of height $d$. For $(n,q)=1$, we define in \autoref{subsection:selmer-space-review} a moduli stack $\sstack n d {\F_q}$ that approximately parameterizes pairs $(E, \alpha)$ for $[E] \in \estack d {\F_q}$ an elliptic curve and $\alpha \in \Sel_n(E)$. The basic point here is that there is a dense open set of points of $\estack d {\mathbb F_q}$ whose corresponding minimal Weierstrass models are smooth over $\mathbb F_q$. For elliptic curves $E$ corresponding to points in this open set,
if $\mathscr E^0$ is the identity component of the N\'{e}ron model of $E$ over $\PP^1_{\mathbb F_q}$,
$\Sel_n(E) = H^1_{\et}(\PP^1_{\mathbb F_q}, \mathscr E^0[n])$. 
(We observe that $\mathscr E^0[n]$ is \'{e}tale over $\PP^1_{\F_q}$ by our
assumption that $(n,q) = 1$: indeed, by miracle flatness it suffices to check
this is \'etale over each point of $\mathbb P^1_{\mathbb F_q}$. Each fiber of $\mathscr E^0$ is a
$1$-dimensional group scheme isomorphic to $\mathbb G_a, \mathbb G_m$, or an
elliptic curve $E$, in which case its $n$-torsion is $\id, \mu_n,$ or $E[n]$,
all of which are \'etale when $(n,q) = 1.$) 
In other words, $\sstack n d {\F_q} $ is the stack classifying $E$ along with \'{e}tale $\mathscr E^0[n]$-torsors over $\PP^1_{\mathbb F_q}$. 

There is an natural quasi-finite map $\pi \co \sstack n d {\F_q} \rightarrow \estack d {\F_q}$, and over an open dense substack $\smestack d {\F_q} \subset \estack d {\F_q}$ the restriction 
\begin{equation}\label{eq: finite etale cover}
	\pi \co \smsstack n d {\F_q} := \sstack n d {\F_q}|_{\smestack d {\F_q}} \rightarrow \smestack d {\F_q}
\end{equation}
is finite \'{e}tale. The $n$-Selmer group of $[E] \in \smestack d {\F_q}(\F_q)$ is then identified with $\F_q$-points of $\pi^{-1}(E)$. The cover $\pi$ is associated to a monodromy representation $\mono n d {\mathbb F_q} \co \pi_1(\smestack d {\F_q}) \rightarrow \o(\qsel n d {\mathbb F_q})$, where $(\vsel n d {\mathbb F_q}, \qsel n d {\mathbb F_q})$ is a particular rank $12d-4$ quadratic space over $\Z/n\Z$, and $\pi^{-1}(E)(\F_q)$ identifies with $\ker (\mono n d {\mathbb F_q} (\Frob_E)-\id) \subset \vsel n d {\mathbb F_q}$. 

After determining the monodromy group, this reduces to a combinatorial problem: compute the distribution of $\dim \ker (g-\id)$ for a $g$ drawn randomly from the monodromy group. For $\vsel n d {\mathbb F_q}$ over $\bz/\ell\bz,$ (i.e., the case that $n = \ell$ is prime,) and $g$ drawn from the full $\o(\qsel \ell d {\mathbb F_q})$, 
this computation was done in unpublished work of Rudvalis and Shinoda, as we learned from \cite{fulmanS:distribution-number-fixed}. We give an alternative proof which generalizes to the case where $g$ is drawn from certain proper subgroups of $\o(\qsel \ell d {\mathbb F_q})$ related to the monodromy group (which is needed for our results).

After handling the case where $n=\ell$ is prime, we move on to the case of $\Sel_{\ell^e}$. In this case, we prove that there is a characterization of $\ker (g-\id)$ in terms of a Markov property, and that the BKLPR heuristic is also characterized by this same Markov property. The case of general $\Sel_n$ for $n$ composite follows from the prime power case by the Chinese remainder theorem.

\begin{figure}
	\centering
\begin{equation}
  \nonumber
  \adjustbox{scale=0.8}{
   \begin{tikzcd}[column sep = 1.7em]
	  \qquad && & \text{Thm.}~\ref{theorem:p-selmer-distribution} \ar{ld} & \text{Thm.}~\ref{theorem: coset generating functions} \ar{l}{\text{\cite{fulmanS:distribution-number-fixed}}}
	  &  \\
	  \qquad && \text{Thm.}~\ref{thm: random intersection markov} \ar{ld} & \text{Lem.}~\ref{lemma:alternating-and-radical-bklpr} \ar{l}& \text{Lem.}~\ref{lem: trans action} \ar{l}\\
	  \qquad \text{Thm.}~\ref{theorem:main} & \text{Thm.}~\ref{theorem:kernel vs bklpr} \ar{l} & \text{Thm.}~\ref{thm: 1-eigenspace markov} \ar{l} & \text{Thm.}~\ref{theorem: eigenspace markov} \ar{l} & \text{Lem.}~\ref{lemma:orthogonal-reduction}\ar{ld} &
	  {\text{\cite[Thm. 4.4]{landesman:geometric-average-selmer}}} \ar{l}{\text{\cite{kneser:erzeugung-ganzzahliger-orthogonaler}}}
	  \\
	  \qquad & & \text{Prop.}~\ref{proposition:component-and-rank} \ar{dl} &
	  \text{Thm.}~\ref{theorem:monodromy} \ar{l}\ar{dl}
	  &  \text{Lem.}~\ref{lemma:spinor-norm-commutes}\ar{l}& \text{Prop.}~\ref{proposition:zywina} \ar{ld} \ar{l} \\
	  \qquad & \text{Thm.}~\ref{theorem:large-q-distribution} \ar{uul} & \text{Cor.}~\ref{corollary:finite-field-distribution} \ar{l}&  \text{Prop.}~\ref{proposition:eigenvalues}\ar[crossing over]{ul}  & \text{Lem.}~\ref{lemma:dickson-monodromy}  \ar{ul} & \text{Prop.}~\ref{proposition:chavdarov} \ar{l} \ar[bend left]{lll}
	  \end{tikzcd}
	  }
\end{equation}
\caption{
A schematic diagram depicting the structure of the proof of \autoref{theorem:main}.}
\label{figure:proof-schematic}
\end{figure}
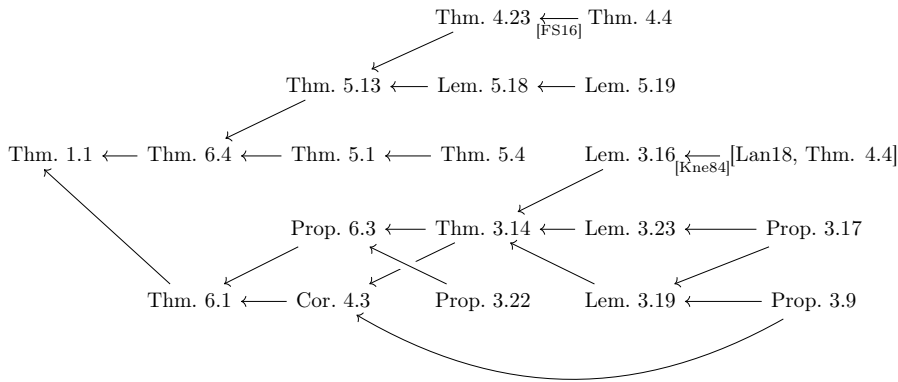

\subsection{Outline of Paper}
We next give a brief outline of the content of the various sections in this paper.
In
\autoref{section:summary}
we recall the construction of Selmer spaces, which parameterize Selmer elements of elliptic curves.
The Selmer spaces mentioned above are generically finite \'etale covers of the moduli space of height $d$ elliptic surfaces.
In
\autoref{section:monodromy}
we compute the monodromy associated to these covers.
Next, in
\autoref{sec: prime distribution}
we establish that the geometric distribution of prime order Selmer groups agree with that predicted by the BKLPR heuristic.
In
\autoref{section:markov},
we show that both the BKLPR heuristic distribution and our geometric distribution agree for prime powers, by relating the two distributions for $\ell^j$-Selmer groups to the two distributions for $\ell^{j+1}$-Selmer groups via separate Markov processes.
Finally, in
\autoref{section:proofs}
we put the pieces together to the prove our main results.

\subsection{Acknowledgements}

It is our pleasure to thank Ravi Vakil for organizing the ``What's on My Mind'' seminar, which led to the genesis of this paper. 
We thank Johan de Jong, Chao Li, Bjorn Poonen, Arul Shankar, Doug Ulmer, and Melanie Matchett Wood for helpful discussions.
We thank Lisa Sauermann for help translating \cite{kneser:erzeugung-ganzzahliger-orthogonaler}.
We also thank David Zureick-Brown and Jackson Morrow for help with writing and running MAGMA code.
The first author was supported by a Stanford ARCS Fellowship and an NSF Postdoctoral Fellowship under Grant No. 1902927, and the second author was supported by the National Science Foundation Graduate Research Fellowship Program under Grant No. DGE-1656518.

\section{Summary of Selmer spaces}\label{section:summary}

\subsection{Reviewing the definition of the Selmer space}
\label{subsection:selmer-space-review}

Here, we briefly recall the construction of the Selmer space and related spaces introduced in \cite[\S3]{landesman:geometric-average-selmer}.
The new content in this section occurs in \autoref{subsection:rank-sheaf}
where we introduce an sheaf is isomorphic to the Selmer sheaf (\autoref{subsubsection:selmer-space} for the definition) on a dense open.
This sheaf is closely related to the L-function of elliptic curves, and hence gives us a way to access the analytic ranks of elliptic curves
in terms of the Selmer sheaf.
Our notation differs slightly from that of \cite{landesman:geometric-average-selmer} 
due to a minor error 
(only appearing in characteristic $3$), as we will explain further in \autoref{remark:espace-notation}.

\subsubsection{The space of Weierstrass equations} 
\label{subsubsection:minimal-weierstrass}
Throughout this section, we work relatively over a scheme $B$ on which $2$ is invertible.
As in \cite[Definition 3.1]{landesman:geometric-average-selmer},
define $\bp^1_B := \proj_B \mathscr O_B[s,t]$.
	Form the affine space,
	\begin{align*}
		\aff d B := \spec_B \mathscr O_B[a_{2,0}, a_{2,1} \ldots, a_{2,2d}, a_{4,0}, \ldots, a_{4,4d}, a_{6,0} \ldots, a_{6,6d}].
	\end{align*}
	For $i \in \{1,2,3\}$, define $a_{2i}(s,t) := \sum_{j=0}^{2id} a_{2i,j} t^j s^{2id-j}$.
	Let $\espace d B \subset \aff d B$ denote the open subscheme parameterizing those points such that the Weierstrass equation 
	\[
	y^2z = x^3 + a_2(s,t)x^2z + a_4(s,t)xz^2 + a_6(s,t)z^3
	\]
	defines an elliptic surface with smooth generic fiber.
	This is open as it corresponds to the open subscheme of $\aff d B$ such that the discriminant is nonzero.
\begin{remark}
	\label{remark:espace-notation}
	There was a minor error in 
	\cite[Definition 3.1]{landesman:geometric-average-selmer}
	where it was claimed that a Weierstrass model is minimal if and only if it is of the form
	$y^2z = x^3 + a_2(s,t)x^2z + a_4(s,t)xz^2 + a_6(s,t)z^3$
	with no non-constant polynomial $f \in k[s,t]$ with
$f^{2i} \mid a_{2i}(s,t)$ for all $i \in \{1,2,3\}$.
However, it is only true that it can be written in this form after a change of variables. 

This makes it less obvious that in characteristic $3$,
the locus of minimal Weierstrass equations is open $\aff d B$.
It is fairly simple to see this is true in characteristic neither $2$ nor $3$, since one can make a change of variables to assume $a_2(s,t) = 0$,
and then the resulting equation $y^2z = x^3 + a_4(s,t)xz^2 + a_6(s,t)z^3$
is minimal if and only if 
there is no non-constant polynomial $f \in k[s,t]$ with
$f^{2i} \mid a_{2i}(s,t)$ for all $i \in \{2,3\}$.
In characteristic $3$, this non-minimal locus is still open, but we only found a somewhat involved proof which involves tracing through the steps of Tate's algorithm.

To avoid this fairly involved proof, we opt to work over a slightly larger open set $\espace d B$, which does not parameterize minimal Weierstrass models,
but instead parameterizes all Weierstrass models over $\aff d B$ with smooth generic fiber.
Since the two open subsets differ by a divisor, their point counts do not contribute in the large $q$ limit,
and so which set we work with does not substantially alter the argument.
\end{remark}

\subsubsection{The universal Weierstrass equation}
Similarly to \cite[Definition 3.1]{landesman:geometric-average-selmer},
	one can construct a family of minimal Weierstrass  models $\uespace d B$ over $\bp^1 \times \espace d B$ as the 
	subscheme of
\[
\proj_{\bp^1_B \times_B \espace d B} \sym^\bullet \left( \mathscr O_{\bp^1_B \times_B \espace d B} \oplus \mathscr O_{\bp^1_B \times_B \espace d B}(-2d) \oplus \mathscr O_{\bp^1_B \times_B \espace d B}(-3d) \right)
\]
cut out by the equation
\[
y^2z = x^3 + a_2(s,t)x^2z + a_4(s,t)xz^2 + a_6(s,t)z^3.
\]
As mentioned in \autoref{remark:espace-notation}, we work over 
$\espace d B$, a set including non-minimal elliptic curves, which is slightly different than that used in
\cite[Definition 3.1]{landesman:geometric-average-selmer}.

\subsubsection{An open subset} 
Recall our definition of $\espace d B$ from
\autoref{subsubsection:minimal-weierstrass} as a moduli space of height $d$ minimal Weierstrass equations.
Similarly to \cite[Definition 3.9]{landesman:geometric-average-selmer}, let $\smespace d B \subset \espace d B$ denote the open subscheme over which $\uespace d B \ra \espace d B$ is smooth.
In the case $B$ is a field $k$, $\smespace d k$ parameterizes elliptic curves of height $d$ over $k(t)$ so that the associated minimal Weierstrass
elliptic surface
is smooth over $k$.
Let $\usmespace d B := \uespace d B \times_{\espace d B} \smespace d B$
denote the universal elliptic surface over $\smespace d B$.
We also introduce
$\sfespace d B \subset \espace d B$ as the open subscheme parameterizing elliptic surfaces with squarefree discriminant
and let
$\usfespace d B := \uespace d B \times_{\espace d B} \sfespace d B$;
these subsets are indeed open and dense over $B$ as is explained in \cite[Lemma 3.14]{landesman:geometric-average-selmer}.
Loosely speaking, the idea is to show that the elliptic surfaces of height $d$ with squarefree discriminant are the complement of two divisors:
the divisor parameterizing elliptic surfaces of height $d$ which are singular and the divisor paramterizing elliptic surfaces of height $d$ with some cuspidal fiber.
These two divisorial subschemes can be defined via incidence correspondences. One can then use these incidence correspondences to compute the dimensions of these subschemes,
and verify they are indeed divisors, implying that the open locus of elliptic surfaces of height $d$ is fiberwise nonempty, hence fiberwise dense.

\subsubsection{The Selmer space}  
\label{subsubsection:selmer-space}
Similarly to \cite[Definition 3.3]{landesman:geometric-average-selmer},
(but see \autoref{remark:espace-notation} for a slight difference)
denote by $f$ and $g$ the projection maps
\[
\uespace d B \xra{f} \bp^1_B \times_B \espace d B \xra{g} \espace d B.
\]
Assuming further that $2n$ is invertible on $B$.
Define the {\em $n$-Selmer sheaf over $B$ of height $d$}
as
$\ssheaf n d B := R^1 g_* (R^1 f_* \mu_n)$.
Define the 
{\em $n$-Selmer space over $B$ of height $d$,}
denoted $\sspace n d B$
as the algebraic space representing the sheaf of $\bz/n\bz$ modules
$\ssheaf n d B$.
Let 
\[
	\smsspace n d B := \sspace n d B \times_{\espace d B} \smespace d B, 
	\hspace{1cm}
	\sfsspace n d B := \sspace n d B \times_{\espace d B} \sfespace d B, 
	\hspace{1cm}
	\smssheaf n d B := \ssheaf n d B \times_{\espace d B} \smespace d B.
\]

\subsubsection{A moduli stack of elliptic curves} 
\label{subsubsection:elliptic-stack}
Note that $\bg_a^{2d+1} \rtimes \bg_m$ acts on $\uespace d B$ and $\espace d B$ compatibly.
Loosely speaking, $(r_0, \ldots, r_{2d}) \in \bg_a^{2d+1}$ acts by sending $x \mapsto x + r_0 s^{2d} + r_1 t s^{2d-1} + \cdots + r_{2d} t^{2d}$ and
$\lambda \in \bg_m$ acts by sending $a_{2i}(s,t) \mapsto \lambda^{2i} a_{2i}(s,t)$,
see 
\cite[Definition 3.4]{landesman:geometric-average-selmer}
for a more precise formulation in terms of Weierstrass equations.
By
\cite[III.3.1(b)]{Silverman2009},
any two points in $\espace d B$ corresponding to isomorphic elliptic curves lie in the same orbit of this action.
Similarly to 
\cite[Definition 3.4]{landesman:geometric-average-selmer},
we define the {\em moduli stack of height $d$ minimal Weierstrass models over $B$} as the quotient stack 
\[
\estack d B := \left[ \espace d B/ \bg_a^{2d+1} \rtimes \bg_m \right].
\]

\subsubsection{The Selmer stack}
Similarly to
\cite[Definition 3.4]{landesman:geometric-average-selmer},
we define the {\em $n$-Selmer stack over $B$ of height $d$} as the quotient stack 
\[
\sstack n d B := \left[ \sspace n d B/ \bg_a^{2d+1} \rtimes \bg_m \right].
\]
Since the action of $\bg_a^{2d+1} \rtimes \bg_m$ restricts to an action on $\usmespace d B$, $\smespace d B$, and $\smsspace n d B$,
we similarly define
\[
	\smestack d B := \left[ \smespace d B/ \bg_a^{2d+1} \rtimes \bg_m \right],
\hspace{1cm}
\sfestack d B := \left[ \sfespace d B/ \bg_a^{2d+1} \rtimes \bg_m \right],
\]
and
\[
	\smsstack n d B := \left[ \smsspace n d B/ \bg_a^{2d+1} \rtimes \bg_m \right],
	\hspace{1cm}
\sfsstack n d B := \left[ \sfsspace n d B/ \bg_a^{2d+1} \rtimes \bg_m \right].
\]
\begin{remark}
	\label{remark:corresponding-elliptic-curve}
	For $x \in \espace d B$ or $x \in \estack d B$, we use $E_x$ denote the corresponding elliptic curve.
	Specifically, for $x \in \espace d B$, if $f: \uespace d B \ra \bp^1 \times \espace d B$, then $E_x = f^{-1}(\eta \times x)$, for $\eta$ the generic point of $\bp^1$.
	We often notate this by $[E_x] = x \in \espace d B$.
	Similarly, for $x \in \smestack d B$, we notate $[E_x] = x$ where $E_x$ is the elliptic curve corresponding to $x$.
\end{remark}

\subsection{The relation between Selmer spaces and Selmer groups}\label{ssec: relation between selmer space and group}

We have now defined the Selmer space, but have not yet explained the connection to Selmer groups of elliptic curves.
The following lemma
provides the relation.
\begin{lemma}[\protect{\cite[Corollary 3.24]{landesman:geometric-average-selmer}}]
	\label{lemma:selmer-fiber-estimate}
	Let $n \geq 1, d > 0, m \geq 0.$ Let $B$ be a noetherian scheme with $2n$ invertible,
	and let $\pi: \sspace n d B  \ra \espace d B$ denote the projection map.
		For $[E_x] = x\in \smespace d B(\mathbb F_q)$, we have
	\begin{align}
		\label{equation:selmer-equality-on-open}
		\# \Sel_n(E_x) = \# \left(\pi^{-1}(x)\left( \mathbb F_q \right) \right).
	\end{align}
\end{lemma}

\subsection{The sheaf governing rank}
\label{subsection:rank-sheaf}

In this section, we introduce a sheaf 
$\psheaf n d B$.
This is closely related to the Selmer sheaf $\smssheaf n d B$
and governs the rank of the elliptic curve.
This sheaf is not new, and has previously appeared in the literature, see
\autoref{remark:hall-sheaf}.
Our goal will be to show the two sheaves are isomorphic on the fiberwise over $B$ dense open of $\smespace d B$ parameterizing elliptic surfaces with squarefree discriminant.
We now define $\psheaf n d B$.

\begin{notation}
	\label{notation:sspace-notation}
	Let $B$ be a scheme with $2n$ invertible on $B$.
	Let $j: U \subset \bp^1_B \times_B \smespace d B$
	denote the open subscheme over which the projection
	$f : \usmespace d B \ra \bp^1_B \times_B \smespace d B$
	is smooth.
	Let $g: \bp^1_B \times_B \smespace d B \ra \smespace d B$ denote
	the projection.
	Then, if $\alpha_S : S \ra \smespace d B$ is a map of schemes,
	set up the following commutative diagram, where both squares are fiber squares.
	\begin{equation*}
		\begin{tikzcd} 
			\usmespace d B \times_{\smespace d B} S \ar{rd}{f^S}	& U_S \ar {r}{\alpha_S'} \ar {d}{j^S} \ar[bend right=70, swap, crossing over]{dd}{\ol{g}^S}& U \ar {d}{j} \ar[bend left=100]{dd}{\ol{g}} &  & \usmespace d B \ar[crossing over]{lld}{f}\\
	\qquad&		\bp^1_B \times_B S \ar{d}{g^S} \ar {r} & \bp^1_B \times_B \smespace d B \ar{d}{g} \\
	\qquad	&		S \ar {r}{\alpha_S} & \smespace d B \\
	\end{tikzcd}\end{equation*}
	Define
$\esheaf n S := (j^S)^* R^1 f^S_* \mu_n$ (we note that $\esheaf n S$ is a slight abuse of notation since it depends on the map $\alpha_S$ and not just the scheme $S$). This sheaf represents the relative
	$n$ torsion of $f^S$.
	Define the sheaf $\psheaf n d B := R^1 g_*(j_* \esheaf n {\smespace d B})$, with the implicit map $\alpha_{\smespace d B}: \smespace d B \ra \smespace d B$ taken to be the identity.
\end{notation}
\begin{remark}
	\label{remark:hall-sheaf}
	Sheaves defined analogously to $\psheaf n d B$ appeared in the context of quadratic twist families of elliptic curves in
	\cite[\S6.2]{hall:bigMonodromySympletic} and \cite[\S3.2]{zywina:inverse-orthogonal}.
	In fact, $\psheaf n d B$ is itself a reasonable candidate for the Selmer sheaf, but we will instead work with $\smssheaf n d B$, which has the advantage that it
	commutes with base change.
	On the other hand, we are not sure if $\psheaf n d B$ commutes with base change in general, though it does
	over $\sfespace d B$, as we show in \autoref{lemma:psheaf-basechange}.
\end{remark}

Having defined $\psheaf n d B$, we next wish to show it agrees with $\smssheaf n d B$, at least when both are restricted to $\sfespace d B$.
To verify this isomorphism, we will construct a map between them and check it is an isomorphism by checking it on fibers.
The verification on fibers is fairly immediate once we know that the formation of $\psheaf n d B$ commutes with base change, as we now verify.
A variant of the following \autoref{lemma:psheaf-basechange} is explained in \cite[Construction-Proposition 5.2.1(3)]{katz:twisted-l-functions-and-monodromy}.

\begin{lemma}
	\label{lemma:psheaf-basechange}
	With maps $f$ and $g$ as in \autoref{notation:sspace-notation},
	the sheaf $\psheaf n d B$ is a constructible sheaf of $\bz/n\bz$ modules whose formation commutes with base change on $\sfespace d B$.
	More precisely, for any base scheme $S$ factoring through $\sfespace d B$, the base change map 
	\[
	\alpha_S^* R^1 g_*(j_* \esheaf n {\smespace d B}) \ra R^1 g^S_* (j^S_*{\alpha'_S}^* \esheaf n {\smespace d B}),\]
	is an isomorphism.
\end{lemma}
\begin{proof}
	Let $R^1 \ol{g}_! \esheaf n {\smespace d B} \xra{\phi} \psheaf n d B$
	denote the map 
	induced by $j_! \esheaf n {\smespace d B} \ra j_* \esheaf n {\smespace d B}$,
	using the identification $R^1 \ol{g}_! \esheaf n {\smespace d B} = R^1 g_*(j_! \esheaf n {\smespace d B})$.
	Let $\psheaf n d B \xra{\psi} R^1\ol{g}_* \esheaf n {\smespace d B}$
	denote the map induced from the composition of functors spectral sequence for $g \circ j$.
	We will show
	that $\psheaf n d B$ is the image of the composition 
	$R^1 \ol{g}_! \esheaf n {\smespace d B} \xra{\phi} \psheaf n d B \xra{\psi} R^1\ol{g}_* \esheaf n {\smespace d B}$.
	Once we show this, it will immediately follow that $\psheaf n d B$
	is constructible, being the image of a map of constructible sheaves.

	By the Leray spectral sequence, $\psi$ is always injective.
	Hence, to identify $\psheaf n d B$ as the image of $\psi \circ \phi$,
	we only need to show $\phi$ is surjective. 
	To this end, 
	define $M$ as the quotient sheaf $j_* \esheaf n {\smespace d B}/j_! \esheaf n {\smespace d B}$.
	Note that $M$ is supported on the complement of $U$
	which is finite over $\smespace d B$.
	Therefore, $R^1 g_* M = 0$ and we conclude that
	$R^1 \ol{g}_! \esheaf n {\smespace d B}= R^1 g_* (j_! \esheaf n {\smespace d B}) \ra R^1 g_* (j_* \esheaf n {\smespace d B}) = \psheaf n d B$
	is surjective.
	Hence, $R^1 g_*\left( j_* \esheaf n {\smespace d B} \right)$ is a constructible
	$\bz/n\bz$ module, being the image of a map of constructible
	$\bz/n\bz$ modules.

	To conclude, we show that the formation of $\psheaf n d B$ commutes with base change over $\sfespace d B$.
	Since $\psheaf n d B$ is the image of $\psi \circ \phi: R^1 \ol{g}_! \esheaf n {\smespace d B} \ra R^1\ol{g}_* \esheaf n {\smespace d B}$, it suffices to show that the formation of 
	both $R^1 \ol{g}_! \esheaf n {\smespace d B}$
	and $R^1\ol{g}_* \esheaf n {\smespace d B}$ commute with base change over $\sfespace d B$.
	The former commutes with base change by 
	proper base change with compact supports.

	To conclude, it remains to show the formation of $R^1\ol{g}_* \esheaf n {\smespace d B}$
	commutes with base change over $\sfespace d B$.
	We will do this using Poincar\'e duality and Deligne's semicontinuity theorem for Swan conductors \cite[Corollaire 2.1.2 and Remarque 2.1.3]{laumon:semi-coninuity-du-conducteur-de-swan}.
	We first use Deligne's semicontinuity theorem to show $R^i\ol{g}_! \esheaf n {\smespace d B}$ is locally constant constructible for all $i \geq 0$.
	The semicontinuity theorem says that $R^i\ol{g}_! \esheaf n {\smespace d B}$ will be locally constant over any open subscheme
	of $\smespace d B$ for which the degree of $\bp^1 \times \smespace d B - U \to \smespace d B$ is constant
	and the total Swan conductor associated to $\esheaf n {\smespace d B}$ is constant. 

	We now verify the hypotheses of Deligne's semicontinuity theorem by verifying 
	$\bp^1 \times \smespace d B - U \to \smespace d B$ has constant fiber degree over $\sfespace d B$ and that the Swan conductor
	vanishes over $\sfespace d B$.
	Indeed, any elliptic curve corresponding to a point of $\sfespace d B$ has reduced discriminant, and hence $12d$ geometric fibers of type $I_1$
	reduction and no other singular fibers, by Tate's algorithm.
	This shows 
	$\bp^1 \times \smespace d B - U \to \smespace d B$ has constant fiber degree over $\sfespace d B$.
	Finally, the Swan conductor always vanishes when the reduction is multiplicative \cite[IV.10.2(b)]{Silverman:Advanced}.

	Using that $R^1\ol{g}_! \esheaf n {\smespace d B}$ is locally constant constructible 
	over $\sfespace d B$
	we next deduce 
	$R^1\ol{g}_* \esheaf n {\smespace d B}$
	is as well
	via Poincare duality.
	Namely, Poincar\'e duality \cite{Verdier} gives an isomorphism of sheaves in the derived category
	\begin{align*}
		R\overline{g}_* R\shom(\esheaf n {\smespace d B}, \mu_n[2]) \simeq R\shom(R\overline{g}_! \esheaf n {\smespace d B}, \mu_n).
	\end{align*}
	Note that the $[2]$ denotes a cohomological shift by $2$ while the $[n]$ refers to the $n$-torsion.
	
	We will now take $(-1)$st cohomology of both sides.
	By construction of $U$, $\esheaf n {\smespace d B}$ is locally constant on $U$, and therefore
	the $i$th cohomology of 
	$R \overline{g}_* R\mathscr{H}om(\esheaf n {\smespace d B}, \mu_n[2])$ is given by $R^{i+2} \overline{g}_* \shom(\esheaf n {\smespace d B}, \mu_n) \simeq R^{i+2} \overline{g}_* \esheaf n {\smespace d B},$ the latter isomorphism
	induced
	by the Weil pairing.
	Additionally, since $R^{-i}\overline{g}_! \esheaf n {\smespace d B}$ is locally constant constructible, we get that the $i$th cohomology of
	$R\shom(R\overline{g}_! \esheaf n {\smespace d B}, \mu_n[2])$ is given by $\shom(R^{-i} \overline{g}_! \esheaf n {\smespace d B}, \mu_n)$.
	Therefore, taking $(-1)$st cohomology of the Poincar\'e duality isomorphism yields an isomorphism
	$R^1 \overline{g}_* \esheaf n {\smespace d B} \simeq (R^1 \overline{g}_! \esheaf n {\smespace d B})^\vee$.
	Since the right hand side is locally constant constructible over $\sfespace d B$, the left hand side is as well, and therefore commutes with base change.
\end{proof}

We next produce an isomorphism $\smssheaf n d B|_{\sfespace d B} \simeq \psheaf n d B|_{\sfespace d B}$ over $\sfespace d B$, crucially using that the formation of both sheaves commute with base change.
\begin{proposition}
	\label{lemma:mu-n-pushforward-to-selmer-space}
	Retain notation from \autoref{notation:sspace-notation}.
	There is canonical map $R^1 f_* \mu_n \ra j_* \esheaf n {\smespace d B}$ of sheaves on $\bp^1_B \times_B \smespace d B$.
	This map induces an isomorphism $R^1 g_*(R^1 f_* \mu_n)|_{\sfespace d B} \simeq R^1 g_*(j_* \esheaf n {\sfespace d B})$, which commutes with base change.
\end{proposition}
\begin{proof}
Retaining notation from \autoref{notation:sspace-notation},
	define the maps $j'$ and $f'$ as in the fiber square
	\begin{equation}
		\label{equation:}
		\begin{tikzcd} 
			W_U \ar {r}{j'} \ar {d}{f'} & \usmespace d B  \ar {d}{f} \\
			U \ar {r}{j} & \bp^1_B \times_B \smespace d B.
		\end{tikzcd}\end{equation}
	We have canonical maps coming from Leray spectral sequences 
	\begin{align}
		\label{equation:pushforward-switch}
		\begin{split}
	R^1 f_* (\mu_{n}) &\simeq R^1 f_* (j'_* \mu_{n}) \\
	& \ra R^1 (f \circ j')_* \mu_n \\
		& = R^1(j \circ f')_* \mu_n \\
		& \ra j_* R^1 f'_* \mu_n.
	\end{split}
	\end{align}
	Using the Kummer exact sequence (possible since $n$ is invertible by \autoref{notation:sspace-notation}) and the assumption that the fibers of $f'$ are smooth connected elliptic curves
	so \cite[\S9.5, Theorem 1]{BoschLR:Neron} applies, we obtain isomorphisms
	\begin{align}
		\label{equation:j-pic-to-sheaf}
j_* R^1 f'_* \mu_n
\simeq j_* \pic_{W_U/U} [n] \simeq j_* \pic^0_{W_U/U}[n] \simeq j_* \esheaf n {\smespace d B}.
	\end{align}
	Composing \eqref{equation:pushforward-switch} with \eqref{equation:j-pic-to-sheaf}, we obtain the desired map $R^1 f_* (\mu_{n,W}) \ra j_* \esheaf n {\smespace d B}$.

	We show this map induces an isomorphism
$R^1 g_*(R^1 f_* \mu_{n}))|_{\sfespace d B} \ra R^1 g_*(j_* \esheaf n {\sfespace d B})$.
	To verify this is an isomorphism, it suffices to do so on stalks.
	As the formation of both sides commutes with base change by proper base change and \autoref{lemma:psheaf-basechange},
	we can check this is an isomorphism in the case that the base is a geometric point.

	Thus, it suffices to show that if $f^x: W_x \ra \bp^1_x$ is a smooth minimal Weierstrass model corresponding to a point $x \in \smespace d B$,
	$j^x$ is the restriction of $j$ to $x$, and $g^x$ is the restriction of $g$ to $x$,
	then the map on stalks $\phi_x: R^1 g^x_* (R^1 f^x_* \mu_n) \ra R^1 g^x_* (j^x_* (\esheaf n x))$ is an isomorphism.
	It suffices to check the map $R^1 f^x_* \mu_n \ra j^x_* (\esheaf n x)$ inducing $\phi_x$ under $R^1g^x_*$ is an isomorphism.
	To this end, 
	by
	\cite[Lemma 3.7]{landesman:geometric-average-selmer},
the \'{e}tale sheaf $R^1 f^x_* \mu_n$ is represented by the N\'eron model of $E_x[n]$ on the small \'etale site of $\bp^1_x$,
while $j^x_* (\esheaf n x)$ is also represented by the N\'eron model of $E_x[n]$  by the N\'eron mapping property.
The N\'eron mapping property implies that to check the map 
$R^1 f^x_* \mu_n \ra j^x_* (\esheaf n x)$ 
constructed 
in
\eqref{equation:pushforward-switch}
is an isomorphism, it suffices to check
its restriction to $U$ is an isomorphism.
That is, we want to show
the base change of
$j^* R^1 f_* (\mu_{n}) \ra j^* j_* \esheaf n {\smespace d B} \simeq R^1 f'_* j'^* \mu_n$ 
to $x$
is an isomorphism.
If we could show this is the natural base change map, it would indeed be an isomorphism by proper base change.

So, to conclude the proof, we only need to check the constructed map
$j^* R^1 f_* (\mu_{n}) \ra R^1 f'_* j'^* \mu_n$, coming from
pulling back \eqref{equation:pushforward-switch} along $j$,
is the base change map.
Indeed, this follows from the definitions.
In more detail, recall that for $\scf$ a sheaf on $\usmespace d B$,
the base change map is given as the map of $\delta$-functors
$j^* \circ (R^\bullet f_*)  \scf \ra (R^\bullet f'_*) \circ j'^* \scf$
induced via the degree $0$ composition $j^* f_* \scf \ra j^* f_* j'_* j'^* \scf \ra j^* j_* f'_* j'^* \scf \ra f'_* j'^* \scf$,
see
\cite[\S6, p. 60-61]{FreitagK:lectures-etale}.
However, pulling back the map of 
\eqref{equation:pushforward-switch}
along $j$ is given by the composition
$j^* R^1 f_* \mu_n \ra j^* R^1f_* (j'_* j'^* \mu_n) \ra j^* R^1(j\circ f')_* (j'^* \mu_n) \ra R^1 f'_* (j'^* \mu_n)$.
This is precisely the resulting map on degree $1$ $\delta$-functors, and hence is the natural base change map.
\end{proof}

\section{The precise monodromy of Selmer spaces}
\label{section:monodromy}

The main result of this section is \autoref{theorem:monodromy} where we compute precisely the monodromy group associated to the cover $\smsstack n d B \ra \smestack n B$.
In order to state the theorem, we first introduce some various notation relating to orthogonal groups and the monodromy representation.
Following this, we recall a general result on equidistribution of Frobenius elements in \autoref{subsection:equidistribution}.
The remainder of the section is devoted to proving
\autoref{theorem:monodromy}, whose proof is outlined at the end of \autoref{subsection:determining-image-of-monodromy}.

\subsection{Adelic notation}

For $R$ an integral noetherian ring with fraction field $\Frac(R)$ such that $\chr(\Frac(R)) = p$, let 
\[
\zprime p := \lim_{\gcd(n,p)=1} \bz/n\bz \simeq \prod_{\substack{\ell \text{ prime} \\ r \neq p}} \bz_\ell.
\]
We allow $p = 0$, in which case $\zprime 0 = \widehat{\bz}$.

\subsection{Notation for orthogonal groups}
\label{subsection:orthogonal-notation}

\subsubsection{Notation for quadratic forms}\label{subsubsection:quadratic-forms-base-change}
Let $R$ be a ring.
A {\em quadratic space} over $R$ is a pair $(V, Q)$ 
where $V$ is a free module over $R$ and $Q: V \ra R$ is a quadratic form. We say a quadratic space $(V,Q)$ is {\em nondegenerate} if the hypersurface defined by the vanishing of $Q$
in $\bp V^\vee$ is smooth over $\spec R$. When $2$ is invertible or $\rk V$ is even, this is equivalent to the discriminant of $Q$ being a unit on $\spec R$, see \cite[Remark C.1.1]{conrad:reductive-group-schemes}. See \cite[C.1]{conrad:reductive-group-schemes} for a characterization in terms of non-degeneracy of the associated bilinear form on fibers.
Let $\o(Q)$ the corresponding orthogonal group. 
Note that we will use $\o(Q)$ to denote both the group and the group scheme. We will primarily consider it as a group, and whenever we use it to denote the group scheme $\o(Q)$,
we refer to it as ``the algebraic group $\o(Q)$.''

For $\phi: R \ra S$ a map of rings, we denote $(V_\phi, Q_\phi):= (V \otimes_R S, Q \otimes_R S)$.
When the map $\phi$ is understood, we notate this as $(V_S, Q_S) := (V_\phi, Q_\phi)$.
In the special case that $S = \bz/n\bz$, we will also use $(V_n, S_n) := (V_{\bz/n\bz}, Q_{\bz/n\bz})$.

\begin{definition}
	\label{definition:selmer-space-quadratic-form}
	For $d \geq 1$, define the quadratic space $(\vsel {\bz} d k, \qsel {\bz} d k)$
to be the rank $12d-4$ free $\bz$ module 
associated to
$U^{\oplus (2d-2)} \oplus (-E_8)^{\oplus d}$, for $U$ a hyperbolic plane and $-E_8$ the
$E_8$ lattice with the negative of its usual pairing.
Then $(\vsel n d k, \qsel n d k)$ denotes the reduction of this quadratic space modulo $n$.
\end{definition}

For $Q$ a quadratic form on a free module $V$ over a ring $R$, the \emph{associated bilinear form} $B_Q: V \times V \rightarrow R$ is defined by 
\[
B_Q(x, y) := Q(x+y) - Q(x) - Q(y).
\]
In what follows, we assume the quadratic form $Q$ is nondegenerate.

For $v \in V$, with $Q(v) \in R^\times$ invertible, 
denote the {\em reflection about $v$} (sometimes also called an \emph{orthogonal transvection}, cf. \cite[3.8.1]{wilson:the-finite-simple-groups})
\begin{align*}
	r_v: V & \rightarrow V \\
	w & \mapsto w - \frac{B_Q(w, v)}{Q(v)} v.
\end{align*}
\begin{remark}
	\label{remark:reflections-generate}
	When $R$ is a field, $\o(Q)$ is generated by these reflections so long as $(R, \rk V) \neq (\mathbb F_2, 4)$
\cite[I.5.1]{chevalley:the-algebraic-theory-of-spinors}.
\end{remark}

\subsubsection{The spinor norm}
\label{subsubsection:spinor}

For completeness, we briefly recall the formal definition of the $-1$-spinor norm.
We follow \cite[p. 349]{conrad:reductive-group-schemes} which gives the definition in the more general 
context of algebraic groups. 
Let $(V, Q)$ be a quadratic space over $R$, and suppose that either $\rk V$ is even or $2$ is invertible on $R$.
The $+1$-spinor norm is then defined as the boundary map on cohomology 
\[
	\mathrm{sp}^+_Q: \o(Q) \ra H^1(\spec R, \mu_2) \simeq R^\times/\left( R^\times \right)^2
\]
induced by the sequence of algebraic groups $\mu_2 \ra \mrm{Pin(Q)} \ra \o(Q)$. 
Then the {\em $-1$-spinor norm} on $\o(Q)$ is the $+1$-spinor norm for $\o(-Q)$ composed with the identification $\o(Q) \xrightarrow{\sim} \o(-Q)$ \cite[Remark C.4.9, Remark C.5.4, and p. 348]{conrad:reductive-group-schemes}.\footnote{Although it will not be relevant to this paper, as we shall ultimately only be interested in the even rank quadratic space of \autoref{definition:selmer-space-quadratic-form},
one can define the spinor norm on $\o(Q)$ in the case that $R$ is a field of characteristic $2$ and $\rk V$ is odd.
This can be done using the equality $\o(Q) = \so(Q)$ as abstract groups (even though the corresponding {\em group schemes} are not isomorphic) since the group scheme $\so(Q)$ is the underlying reduced subscheme of the group scheme $\o(Q)$, see \cite[Remark C.5.12]{conrad:reductive-group-schemes}.}

In the case $Q(v) \in R^\times$, the reflection $r_v$ satisfies $\sp{Q}(r_v) = [-Q(v)]$, the coset represented by $-Q(v)$ in $R^\times/\left( R^\times \right)^2$.
Note that the spinor norm is trivial in the case $R = \mathbb F_2$.
When $R = k$ is a field with $k \neq \mathbb F_2$, then $\o(Q)$ is generated by reflections (cf. \autoref{remark:reflections-generate}),
and $\sp{Q}$ is then characterized by $\sp{Q}(r_v) = [-Q(v)]$. 

\begin{definition}
	\label{definition:}
	For $(V,Q)$ a nondegenerate quadratic space over a ring $R$, 
define $\osp(Q)  := \ker \sp{Q} \subset \o(Q)$ to be the kernel of the $-1$-spinor norm. 
\end{definition}

\subsubsection{The adelic spinor map}
\label{subsubsection:adelic-spinor}

We now spell out some notation to describe the spinor map for a quadratic form over $\zprime p$.
Let $p$ either be a prime or $p =0$.
Let $(V,Q)$ be a nondegenerate quadratic space over $\zprime p$.
Let 
\begin{align*}
	\sp{Q}: \o(Q) \ra \left(\zprime p \right)^\times/\left( \left(\zprime p\right)^\times \right)^2 \simeq (\bz/2\bz)^2 \times \prod_{\text{odd primes } \ell \neq p} \bz/2\bz,
\end{align*}
where the first copy of $(\bz/2\bz)^2$ comes from $(\bz/2\bz)^2 \cong \bz_2^\times/\left( \bz_2^\times \right)^2 \simeq \left( \bz/8\bz \right)^\times / \left(\left( \bz/8\bz \right)^\times \right)^2$ and the copy of $\bz/2\bz$ indexed by an odd prime $\ell$ comes from $\bz_\ell^\times/\left( \bz_\ell^\times \right)^2 \simeq \left( \bz/\ell\bz \right)^\times / \left(\left( \bz/\ell\bz \right)^\times \right)^2$.
When $p \neq 0$ and $q$ is a power of $p$,
we let 
\begin{align*}
[q] \in \left(\zprime p \right)^\times/\left( \left(\zprime p\right)^\times \right)^2 \simeq (\bz/2\bz)^2 \times \prod_{\text{odd primes } \ell \neq p} \bz/2\bz
\end{align*}
denote the element induced by multiplication by $q$ on $\zprime p$.

\subsubsection{The Dickson invariant}
\label{subsubsection:dickson}
Next, for $(Q, V)$ a quadratic space over a ring $R$ with $\spec R$ connected, the {\em Dickson invariant} is a map
\[
	\dickson{Q}: \o(Q) \ra \bz/2\bz,
\]
as defined in \cite[(C.2.2) and Remark C.2.5]{conrad:reductive-group-schemes}.
In the case $(Q, V)$ is a quadratic space over a ring $R$ such that $\spec R$ is a disjoint union of finitely many connected components, such as when $R = \bz/n\bz$, we define the Dickson invariant
as the resulting map
\[
\dickson{Q}: \o(Q) \ra \left( \bz/2\bz \right)^{\#\pi_0(\spec R)},
\]
obtained by restricting to a given connected component of $\spec R$ and then applying the Dickson invariant on that component.

In the case $R = \zprime p$, we define the Dickson invariant as the resulting composition
\begin{align*}
	\dickson{Q}: \o(Q) \ra \prod_{\text{primes }\ell \neq p} \o(Q|_{\bz_\ell}) \xra{\prod_{\text{primes }\ell \neq p} \dickson{Q|_{\bz_\ell}}} \prod_{\text{primes }\ell \neq p} \bz/2\bz.
\end{align*}
In all cases above, for $\dickson{Q}: \o(Q) \ra \prod_{s \in S} \bz/2\bz$ for an appropriate set $S$, we let $\Delta_{\bz/2\bz}: \bz/2\bz \ra \prod_{s \in S} \bz/2\bz$ denote
the diagonal inclusion sending $1 \mapsto \left( 1,1,\ldots, 1 \right)$.

\begin{warn}
	\label{warning:}
	Our definition of the Dickson invariant for a quadratic space over $\zprime p$ may differ from the more general scheme theoretic
	definition
	given in \cite[(C.2.2) and Remark C.2.5]{conrad:reductive-group-schemes}.
	There, it is defined as a map to $\left( \bz/2\bz \right)(\spec R)$,
	the global sections of the locally constant sheaf $\bz/2\bz$ on $\spec R$.
	However, there is a natural map $\left( \bz/2\bz \right)(\spec \zprime p) \ra \prod_{\text{primes }\ell \neq p} \bz/2\bz$,
	and our definition of the Dickson invariant is the composition of the 
	Dickson invariant as in \cite[(C.2.2) and Remark C.2.5]{conrad:reductive-group-schemes} with this natural map.
\end{warn}

\begin{remark}In the case that $2$ is invertible on $R$ with $\spec R$ connected, the Dickson invariant agrees with the determinant \cite[Corollary C.3.2]{conrad:reductive-group-schemes}.
However, over a field $k$ of characteristic $2$, 
the determinant is trivial while the Dickson invariant is nontrivial 
(and it is nontrivial on $k$-points when the rank of the quadratic
space is even) \cite[Proposition C.2.8]{conrad:reductive-group-schemes}.

Over a field of characteristic 2, the Dickson invariant is sometimes also called the pseudodeterminant,
and the following explicit description, which follows from the fact that reflections always have nontrivial Dickson invariant, will be useful:
For any $T \in \o(Q)$, and any expression of $T$ as a product of reflections $T = r_{v_1} \cdots r_{v_s}$, (which exists so long as $(k, \rk V) \neq (\mathbb F_2, 4)$ by \autoref{remark:reflections-generate},)
the Dickson invariant is given by the map $\o(Q) \ra \bz/2\bz$ which sends $T \mapsto s \bmod 2.$
\end{remark}

\subsubsection{The Joint Kernel}
\label{subsubsection:omega}

\begin{definition}
	\label{definition:}
Define $\Omega(Q) \subset \o(Q)$ as $\Omega(Q) := \ker \dickson{Q} \cap \ker \sp{Q}$.
\end{definition}

Because the $-1$-spinor norm agrees with the $+1$-spinor norm when restricted to $\so(Q)$,
it follows that $\Omega(Q)$ is also the joint kernel of the Dickson map and the $+1$-spinor norm.

\subsection{Notation for the monodromy representation}
\label{subsection:notation-monodromy}

When $d > 0$, the map $\pi: \smsspace n d B \ra \smespace d B$ is finite \'etale, representing a locally constant constructible sheaf of rank $12d-4$ free $\bz/n\bz$ modules by 
\cite[Corollary 3.22]{landesman:geometric-average-selmer}.
For $B$ an integral noetherian $\bz[1/2n]$ scheme, letting $\vsel n d B$ denote the rank $12d-4$ free $\bz/n\bz$ module corresponding to the geometric generic fiber of $\pi$, 
we obtain a monodromy representation $\mono n d B : \pi_1(\smespace d B) \ra \gl(\vsel n d B)$ \cite[Definitions 4.1 and 4.2]{landesman:geometric-average-selmer}.

\begin{remark}
	\label{remark:}
	Strictly speaking, we should keep track of base points in our
	fundamental groups.
	However, as we will ultimately be concerned with integral base schemes $B$,
	changing basepoint only changes the map $\mono n d k$ by conjugation on the domain.
	Since
	we will only care about the image of $\mono n d k$,
	we will often omit the basepoint from our notation.
\end{remark}

For $R$ a ring, we use $\mono n d R$ to denote $\mono n d {\spec R}$.

\subsubsection{The adelic monodromy map}
For $n' \mid n$ both prime to $\chr(k)$, 
we obtain a map $\smsspace n d R \ra \smsspace {n'} d R$ over $\smespace d R$ induced by the corresponding map $\phi_{n,n'}: \mu_n \ra \mu_{n'}$ sending $y  \mapsto y^{n/n'}$
in the definition of
$\sspace n d R$ from \autoref{subsubsection:selmer-space}.
Because $\phi_{n, n''} = \phi_{n', n''} \circ\phi_{n,n'}$, the monodromy maps $\mono n d R: \pi_1(\smespace d R) \ra \gl(\vsel n d R)$
fit together compatibly to define a monodromy representation $\mono {\zprime p} d R : \pi_1(\smespace d R) \ra \gl(\vsel {\zprime p} d R)$.
For $n$ prime to $p$, we have a natural reduction $\bmod n$ map $r_n: \gl(\vsel {\zprime p} d R) \ra \gl(\vsel n d R)$
and $\mono {\zprime p} d R$ is uniquely characterized by the property that for all $n$ prime to $p$, $r_n\left( \mono {\zprime p} d R \right) = \mono n d R$.

\subsection{An equidistribution result}
\label{subsection:equidistribution}

For $x \in \espace d {\bz[1/2]}$ let $\frob_x$ be the conjugacy class of (geometric) Frobenius at $x$ in $\pi_1(\espace d {\bz[1/2]})$. In this section we prove an equidistribution result for Frobenius classes in the monodromy group, in the large $q$ limit.
To state the proposition, we define the ``mult'' map.

\begin{definition}
	\label{definition:mult}
	Let $X$ be a geometrically connected finite type scheme over $\mathbb F_q$, let $G$ be a profinite group, and let $\lambda: \pi_1(X) \ra G$ be a group homomorphism.
	Let $G_0$ denote the image of the composition $\pi_1^{\mrm{geom}}(X) := \pi_1(X_{\overline{\mathbb F}_q}) \ra \pi_1(X) \ra G$ and let $\Gamma := G/G_0$. Then, we define
	$\mult : G \ra \Gamma$ as the natural projection.
	Because $\pi_1(\spec \mathbb F_q) = \pi_1(X)/\pi_1^{\mrm{geom}}(X)$, we obtain a resulting map $\pi_1(\spec \mathbb F_q) \ra \Gamma$.
	We let $\gamma_q$ denote the image in $\Gamma$ of geometric Frobenius. 
\end{definition}
The following is an equidistribution result for Frobenii in a monodromy group, which is a generalization of \cite[Theorem 1]{kowalski:on-the-rank-of-quadratic-twists}.

\begin{proposition}
	\label{proposition:chavdarov}
Let $\Cal{X}$ be a smooth affine scheme of finite type over $\Cal{O}[1/S]$, where $\Cal{O}$ is a ring of integers in a number field, with geometrically irreducible fibers. For $\mf{q}$ a maximal ideal of $\Cal{O}[1/S]$ with residue field $\F_q$, write $X := \Cal{X}|_{\Cal{O}/\mf{q}}$. Assume that we have a commutative diagram 
\begin{equation}
	\label{equation:monodromy-and-mult}
\begin{tikzcd}
1 \ar[r] &  \pi_1^{\mrm{geom}}(X) \ar[d, "\lambda_0"] \ar[r] & \pi_1(X) \ar[r, "\deg"] \ar[d, "\lambda"] & \wh{\Z} \ar[r] \ar[d, "1 \mapsto \gamma_q^{-1}"] & 1 \\
1 \ar[r] & G_0 \ar[r] & G \ar[r, "\mult"] & \Gamma \ar[r] & 1 
\end{tikzcd}
\end{equation}
with $\lambda_0$ tamely ramified and surjective, $G$ a finite group, and $\Gamma$ abelian. Suppose $C \subset G$ is a conjugacy-invariant subset. Then 
\[
	\mrm{Prob}\{x \in X(\F_{q^n}) \co \lambda(\Frob_x) \in C \} = \frac{\# C \cap G^{\mult \gamma_q^n}}{\# G_0}  + O_{\Cal{X}}\left(\# G \sqrt{ \frac {\#C \cap G^{\mult \gamma_q^n}}{q^n}}\right).
\]
where $G^{\mult \gamma_q^n} := \mult^{-1}(\gamma_q^n)$. 
Here the constant in the error term $O_{\Cal{X}}\left(\# G \sqrt{ \frac {\#C\cap G^{\mult \gamma_q^n}}{q^n}}\right)$ is independent of $\mf{q}$, the choice of $G$, and the choice of $\lambda$, so long as $\lambda_0$ is tamely ramified and surjective. 
\end{proposition}
\begin{proof}
	By the Lang--Weil bound, we have $\# \Cal{X}(\mathbb F_q) = q^{\dim \Cal{X}_{\mathbb F_q}}+ O_{\Cal{X}}(q^{\dim \Cal{X}_{\mathbb F_q} -1/2})$ 
	and so
	after multiplying both sides by $\# \Cal{X}(\mathbb F_q)$
	(see also \cite[Remark 2]{kowalski:on-the-rank-of-quadratic-twists}),
	this statement nearly appears in 
	\cite[Theorem 1]{kowalski:on-the-rank-of-quadratic-twists}.
	There are two differences however: First, Kowalski 
	assumes that $\#G$ is prime to $q$ instead of only that $\lambda_0$ is tamely ramified.
	Second, Kowalski works over a field instead of over $\Cal{O}[1/S]$.
	The proof of \autoref{proposition:chavdarov} is the same as that given in \cite[Theorem 1]{kowalski:on-the-rank-of-quadratic-twists}, once these two differences are addressed.

	First we address the tamely ramified constraint. 
	Indeed, a careful examination of the proof of \cite[Theorem 1]{kowalski:on-the-rank-of-quadratic-twists}, shows that
the only reason for assuming $\#G$ is prime to $q$ appears in the reference to \cite[Proposition 4.7]{kowalski2006large},
	which in turn only uses this assumption in its reference to
	\cite[Proposition 4.5]{kowalski2006large},
	which in turn only uses this assumption in \cite[(4.13)]{kowalski2006large}. 
	However,
	\cite[(4.13)]{kowalski2006large}
	 holds whenever
	$\lambda_0$, or the associated map labeled $\phi$ in \cite{kowalski2006large}, is tamely ramified,
	see \cite[2.6, Cor 2.8]{Illusie:bookTheEuler}.
	We note that a generic hyperplane section of a tamely ramified cover remains tamely ramified, using Bertini's theorem to ensure that the hyperplane intersects the divisor of ramification
	generically. Hence, \cite[Proposition 4.6]{kowalski2006large}, used in the proof of
	\cite[Proposition 4.5]{kowalski2006large},
	can be suitably generalized to include the assumption that the restriction of $\phi$ to the hyperplane is tamely ramified.

	Second, we address the issue of working over $\Cal{O}[1/S]$ in place of a finite field. 
	The proof in \cite{kowalski:on-the-rank-of-quadratic-twists} shows that if $X$ comes as the reduction of a smooth $\Cal{X}$ over $\Cal{O}[1/S]$, then the constant in the error term $O_{\Cal{X}}\left(\# G \sqrt{ \frac {\#C}{q^n}}\right)$ of \autoref{proposition:chavdarov} can be taken to be a sum of (compactly supported) Betti numbers of $\Cal{X}$, which is uniform in $\mf{q}$ by Ehresmann's Theorem and proper base change for compactly supported \'{e}tale cohomology. This applies in particular to the Selmer spaces, as they are smooth over $\Z[1/2]$. 
\end{proof}

In computing the image of the monodromy representation associated to the Selmer space, the following criterion for when an irreducible cover is geometrically connected will be crucial. 

\begin{corollary}
	\label{corollary:geometric-monodromy-criterion}
	Let $Y$ be a geometrically irreducible finite type $\mathbb F_q$ scheme and let $\pi: X \ra Y$ be a finite \'etale connected Galois $G$ cover
	corresponding to a surjective map $\rho: \pi_1(Y) \ra G$ which is tamely ramified.
	Then, $X$ is geometrically disconnected if and only if there exist infinitely many positive integers
	$i$ such that for all $y \in Y(\mathbb F_{q^i})$,
	$\rho(\frob_y) \neq \id \in G$.
\end{corollary}

\begin{proof}
	If $X$ is geometrically connected, then once $i$ is sufficiently large,
	there do exist $y \in Y(\F_{q^i})$ with $\rho(\frob_y) = \id$, using
the equidistribution of Frobenius elements in $G$ resulting from \autoref{proposition:chavdarov} (using that $G = G_0$ in that statement). 

We next show the converse.
Suppose $X$ is geometrically disconnected and let $j$ denote the number of components of $X_{\overline{\mathbb F}_q}$.
We claim that for any
$i$ relatively prime to $j$, $X_{\mathbb F_{q^i}}$ is connected.
Indeed, if 
$X_{\mathbb F_{q^i}}$ is disconnected, $\gal(\mathbb F_{q^i}/\mathbb F_q) \simeq \mathbb Z/i\mathbb Z$ would act nontrivially on the components of
$X_{\mathbb F_{q^i}}$, implying that $\gcd(j,i) > 1$.

To conclude the proof, it suffices to show that for any such $i$ relatively prime to $j$, 
and any
$y \in Y(\mathbb F_{q^i})$,
$\rho(\frob_y) \neq \id \in G$.
Indeed, if $\rho(\frob_y) = \id \in G$,
the fiber of $\pi: X \to Y$ over $y$ would necessarily be $\deg \pi$ copies of $y$, so in particular, $X$ would have some $\mathbb F_{q^i}$ point.
However, since $X_{\mathbb F_{q^i}}$ is connected but geometrically disconnected,
the $j$ geometric components of $X_{\overline{\mathbb F}_q}$ must be nontrivially permuted by the action of
$\gal(\overline{\mathbb F}_q/\mathbb F_{q^i})$.
In particular, this Galois action on the fiber $X_y$ over $y$ must be nontrivial, and so $X$ cannot have any $\mathbb F_{q^i}$ points.
\end{proof}

\begin{corollary}
	\label{corollary:selmer-equidistribution}
Retain the notation of \autoref{definition:mult}.
	For any $n \geq 1$ and $C \subset \im \mono n d {\bz[1/2n]}$ a conjugacy class and $\mathbb F_q$ a finite field of characteristic $p$ with $\gcd(p,2n) = 1$, we have
\[\Scale[0.9]{
\begin{aligned}
\frac{\# \left\{ x \in \smespace d {\bz[1/2n]}(\mathbb F_q): \mono n d {\bz[1/2n]}(\frob_x) \in C \right\} }{\# \smespace d {\bz[1/2n]}(\mathbb F_q) } = 
\begin{cases}
	\frac{\# C}{\# \im \mono n d {\overline{\mathbb F}_p}} + O_{n,d}\left( q^{-1/2} \right)  & \text{ if } \mult (C)= \gamma_q, \\
	0 & \text{ if } \mult (C) \neq \gamma_q.	  \\
\end{cases}
\end{aligned}}
\]
The same statement holds true with $\smestack d k$ in place of $\smespace d k.$ 
\end{corollary}

\begin{proof}
	Note that in this setting, the tameness assumption on $\mono n d {\overline k}$
	was verified in the proof of \cite[Proposition 4.9]{landesman:geometric-average-selmer},
	see especially the end of the first paragraph of \cite[p. 702]{landesman:geometric-average-selmer}.
	The first statement follows immediately from \autoref{proposition:chavdarov}.
	Note here that $G$ and $C$ as in the statement of \autoref{proposition:chavdarov} are fixed, and so we may absorb their orders into the constant in the error term $O_{n,d}(q^{-1/2})$.

	To deduce the equidistribution statement for $\smestack d k$ from $\smespace d k$, note that the monodromy representation for $\smestack d k$ is induced by the cover
	$\smsstack n d k \ra \smestack d k$. Further $\smsspace n d k$ is the pullback of $\smsstack n d k$ along $\smespace d k \ra \smestack d k$, i.e. the diagram
	\[
	\begin{tikzcd}
\smsspace n d k \ar[d] \ar[r] &    	\smsstack n d k \ar[d]   \\
\smespace d k  \ar[r] & 	   \smestack d k
	\end{tikzcd}
	\]
	is cartesian.
In other words, the monodromy representation associated to $\smsspace n d k \ra \smespace d k$ factors through $\pi_1(\smespace d k) \surj \pi_1(\smestack d k) $.
	This implies that if $x, y \in \smespace d k$ map to the same point in $\smestack d k$ then $\mono n d k(\frob_x) = \mono n d k(\frob_y)$.
	Because $\smestack d k = [\smespace d k/\bg_a^{2d+1} \rtimes \bg_m]$, Lang's theorem applied to the group $\bg_a^{2d+1} \rtimes \bg_m$
	shows that each $z \in  \smestack d k(\mathbb F_q)$ (counted with multiplicity according to automorphisms)
	has precisely $\bg_a^{2d+1} \rtimes \bg_m(\mathbb F_q)$ points
	lying over it in $\smespace d k(\mathbb F_q)$, all mapping to the same conjugacy class under $\mono n d k$.
	Therefore, the distribution of $\mono n d k(\frob_x)$ for $x \in \smespace d k(\mathbb F_q)$ agrees with the distribution
	$\mono n d k(\frob_z)$ for $z \in \smestack d k(\mathbb F_q)$.
\end{proof}

\subsection{Determining the image of monodromy}
\label{subsection:determining-image-of-monodromy}
In \cite[Theorem 4.4]{landesman:geometric-average-selmer}, a partial description of $\im \mono n d k$ was given for $k$ a field.
The goal of this section is to precisely compute $\im \mono n d k$.
First, we recall the description from \cite[Theorem 4.4]{landesman:geometric-average-selmer}.
Keeping notation as in \autoref{subsubsection:quadratic-forms-base-change},
for $(V, Q)$ a quadratic space over a ring $R$ 
with a map $R \ra \bz/n\bz$, we let $(V_n, Q_n) := (V_{\bz/n\bz},Q_{\bz/n\bz})$
and let $r_n: \o(Q) \ra \o(Q_n)$ denote the induced reduction $\bmod n$ map of orthogonal groups.
We will be most concerned with the case $R = \bz$ or $R = {\zprime p}$.

In \cite[Theorem 4.4]{landesman:geometric-average-selmer} 
a quadratic space $(\vsel \bz d k, \qsel \bz d k)$ over $\bz$ is defined.
This agrees with that defined in \autoref{definition:selmer-space-quadratic-form} by
\cite[Remark 4.5]{landesman:geometric-average-selmer}.
With these definitions, \cite[Theorem 4.4]{landesman:geometric-average-selmer} states
\[
r_n(\osp(\qsel \bz d \bz)) \subset \im \mono n d {\ol k} \subset  \im \mono n d {k} \subset \o(\qsel n d {\ol k}).
\]

We next recall a slight generalization of the usual cyclotomic character, which we shall need to characterize $\im \mono n d k$.
\begin{definition}
	\label{definition:cyclotomic-character}
	For $k$ a field of characteristic $p$, allowing $p = 0$, we define the {\em cyclotomic character} as the map
	$\chi_{\cyc}: \gal(\ol{k}/k) \ra \left( \zprime p \right)^\times$ defined as follows:
	For $\nu$ a positive integer with $(\nu, p) = 1$ when $p > 0$ and $\nu$ arbitrary when $p = 0$,
	let $\zeta_\nu$ be a primitive $\nu$th root of unity.
	For $\sigma \in \gal(\ol{k}/k)$,
	suppose	$\sigma(\zeta_\nu) = \zeta_\nu^{a_{\nu,\sigma}}$.
	Then, define
	$\chi_{\cyc}(\sigma) := (a_{\nu,\sigma})_{\nu}$, considered as an element of $\left( \zprime p \right)^\times$.
\end{definition}
\begin{remark}
	\label{remark:cyclotomic-character}
	Note that $\chi_{\cyc}$ of \autoref{definition:cyclotomic-character} is the usual cyclotomic character when $\chr(k) = 0$.
	Further, from the definition,
in the case $p \neq 0$, $k = \mathbb F_p$, and $q$ is a power of $p$, we have $\chi_{\cyc}(\frob_q) = q \in \left( \zprime p \right)^\times$.
\end{remark}

For the statement of \autoref{theorem:monodromy}, recall the notation for the spinor norm and Dickson invariant from
\autoref{subsection:orthogonal-notation}.
Also, 
let $\Delta_{\bz/2\bz} : \bz/2 \bz \ra \prod_{\text{primes }\ell \neq p} \bz/2\bz$ the diagonal inclusion.
For $k$ a field of characteristic $p$ and $d \in \bz_{\geq 2}$, let $\chi^{d-1}$ denote the composition
	\[
		\gal(\ol{k}/k) \xra{\chi_{\cyc}^{d-1}} \left(\zprime p \right)^\times \ra  \left(\zprime p\right)^\times / \left( \left(\zprime p \right)^\times  \right)^2.
	\]
\begin{theorem}
	\label{theorem:monodromy}
	Let $k$ be a field of characteristic $p$, allowing $p =0$,
	and let
	$d \in \bz_{\geq 2}$.
	With $\Delta_{\bz/2\bz}$ and $\chi^{d-1}$ defined above, 
	\begin{align*}
		\im \mono n d k = \dickson{\qsel {\zprime p} d k}^{-1}(\im \Delta_{\bz/2\bz}) \cap \left(\sp{\qsel {\zprime p} d k}\right)^{-1}(\im \chi^{d-1}).
	\end{align*}
	\end{theorem}
\begin{example}
	Let's explicate what \autoref{theorem:monodromy} says in the cases of interest to this paper.
\begin{itemize}
	\item If $k$ is algebraically closed or $d$ is odd, then 
	\[
	\im \mono {\zprime p} d {\ol k} = \dickson{\qsel {\zprime p} d k}^{-1}(\im \Delta_{\bz/2\bz}) \cap \ker \left(\sp{\qsel {\zprime p} d k}\right).
	\]
\item If $d$ is even and $k = \mathbb F_q$ has characteristic $p>0$, 
	using \autoref{remark:cyclotomic-character},
	we have 
	\begin{align*}
		\im \mono {\zprime p} d k = \dickson{\qsel {\zprime p} d k}^{-1}(\im \Delta_{\bz/2\bz}) \cap (\sp{\qsel {\zprime p} d k})^{-1}(\langle \left[ q \right] \rangle)
	\end{align*}
	where $\langle \left[ q \right] \rangle$ is the group generated by the class of $q$.
\end{itemize}
\end{example}
	
	We will prove \autoref{theorem:monodromy}
	at the end of this section in \autoref{subsection:monodromy-proof}.
	The general outline of the proof is as follows.
	First, in \autoref{subsection:big-monodromy}, we show the image of the monodromy representation contains $\Omega(\qsel {\zprime p} d k)$.
	Next, in \autoref{subsection:monodromy-tools}, we explain how to compute the spinor norm and Dickson invariant of images of Frobenius, in certain cases.
	Then, in \autoref{subsection:controlling-dickson} and \autoref{subsection:controlling-spinor} we compute the spinor norm and Dickson invariants on
	$\im \mono {\zprime p} d k$, for $k$ a finite field.
	Finally, we piece these parts together in \autoref{subsection:monodromy-proof}.

\subsection{Showing the monodromy is big}
\label{subsection:big-monodromy}

We next explain how to deduce $\Omega(\qsel {\zprime p} d k) \subset \im \mono {\zprime p} d {\ol k}$
by combining \cite[Theorem 4.4]{landesman:geometric-average-selmer} with some group theory.
\begin{lemma}
	\label{lemma:orthogonal-reduction}
	For $d \geq 2$ and $n \geq 1$, 
	we have $r_n(\osp(\qsel \bz d k)) \supset \Omega(\qsel n d k)$.
	In particular, combining this with \cite[Theorem 4.4]{landesman:geometric-average-selmer} gives $\Omega(\qsel n d k) \subset \im \mono n d {\ol k}$
	and so $\Omega(\qsel {\zprime p} d k) \subset \im \mono {\zprime p} d {\ol k}$.
\end{lemma}
\begin{proof}
	The last sentence follows from the first by \cite[Theorem 4.4]{landesman:geometric-average-selmer},
	which says $\osp(\qsel \bz d k) \subset \im \mono {\zprime p} d {\ol k}$.

	We turn our attention to proving the first statement.
	For every $v \in \vsel n d k$, with $Q_n^d(v) = -1$, there exists a lift $\widetilde{v} \in \vsel {\bz} d k$ with $Q_{\Z}^d(\wt{v}) = -1$, as is shown in the proof of \cite[Lemma 4.13]{de-jongF:on-the-geometry-of-principal-homogeneous-spaces} (which implicitly assumes $d \geq 2$ so that $(\vsel {\bz} d k, \qsel {\bz} d k)$ contains summands isomorphic to the hyperbolic plane).
	Let $R(\qsel n d k)$ denote the subgroup of $\o(\qsel n d k)$
	generated by elements of the form $r_w$ for 
	$v \in \vsel n d k$
	and 
	let $R'(\qsel n d k)$ denote the subgroup of $\o(\qsel n d k)$
	generated by elements of the form $r_v \circ r_w$ for 
	$v, w \in \vsel n d k$ with $\qsel n d k(v) = \qsel n d k(w) = -1$.
	We next show $R(\qsel n d k) = \o(\qsel n d k)$ and $R'(\qsel n d k) = \Omega(\qsel n d k)$.

Recall a quadratic space $(V,Q)$ over $\bz$ is \emph{unimodular} if $B_Q$ is invertible as a linear transformation over $\bz$ or equivalently the natural map induced by $B_Q$ from $V$ to $V^\vee$, the dual lattice, is an isomorphism.

	In the case that $n$ is a prime power,
	since $(\vsel \bz d k, \qsel \bz d k)$ is unimodular and nondegenerate
	of rank more than $5$ (see \cite[Remark 4.5]{landesman:geometric-average-selmer}), 
	it follows from \cite[Satz 2]{kneser:erzeugung-ganzzahliger-orthogonaler}
	that $R(\qsel n d k) = \o(\qsel n d k)$.
	By \cite[Satz 3]{kneser:erzeugung-ganzzahliger-orthogonaler} 
	it follows $R'(\qsel n d k) = \Omega(\qsel n d k)$. 
	Note that \cite[Satz 3]{kneser:erzeugung-ganzzahliger-orthogonaler}
	is stated for $R'(\qsel n d k)$ generated by 
	elements of the form $r_v \circ r_w$ for 
	$v, w \in \vsel n d k$ with $\qsel n d k(v) = \qsel n d k(w) = 1$,
	instead of $\qsel n d k(v) = \qsel n d k(w) = -1$.
	However, we may arrange the latter by applying
	\cite[Satz 3]{kneser:erzeugung-ganzzahliger-orthogonaler}
	to 
	$-\qsel n d k$
	in place of $\qsel n d k$.
	Therefore, $\Omega(\qsel n d k) = R'(\qsel n d k) \subset r_n(\o(\qsel \bz d k))$.

	For the general case, write $n = \prod_{i=1}^t p_i^{a_i}$ for pairwise
	distinct primes $p_i$.
	Since $\Omega(\qsel n d k) = \prod_{i=1}^t \Omega(\qsel {p_i^{a_i}} d k)$,
	it suffices to show the image of $\Omega(\qsel {p_i^{a_i}} d k) \ra \prod_{i=1}^t \Omega(\qsel {p_i^{a_i}} d k),$ included as the $i$th component,
	is contained in $r_n(\osp(q))$.
	For this, choose $v, w \in \vsel {p_i^{a_i}} d k$ with $\qsel {p_i^{a_i}} d k (v) = \qsel {p_i^{a_i}} d k(w) = -1$
	and choose lifts $\widetilde{v}, \widetilde{w}$ to $\vsel n d k$
	so that $\widetilde{v} \equiv \widetilde{w} \mod \prod_{1 \leq j \leq n, j \neq i} p_j^{a_j}$ and $\qsel n d k(\widetilde{v}) = \qsel n d k(\widetilde{w}) = -1$.
	We then find that $r_{\widetilde{v}} \circ r_{\widetilde{w}}$ agrees with
	$r_v \circ r_w$ when reduced $\bmod p_i^{a_i}$ and is the identity
	when reduced $\mod p_j^{a_j}$ for any $j \neq i$.
	It follows that 
	$r_n(\osp(q)) \supset \im(\Omega(\qsel {p_i^{a_i}} d k) \ra \prod_{i=1}^t \Omega(\qsel {p_i^{a_i}} d k)),$ as desired.
\end{proof}

\subsection{Tools to compute the Dickson invariant and spinor norm of Frobenius}
\label{subsection:monodromy-tools}
In this section, we prove \autoref{proposition:zywina} which allows us to compute the spinor norm and Dickson invariants of the images of Frobenius elements under the monodromy representation.
The following result essentially appears as \cite[Proposition 2.9]{zywina:inverse-orthogonal}, where an analog is stated over $\bz/\ell \bz$ in place of
$\zprime p$.
The following generalization has essentially the same proof, using that $L$-functions associated to elliptic curves
are power series with coefficients in $\bz$.
Slight care must be taken to deal with the fact that the determinant disagrees with the Dickson invariant over fields of characteristic $2$.

For $E$ an elliptic curve over $\mathbb F_q(t)$, we let $L(T, E)$ denote the {\em $L$-function} associated to $E$ and let $\varepsilon_E \in \{\pm 1\}$ denote {\em root number} associated to $E$, see 
\cite[\S 2.3]{zywina:inverse-orthogonal} and 
\cite[\S 2.2]{zywina:inverse-orthogonal} respectively for a definitions.
The only property of root numbers we will use is that they appear in the functional equation of the $L$ function associated to $E$.
Recall our notation $[E_x] = x \in \espace d k$ where $E_x$ is the elliptic curve corresponding to $x$ as in \autoref{remark:corresponding-elliptic-curve}.

\begin{proposition}[Mild generalization of \protect{\cite[Proposition 2.9]{zywina:inverse-orthogonal}}]
	\label{proposition:zywina}
	Let $d \geq 1$.
	\begin{enumerate}
		\item For $[E_x] = x \in \smespace d {\mathbb F_p}(\mathbb F_q)$,
			$\dickson{\qsel {\zprime p} d k}(\mono {\zprime p} d k (\frob_x)) = \Delta_{\bz/2\bz}((1-\varepsilon_{E_x})/2)$.
\item For $[E_x] = x \in \sfespace d {\mathbb F_p}(\mathbb F_q)$,
	whenever $\det(\id - \mono {\zprime p} d k (\frob_x)) \neq 0,$ we have
	\[
	\sp{\qsel {\zprime p} d k}(\mono {\zprime p} d k (\frob_x)) = [q^{d-1}],
	\]
where $[q]$ is the class of the integer $q$ in $\left( \zprime p \right)^\times / \left(\left( \zprime p \right)^\times  \right)^2$.
	\end{enumerate}
\end{proposition}
In order to prove \autoref{proposition:zywina} we will need the following Lemma, which is essentially shown in \cite[p. 10]{zywina:inverse-orthogonal}.

\begin{lemma}
	\label{lemma:generalized-eigenspace-is-rank}
	Let $d \geq 1$, $p$ an odd prime, $\ell$ a prime with $\ell \neq p$, 
	and $[E_x] = x \in \sfespace d {\mathbb F_p}(\mathbb F_q)$. 
	Then, letting $L(T, E_x)$ be the $L$-function associated to $E_x$, we have
	\[
	\det(\id -\mono {\bz_\ell}d {\mathbb F_p} (\frob_x) T | \vsel {\bz_\ell} d {\ol{\mathbb F}_q} ) = L(T/q, E_x),
	\]
	viewed as an equality of polynomials with coefficients in $\bz_\ell$.
	In particular, the analytic rank of $E_x$ is equal to the $\bz_\ell$-rank 
	of the generalized $1$-eigenspace of $\mono {\bz_\ell}d {\mathbb F_p} (\frob_x)$ on $\vsel {\bz_\ell} d {\ol{\mathbb F}_q} $.
	\end{lemma}

\begin{proof}
	Let $L(T, E_x)$ denote the $L$-function of $E_x$, which is in fact a polynomial of degree $12d-4$ with integral coefficients \cite[Theorem 2.2]{zywina:inverse-orthogonal}.
	Define $g_{x,\ell} := \mono {\bz_\ell} d {\mathbb F_p} (\frob_x)$.
	It suffices to show that 
	\[
	\det(\id - g_{x,\ell} T | \vsel {\bz_\ell} d {\ol{\mathbb F}_q} \otimes_{\bz_\ell} \bq_\ell ) = L(T/q, E_x)
	\]
	viewed as an equality with coefficients in $\bq_\ell$. As explained in \cite[p. 10]{zywina:inverse-orthogonal}, we have 
	\[
	L(T/q, E_x)  = \det(\id - \frob_x T | H^1(\bp^1_{\ol{\mathbb F}_q}, j_* T_\ell(E_{\ol x})) \otimes_{\bz_\ell} \bq_\ell)
	\]	
	where $j_* T_\ell(E_{\ol x})$ is defined as follows. Let $U$ denote the open subscheme of $\bp^1_{\ol{\mathbb F}_q}$ over which the minimal proper regular model of $E_x$ is smooth.
	Let $j: U \ra \bp^1_{\ol {\mathbb F}_q}$ denote the inclusion morphism.
	Let $E_{\ol x}[\ell^k]$ denote the rank $2$ locally free sheaf of $\bz/\ell^k \bz$ modules parameterizing the $\ell^k$ torsion of the smooth minimal proper regular model of $E_{\ol x}$ over $U$
with $j_* E_x[\ell^j]$ the pushforward sheaf on $\bp^1_{\ol{\mathbb F}_q}$.
Define $j_* T_\ell(E_{\ol x}):= \varprojlim_k j_* E_{\ol x}[\ell^k]$ with transition maps $j_* E_x[\ell^{k+1}] \ra j_* E_x[\ell^k]$ given by multiplication by $\ell$.

We next identify $H^1(\bp^1_{\ol{\mathbb F}_q}, j_* T_\ell(E_{\ol x}))$ with $\vsel {\bz_\ell} d {\mathbb F_p}$ so as to compare this representation with $\mono {\bz_\ell} d {\mathbb F_p}$.
By \autoref{lemma:mu-n-pushforward-to-selmer-space},
there is a natural identification between the geometric fiber of the Selmer space over $x$, $\sfsspace {\ell^k} d {\mathbb F_p} \times_{\sfespace d {\mathbb F_p}, x} \spec \ol{\mathbb F}_q \simeq H^1(\bp^1_{\ol{\mathbb F}_q}, j_* E_{\ol x}[\ell^k])$. Further, these are both free  $\bz/\ell^k \bz$ modules of rank $12d-4$ by \cite[Corollary 3.19]{landesman:geometric-average-selmer}.
	By compatibility of these isomorphisms with the maps $E[\ell^{k+1}] \ra E[\ell^k]$ we obtain the equality
	$\det(\id - g_{x,\ell} T | \vsel {\bz_\ell} d {\ol{\mathbb F}_q} \otimes_{\bz_\ell} \bq_\ell ) = L(T/q, E_x)$, viewed as an equality of polynomials with coefficients in $\bq_\ell$.

	To conclude the proof, it remains to explain why the final statement regarding analytic rank follows from the equality
	$\det(\id - g_{x,\ell} T) = L(T/q, E_x)$.
	The analytic rank is the largest power of $T-1$ dividing $L(T/q, E_x)= \det(\id - g_{x,\ell} T).$ 
	This agrees with the largest power of $T-1$ dividing $\det \left( g_{x,\ell}^{-1} - T \right)$, which is the characteristic polynomial of
	$g_{x,\ell}^{-1}$.
	Hence, the analytic rank agrees with the dimension of the generalized $1$-eigenspace of $g_{x,\ell}^{-1},$ 
	which is the same as the dimension of the generalized $1$-eigenspace of $g_{x,\ell}$.
	\end{proof}

\begin{proof}[Proof of \autoref{proposition:zywina}]
	Define $g_{x,\ell} := \mono {\bz_\ell} d {\mathbb F_p} (\frob_x)$.
	First, we verify (1) regarding the Dickson invariant.
	From the definition of the Dickson invariant from \autoref{subsubsection:dickson},
	to compute the $\dickson{\qsel {\zprime p} d {\mathbb F_p}}(\mono {\zprime p} d {\mathbb F_p} (\frob_x))$, it is equivalent to compute
	$\dickson{\qsel {\bz_\ell} d {\mathbb F_p}}(\mono {\bz_\ell} d {\mathbb F_p} (\frob_x))$ for each prime $\ell \neq p$ separately and show this is equal to $(1-\varepsilon_{E_x})/2$. 

	Next, observe that $\det(T-g_{x,\ell}) = \det(T- g_{x,\ell}^{-1})$.
	Indeed, for any nondegenerate quadratic space $(V, Q)$ and $M \in \o(Q)$, and for $M^t$ the transpose of $M$, we have $M^t B_Q M = B_Q \implies M^t = B_Q^{-1}M^{-1} B_Q$.
	Hence, the characteristic polynomial of $M$ agrees with that of $M^t$ which agrees with that of $M^{-1}$.
	Therefore, the characteristic polynomial of $g_{x,\ell}$ agrees with that of $g_{x,\ell}^{-1}$ using
	$g_{x,\ell} \in \o(\qsel{\bz_\ell} d k)$ by the easier containment of \cite[Theorem 4.4]{landesman:geometric-average-selmer}. 

	Therefore, we have
	\begin{align*}
	T^{12d-4} \det(\id - g_{x,\ell}T^{-1}) &= \det(T-g_{x,\ell}) = \det (T-g_{x,\ell}^{-1}) = \det(-g_{x,\ell}^{-1}) \det(\id - g_{x,\ell}T) \\
	& = (-1)^{12d-4} \det(g_{x,\ell}) \det(\id - g_{x,\ell}T) = \det(g_{x,\ell}) \det(\id - g_{x,\ell}T).
	\end{align*}
	By \cite[Theorem 2.2]{zywina:inverse-orthogonal} in conjunction with \autoref{lemma:generalized-eigenspace-is-rank}, we also have 
	\[
	T^{12d-4} \det(\id - g_{x,\ell}T^{-1}) = \varepsilon_{E_x} \det(\id - g_{x,\ell}T),
	\]
	implying $\det(g_{x,\ell}) =\varepsilon_{E_x}$. 
Note that in the case $\ell = 2$, we are using 
	crucially that we are working over $\bz_2$ which does not have characteristic $2$.
	The relation between the Dickson invariant and the determinant for matrices over $\bz_2$ given in
	\cite[Corollary C.3.2]{conrad:reductive-group-schemes}
	implies
	$(1)$.

	We next verify (2). It suffices to verify 
	$\sp{\qsel {\bz_\ell} d k}(\mono {\bz_\ell} d k (\frob_x)) = [q^{d-1}],$ for every prime $\ell \neq p$. As in the previous part, let $g_{x,\ell} := \mono {\bz_\ell} d {\mathbb F_p} (\frob_x)$.
	First, observe that as $\det(\id - g_{x,\ell}) \neq 0$, it follows that $g_{x,\ell}$ has trivial $1$-eigenspace. 
	Because the Dickson invariant for an orthogonal group over a nondegenerate
	free module of even rank is congruent to the rank of the $1$-eigenspace $\mod 2$ by
	\cite[p. 160]{taylor:the-geometry-of-the-classical-groups},
	we find $g_{x, \ell} \in \so(\qsel {\bz_\ell}d {\mathbb F_p})$.
	Therefore, $\sp{\qsel {\bz_\ell} d k}(g_{x,\ell})= \mrm{sp}^+_{\qsel {\bz_\ell} d k}(g_{x,\ell})$.
	By \cite[\S 2, Cor.]{zassenhaus:on-the-spinor-norm} (see also \cite[Theorem C.5.7]{conrad:reductive-group-schemes}),
	and $\sp{\qsel {\bz_\ell} d k}(-1) = \mrm{disc}(\qsel {\bz_\ell} d k)$ \cite[Lemma C.5.8]{conrad:reductive-group-schemes},
	one can compute the spinor norm of $g_{x,\ell}$ as 
	\begin{align*}
\sp{\qsel {\bz_\ell} d k}(g_{x,\ell}) &= \mrm{sp}^+_{Q_{\mathbb Z_\ell}^d}(g_{x,\ell}) = 
		\mrm{sp}^+_{Q_{\mathbb Z_\ell}^d}(-\id) \mrm{sp}^+_{Q_{\mathbb Z_\ell}^d}(-g_{x,\ell})  \\
	&= \mrm{disc}(\qsel {\bz_\ell} d k) \cdot \det\left(\frac{1-g_{x,\ell}}{2}\right)  \cdot (\bz_\ell^\times)^2 =
	2^{\rk \vsel {\bz_\ell} d k} \det (1-g_{x,\ell}))  \cdot (\bz_\ell^\times)^2 \\
	&= \det(\id - g_{x,\ell}) \cdot (\bz_\ell^\times)^2.
	\end{align*}
Then, using the identification
	$\det(\id - g_{x,\ell} T | \vsel {\bz_\ell} d {\ol{\mathbb F}_q} ) = L(T/q, E_x)$ of \autoref{lemma:generalized-eigenspace-is-rank},
	\begin{align*}
		\sp{\qsel {\bz_\ell} d k}(g_{x,\ell}) = \det(\id - g_{x,\ell}) \cdot (\bz_\ell^\times)^2 = L(1/q, E_x) \cdot (\bz_\ell^\times)^2.
	\end{align*}

	To conclude the proof, we only need check $L(1/q,E) \in q^{d-1} (\bz_\ell^\times)^2$.
	In fact, considering $L(T, E)$ as a polynomial with integer coefficients, we will verify $L(1/q, E) \in q^{d-1} (\bq^\times)^2$, and the fact that both $L(1/q,E)$ and $q^{d-1}$ lie in
	$\bz_\ell^\times$ will imply they agree up to a square in $\bz_\ell^\times$.
	Since 
	$\det(\id - g_{x,\ell}) = L(1/q, E_x)$
	and $\det(\id - g_{x,\ell}) \neq 0$,
we find that the $L$ function of $E_x$ has analytic rank $0$, meaning that $\ord_{T=1/q} L(T, E_x) = 0$ or equivalently $L(1/q, E_x) \neq 0$.
It follows from \cite[Corollary 2.6]{zywina:inverse-orthogonal} 
(as is deduced from the Birch and Swinnerton Dyer conjecture, applicable because the analytic rank and algebraic rank are both $0$)
that $L(1/q, E_x) = q^{0-1 + d} c_{E_x}\cdot \left( \bq^\times \right)^2$, for $c_{E_x}$
the Tamagawa number of $E_x$.
Observing that
$c_{E_x} = 1$ as $x \in \smespace d k$, we find $L(1/q,E_x) = q^{-1+d} \cdot \left( \bq^\times \right)^2$, as desired. 
\end{proof}

\subsection{Controlling the Dickson invariant}
\label{subsection:controlling-dickson}

Using \autoref{proposition:zywina}, we next compute the image of $\im \mono {\zprime p} d k$ under the Dickson invariant map.

\begin{lemma}
	\label{lemma:dickson-monodromy}
	For any field $k$ of characteristic $p \neq 2$ (allowing $p = 0$) and any height $d \geq 2$, 
the image of the map 
\[
\dickson{\qsel {\zprime p} d k} \circ \mono {\zprime p} d k: \pi_1(\smespace d k) \ra \prod_{\substack{\text{primes }\ell \neq p}}  \bz/2\bz
\]
is $\im (\Delta_{\bz/2\bz})$.
\end{lemma}
\begin{proof}
	First, because $r_n(\osp(\qsel \bz d \bc)) \subset \mono n d {\ol k}$ by \cite[Theorem 4.4]{landesman:geometric-average-selmer},
	the Dickson invariant must be nontrivial on $\im \mono n d {\ol k}$, as it is nontrivial on $\osp(\qsel \bz d \bc)$.
	Therefore, it is similarly nontrivial on $\im \mono {\zprime p} d {\ol k}$.
Therefore, to conclude the proof, it suffices to show $\im \dickson{\qsel {\zprime p} d k} \circ \mono {\zprime p} d k \subset \im \Delta_{\bz/2\bz}$.
Further, from the definition of profinite groups as a limit of finite groups, it suffices to show that for any integer $n$ of the form $n = \ell_1 \cdots \ell_t$,
for primes $\ell_1,\ldots, \ell_t$
with no $\ell_i =p$,
$\im \dickson{\qsel {n} d k} \circ \mono {n} d k$ is contained in $\im \Delta_{\bz/2\bz}$.

By base change, it suffices to establish the containment
$\im \dickson{\qsel {n} d k} \circ \mono {n} d k \subset \im \Delta_{\bz/2\bz}$ when $k$ is either $\bq$ or a finite field of odd characteristic.
If the composition $\dickson{\qsel n d k} \circ \mono n d k$ defines a surjective map $\pi_1(\smespace d k) \ra G$, we obtain a resulting
finite \'etale Galois $G$-cover $U_{G,n,d,k} \ra \smespace d k$.
By Chebotarev density, for example as in \cite[Lemma 1.2]{ekedahl1988effective},
it suffices to establish that $U_{G,n,d,\bq}$ is geometrically connected and to establish
the claim for all finite fields $k$ of odd characteristic.
Further, geometric irreducibility for $U_{G,n,d,\bq}$ follows from geometric irreducibility of
$U_{G,n,d,\mathbb F_p}$ for all but finitely many primes $p$, because $U_{G,n,d,k} \ra \smespace d k \ra \spec k$ is in fact the base change of a map
$U_{G,n,d,\mathbb Z[1/2]} \ra \smespace d {\bz[1/2]} \ra \spec \bz[1/2]$, and the set of fibers on which a map is geometrically connected is constructible
\cite[Corollaire 9.7.9]{EGAIV.3}.
Hence, it suffices to demonstrate that for each finite field $k$ of odd characteristic, 
$\im \dickson{\qsel {n} d k} \circ \mono {n} d k$ is contained in $\im \Delta_{\bz/2\bz}$
and $U_{G,n,d,k}$ is geometrically connected.

For all finite fields $k$ of odd characteristic and all $x \in \smespace d k(k)$,
by \autoref{proposition:zywina} we have
$\dickson{\qsel n d k} \circ \mono n d k(\frob_x) \subset \im \Delta_{\bz/2\bz}$.
For all sufficiently large finite fields of odd characteristic, it follows from
\autoref{proposition:chavdarov} applied to the $G$-cover $U_{G,n,d,k} \ra \smespace d k$ constructed above
that $\im \dickson{\qsel n d k} \circ \mono n d k \subset \im \Delta_{\bz/2\bz}$.
Since the reverse containment also holds, we have equality for all sufficiently large (in the sense of divisibility of cardinality) finite fields.

We claim that the cover $U_{G,n,d,k} \ra \smespace d k$ is tamely ramified.
Indeed, this holds because we are assuming
$k$ does not have characteristic $2$, while the cover
$U_{G,n,d,k} \ra \smespace d k$
has degree which is a power of $2$ because the Dickson invariant takes values in a $2$-group.

It follows from \autoref{corollary:geometric-monodromy-criterion} that 
over any finite field $k$, the resulting $G$-cover is geometrically connected,
and so the containment 
$\dickson{\qsel n d k} \circ \mono n d k(\frob_x) \subset \im \Delta_{\bz/2\bz}$
in fact holds for all finite fields of odd characteristic.
\end{proof}

\subsection{Controlling the spinor norm}
\label{subsection:controlling-spinor}
We next use \autoref{proposition:zywina}(2) to analyze the spinor norm applied to $\im \mono {\zprime p} d k$.
The general strategy in what follows will be to compute the image of the spinor norm restricted to the kernel 
of the Dickson invariant, and then use this to deduce the joint image of the spinor norm and Dickson invariant.

For this proof, we will need to know there are many elliptic curves $[E_x] \in \smespace d k$
with trivial $1$-eigenspace.
This will follow from the group theoretic statement soon established in \autoref{proposition:eigenvalues}.
In order to state this precisely, we recall a relevant distribution on the $\ell$-adic points of a 
finite type scheme from \cite{bhargavaKLPR:modeling-the-distribution-of-ranks-selmer-groups}.
All but the last statement appears in 
\cite[Lemma 2.1(b)]{bhargavaKLPR:modeling-the-distribution-of-ranks-selmer-groups}, while the last statement appears in
\cite[Corollaire, p. 146]{Serre:quelquesCheb}.

\begin{lemma}
	\label{lemma:subscheme-measure-0}
	Let $X$ be a finite type $\mathbb Z_\ell$ scheme of dimension $d$ and equip $X(\mathbb Z_\ell)$ with the $\ell$-adic topology.
	There exists a unique bounded $\br_{\geq 0}$-valued measure $\mu_X$ on the Borel
	$\sigma$-algebra of $X(\bz_\ell)$ such that for any open and closed subset
	$S$ of $X(\bz_\ell)$, we have
	\begin{align*}
		\mu_X(S) = \lim_{e \ra \infty} \frac{\#\left( \text{image of $S$ in }X(\bz/\ell^e \bz) \right)}{\left( \ell^e \right)^d}.
	\end{align*}
	If $Y \subset X$ is a subscheme of dimension $<d$, $\mu_X(Y(\mathbb Z_\ell))  = 0$
	and 
	\begin{align*}
		\#\left( \im \left(Y(\bz/\ell^e\bz) \ra X(\bz/\ell^e \bz) \right)\right) = O_{Y}(\ell^{e(d-1)}).
	\end{align*}
	
\end{lemma}

\begin{remark}
	\label{remark:}
	\autoref{lemma:subscheme-measure-0} is correct as stated, but the proof in \cite[Proposition 2.1(b)]{bhargavaKLPR:modeling-the-distribution-of-ranks-selmer-groups}
	has a minor error.
	There, it is stated that $\#Y(\bz/\ell^e \bz) = O\left( (\ell^e)^{d-1} \right)$, which is not in general true.
	The correct statement is that $\im \left( Y(\bz_\ell) \ra Y(\bz/\ell^e \bz) \right) = O\left( (\ell^e)^{d-1} \right)$. 
	A counterexample to the incorrect statement is provided by the subscheme $Y = \spec \bz[x]/(x^2)$ and $X = \ba^1_{\bz_\ell}$.
	In this case, we easily see that $\# Y(\bz_\ell) = 1$ because $\bz_\ell$ is reduced, but $\# Y(\bz/\ell^e\bz) = \ell^{\lfloor e/2 \rfloor}$
	as such points are in bijection with elements of $\bz/\ell^e \bz$ which square to $0$.
\end{remark}

In the following proposition only, we use $\o(Q)$ and $\so(Q)$ to denote the algebraic groups
associated to a quadratic form $Q$, and $\o(Q)(R)$ to denote its $\spec R$ points, for $R$ a ring.
\begin{proposition}
	\label{proposition:eigenvalues}
	Let $(V, Q)$ be a nondegenerate quadratic space of even rank at least $4$ over $\bz_\ell$.
	There is a Zariski closed pure codimension $1$ subscheme $Z \subset \o(Q)$,
	such that $g \in Z$ if and only if $g$ has a generalized $1$-eigenspace of dimension at least $2$.
	
Further, any $g \in (\o(Q) - Z)(\bz_\ell)$
	has a zero dimensional generalized $1$-eigenspace and zero dimensional $1$-eigenspace when
	$g \in \so(Q)(\bz_\ell)$
	and a one dimensional generalized $1$-eigenspace and one dimensional $1$-eigenspace when 
	$g \notin \so(Q)(\bz_\ell)$.

	In particular, $Z(\bz_\ell)$ has measure $0$
	with respect to the distribution of \autoref{lemma:subscheme-measure-0}.
\end{proposition}
\begin{proof}
	For $V_L$ an even dimensional
	free module over a field $L$ and $g: V_L \ra V_L$, let 
	$V_L^{g=\lambda}$ denote the $\lambda$-eigenspace and 
	$V_L^{[g=\lambda]}$ denote the generalized $\lambda$-eigenspace. 
	Let $Q_L$ be a nondegenerate quadratic form on $V_L$.
	Recall that the Dickson invariant agrees with $\dim V_L^{g=1} \mod 2$, using that $\dim V_L$ is even and \cite[p. 160]{taylor:the-geometry-of-the-classical-groups}.
	(In \cite[p. 160]{taylor:the-geometry-of-the-classical-groups} 
	the notation $[V,f]$ is used for $\im (1-f)$, whose rank taken $\mod 2$ agrees with $\dim V_L^{g=1} \mod 2$ since $\dim V_L$ is even.)

	In particular, every element in $(\o(Q_L) - \so(Q_L))(L)$ has odd dimensional $1$-eigenspace while every element
	of $\so(Q_L)(L)$ has even dimensional $1$-eigenspace.
	Now, let $(V,Q)$ be a nondegenerate even rank quadratic space over $\bz_\ell$ as in the statement of the proposition.
	We may apply the above discussion to the base change
	$(V_{\bq_\ell}, Q_{\bq_\ell})$ to deduce that any element $g \in \so(Q)(\bz_\ell)$
	has $\rk V_L^{g=1} \equiv 0 \mod 2$ and any element of $g \in (\o(Q) - \so(Q))(\bz_\ell)$ has $\rk V_{\bq_\ell}^{g=1} \equiv 1 \mod 2$.

	Further, the condition that an element $g \in \so(Q)(\bz_\ell)$ has $\rk V_{\bq_\ell}^{[g = 1]} > 0$ is Zariski closed and nonempty in the algebraic group $\so(Q)$ over $\bz_\ell$; it is Zariski closed because
	this condition can be expressed as $T-1$ dividing the characteristic polynomial of $g$ and	
	it is nonempty because there are elements in a maximal
	torus with $\dim V_{\bq_\ell}^{g = 1} = 0$.
	Similarly, the condition that an element $g \in (\o(Q) - \so(Q))(\bz_\ell)$ has $\rk V_{\bq_\ell}^{[g=1]} > 1$ is Zariski closed and nonempty. 
	(This uses that $\chr \mathbb Z_\ell = 0 \neq 2$, as in characteristic $2$ every element of $\o(Q) - \so(Q)$ would have generalized $1$ eigenspace of dimension at least $2$.)
	Therefore, to establish the statement regarding generalized $1$-eigenspaces,
	it suffices to show that a proper Zariski closed subscheme of an integral scheme over $\bz_\ell$ parameterizes a measure $0$ subset,
	which is the content of \autoref{lemma:subscheme-measure-0}.

	The statement for generalized $1$-eigenspaces established above implies the corresponding statement for $1$-eigenspaces because when the generalized
	$1$-eigenspace is at most $1$ dimensional, it is equal to the $1$-eigenspace.
	The final statement that $Z(\bz_\ell)$ has measure $0$ follows from \autoref{lemma:subscheme-measure-0}.
\end{proof}

We next define a double cover
$\scz^d_k \ra \smespace d k$ so that the Dickson invariant is trivial on $\pi_1(\scz^d_k)$.

\begin{definition}
	\label{definition:dickson-double-cover}
	Let $n \geq 1, d\geq 2$, and let $k$ be an integral domain (not necessarily a field) on which $2n$ is invertible. By \autoref{lemma:dickson-monodromy}, the Dickson invariant defines a surjective map $\pi_1(\smespace d k) \ra \bz/2\bz$ and hence corresponds to a
finite \'etale $\bz/2\bz$ cover $\scz^d_k \ra \smespace d k$.
This yields a map $\pi_1(\scz^d_k) \ra \so(\qsel n d k)$ which is identified with the restriction of $\mono n d k$ to the kernel of the Dickson invariant.
\end{definition}

In the case $k$ is a field, by abuse of notation, we have a map $\chi_{\cyc}: \pi_1(\spec k) \ra \left( \bz/n\bz \right)^\times$
(induced by the cyclotomic character $\chi_{\cyc}$ to $\left( \zprime p \right)^\times$ from \autoref{definition:cyclotomic-character}).
In the general case where $k$ is just an integral domain, we also obtain a map $\chi_{\cyc}: \pi_1(\spec k) \ra \left( \bz/n\bz \right)^\times$
which can be defined as the unique map making the diagram below commute: 
\begin{equation}
	\label{equation:}
	\begin{tikzcd}
		\pi_1(\Frac(k)) \ar {rr} \ar[dr, "\chi_{\cyc}"'] && \pi_1(\spec k) \ar {ld}{\chi_{\cyc}} \\
		& \left( \bz/n\bz \right)^\times & 
	\end{tikzcd}
\end{equation}
We have a diagram
\begin{equation}
	\label{equation:spinor-map}
	\begin{tikzcd} 
		\pi_1(\scz^d_k) \ar {rr} \ar {d} && \so(\qsel n d k) \ar {dd}{\sp{\qsel n d k}} \\
		\pi_1(\smespace d k) \ar {d} && \\
		\pi_1(\spec k) \ar {r}{\chi_{\cyc}^{d-1}} & \left( \bz/n\bz \right)^\times \ar{r} & \left( \bz/n\bz \right)^\times / \left( \left( \bz/n\bz \right)^\times \right)^2.
\end{tikzcd}\end{equation}
\begin{lemma}
	\label{lemma:spinor-norm-commutes}
	The square \eqref{equation:spinor-map} commutes when $k$ is a field of characteristic prime to $2n$.
\end{lemma}
\begin{proof}
	Because commutativity of \eqref{equation:spinor-map} is compatible with base change on the integral domain $k$, it suffices to verify it in the cases that $k = \bq$ and that $k$ is a finite field of characteristic prime to $2n$.
	
	First, we verify the claim when $k$ is a finite field of characteristic prime to $2n$.
	It suffices to establish the claim for all sufficiently divisible $n$. Hence, to simplify matters latter, we make the further harmless assumption that $8 \mid n$.
	Using that $\left( \bz/n\bz \right)^\times / \left( \left( \bz/n\bz \right)^\times \right)^2$ has even order,	
	it suffices to verify commutativity of \eqref{equation:spinor-map} for all sufficiently large finite fields of characteristic $p$ with $\gcd(p, 2n) = 1$, and odd degree over $\mathbb F_p$.

	Now, for such sufficiently large finite fields, we only need verify that that for varying $x \in \scz(k)$, 
	$\sp{\qsel n d k}(\mono n d k(\frob_x))$ is always equal to $\left[ q^{d-1} \right]$.
	By 
	\autoref{proposition:chavdarov},
	Frobenius elements are equidistributed in a coset of the geometric monodromy group
	and so it suffices to establish 
	$\sp{\qsel n d k}\mono n d k(\frob_x) = \left[ q^{d-1} \right]$
	for a subset of $x \in \smespace d k(k)$ with density in $\smespace d k(k)$ tending to $1$
	as $\# k \ra \infty$.
	Further, we note that the spinor norm is unchanged upon replacing $n$ with $n^j$ for any $j \geq 1$.
	Note that here we are using the assumption $8 \mid n$, as, for example,
	$\sp{\qsel 2 d k}$ maps to the trivial group while $\sp{\qsel 4 d k}$ maps to a nontrivial group.
	By replacing $n$ with a sufficiently large power we can ensure that the density of $g \in \im \mono n d k$ with a $0$-dimensional $1$ eigenspace is arbitrarily
	close to $1$ by \autoref{proposition:eigenvalues}.
	Recall that, by the Lang-Weil estimates, 
	if $X$ is a scheme over $\spec \mathbb Z$ with geometrically irreducible fibers and $U \subset X$ a fiberwise dense open subscheme
	$\lim_{\#k \to \infty} \frac{\# U(k)}{\# X (k)} = 1$.
	Since $\sfespace d {\spec \bz[1/2]} \subset \smespace d  {\spec \bz[1/2]}$ is a fiberwise dense open subscheme
	by \cite[Lemma 3.14]{landesman:geometric-average-selmer},
	we find that
	$\sfespace d k(k)$ has density $1$ in $\smespace d k(k)$ as $\# k \to \infty$.

	and so it suffices to verify the above when $x \in \sfespace d k(k)$.
	Hence, we want to verify commutativity of \eqref{equation:spinor-map} for all $x \in \sfespace d k(k)$ with a $0$-dimensional $1$ eigenspace,
	which is the content of \autoref{proposition:zywina}(2).

	So, to finish the proof, it only remains to deal with the case $k = \bq$.
	Since \eqref{equation:spinor-map} is in fact defined over the integral domain $k = \bz[1/2]$, and is compatible with base change along $\spec \bq \ra \spec \bz[1/2n]$, it suffices
	to verify commutativity when $k = \spec \bz[1/2n]$.
	Via the bijection between maps $\pi_1(\scz^d_{\bz[1/2n]}) \ra G$ and $G$-covers of $\scz^d_{\bz[1/2n]}$, call $X$ and $Y$ the two 
	induced $\left( \bz/n\bz \right)^\times / \left( \left( \bz/n\bz \right)^\times \right)^2$-covers of $\scz^d_{\bz[1/2n]}$ 
	obtained by traversing the diagram
	\eqref{equation:spinor-map} in the two different paths.
	We wish to show $X$ and $Y$ are isomorphic.
	We obtain a 
	$\left( \bz/n\bz \right)^\times / \left( \left( \bz/n\bz \right)^\times \right)^2$-cover $T \ra \scz^d_{\bz[1/2n]}$ induced by the ``difference''
	of $X$ and $Y$; that is, if $X$ and $Y$ correspond to maps $f, g : \pi_1(\scz^d_{\bz[1/2n]}) \ra \left( \bz/n\bz \right)^\times / \left( \left( \bz/n\bz \right)^\times \right)^2$
	then $T$ corresponds 
	to the homomorphism $t(\alpha) = f(\alpha) g(\alpha^{-1})$. 
	To conclude the proof, it suffices to show $T$ is trivial.

	We first verify $T \times_{\spec \bz[1/2n]} \spec \bq \ra \scz^d_{\bq}$ is the pullback of a cover $S \ra \spec \bq$ along the structure map $\scz^d_\bq \ra \spec \bq$.
	By the established case of finite fields and compatibility with base change, we know $T$ becomes trivial after base change of $T \ra \scz^d_{\bz[1/2n]} \ra \spec \bz[1/2n]$ along any closed point $\spec \mathbb F_p \ra \spec \bz[1/2n]$.
	We now apply \cite[Proposition 9.7.8]{EGAIV.3},
	which states that the number of geometric components of a morphism is constant on some open set,
	to the map $T \ra \spec \bz[1/2n]$.
	It follows that the cover $T \ra \scz^d_{\bz[1/2n]}$ is trivial when restricted to $\spec \overline{\bq} \ra \spec \bz[1/2n]$.
	This implies that the composite morphism $\pi_1(\scz^d_{\ol{\bq}}) \ra \pi_1(\scz^d_{\bq}) \ra \left( \bz/n\bz \right)^\times / \left( \left( \bz/n\bz \right)^\times \right)^2$
	is trivial.
	From the exact sequence \cite[Expos\'e IX, Th\'eor\`eme 6.1]{sga1}
\begin{equation}
	\label{equation:}
	\begin{tikzcd}
		0 \ar {r} & \pi_1(\scz^d_{\ol{\bq}})\ar {r} & \pi_1(\scz^d_{\bq}) \ar {r} & \pi_1(\spec \bq) \ar {r} & 0 
\end{tikzcd}\end{equation}
we obtain that the cover $T \times_{\spec \bz[1/2n]} \spec \bq \ra \scz^d_\bq$ is the pullback of a cover $S \ra \spec \bq$ along the structure map $\scz^d_\bz \ra \spec \bq$.

To conclude, we wish to show $S$ is a trivial cover of $\spec \bq$.
By Chebotarev density, it suffices to show that the normalization of $\spec \bz$ in $S$ is the trivial cover over a density $1$ subset of primes.
Since $S$ pulls back to $T \times_{\spec \bz[1/2n]} \spec \bq$ along the map $\scz^d_{\bz[1/2n]} \ra \spec \bq$, it suffices to show that $T \ra \scz^d_{\bz[1/2n]}$ is the trivial cover over a density $1$ subset of primes. Indeed, this triviality holds by the previously established commutativity of \eqref{equation:spinor-map} when $\chr(k)$ is positive.
\end{proof}

Recall in \autoref{definition:dickson-double-cover},
we defined $\scz^d_k$ as the double cover of $\smespace d k$ corresponding to the kernel of the Dickson invariant.
That is, $\pi_1(\scz^d_k) = \ker(\dickson{\qsel {\zprime p} d k}) : \pi_1(\smespace d k) \to \bz/2\bz$.
\begin{lemma}
	\label{lemma:spinor-monodromy}
	For a field $k$ of characteristic $p \neq 2$ (allowing $p = 0$) and any height $d \geq 2$, 
	the image of the spinor norm map restricted to 
	$\ker(\dickson{\qsel {\zprime p} d k})$
\[
	\sp{\qsel {\zprime p} d k} \circ \mono {\zprime p} d k|_{\ker(\dickson{\qsel {\zprime p} d k})}
		: \pi_1(\scz^d_k) \ra \left( \zprime p \right)^\times / \left( \left( \zprime p \right)^\times \right)^2
	\]
	is identified with the image of the composition
\begin{align}
	\label{equation:cyclotomic}
	\gal(\ol{k}/k) \xra{\chi_{\cyc}^{d-1}} \left(\zprime p \right)^\times \ra  \left(\zprime p\right)^\times / \left( \left(\zprime p \right)^\times  \right)^2.
\end{align}
\end{lemma}
\begin{remark}
	\label{remark:}
	In the case $k$ is algebraically closed or $d$ is odd, \autoref{lemma:spinor-monodromy} says the image of the spinor norm map
	$\sp{\qsel {\zprime p} d k} \circ \mono {\zprime p} d k$, when restricted to the kernel of the Dickson invariant, is trivial.
\end{remark}
\begin{proof}
	It suffices to establish the claim for 
	all finite $n$, with no prime factor of $n$ equal to $p$,
	in place of $\zprime p$. 
	The result then follows from \autoref{lemma:spinor-norm-commutes}.
\end{proof}

\subsection{Proving \autoref{theorem:monodromy}}
\label{subsection:monodromy-proof}
Combining the results of the preceding subsections, we are ready to complete our monodromy computation.
\begin{proof}[Proof of \autoref{theorem:monodromy}]
	First, by \autoref{lemma:orthogonal-reduction}, we find $\Omega(\qsel {\zprime p}d k) \subset \im \mono {\zprime p} d {\ol k}$.
	As 
	\begin{align*}
		\Omega(\qsel {\zprime p}d k) = 
\ker\left( \o\left( \qsel {\zprime p} d k \right) \xra{(\dickson{\qsel {\zprime p} d k}, \sp{\qsel {\zprime p} d k})} \left( \prod_{\substack{\text{primes }\ell,\\ \ell \neq p}}  \bz/2\bz \right) \times \left(  \widehat{\bz}^\times /(\widehat{\bz}^\times)^2\right) \right),
	\end{align*}
determining $\im \mono {\zprime p} d k$ is equivalent to determining the image of $(\dickson{\qsel {\zprime p} d k}, \sp{\qsel {\zprime p} d k})\circ\im \mono {\zprime p} d k$.

First, because $r_n(\osp(\qsel \bz d \bc)) \subset \mono n d {\ol k}$ for every $n \geq 1$ and prime to $p$, by \cite[Theorem 4.4]{landesman:geometric-average-selmer},
$\mono n d {\ol k}$ does contain elements with trivial spinor norm and nontrivial Dickson invariant.
Therefore, since we know the image of the Dickson invariant map is $\Delta_{\bz/2\bz}(\bz/2\bz)$ by \autoref{lemma:dickson-monodromy},
it follows that $\im \mono {\zprime p} d k$ contains $\ker \sp {\qsel {\zprime p} d k} \cap (\dickson {\qsel {\zprime p} d k})^{-1}(\Delta_{\bz/2\bz}(\bz/2\bz))$.

Therefore, the image of the joint map $(\dickson{\qsel {\zprime p} d k}, \sp{\qsel {\zprime p} d k})\circ\im \mono {\zprime p} d k$
is generated by $\Delta_{\bz/2\bz}(\bz/2\bz) \times \id$ together with the image of the spinor norm when restricted to the kernel of the Dickson invariant.
The latter image is given in the theorem statement by \autoref{lemma:spinor-monodromy}.
Therefore, 
the joint map 
$(\dickson{\qsel {\zprime p} d k}, \sp{\qsel {\zprime p} d k})$ has image
as claimed in the statement of
\autoref{theorem:monodromy}.
\end{proof}

\section{The distribution of $\Sel_\ell$}\label{sec: prime distribution}

In this section we will prove the key results towards showing that the BKLPR heuristic agrees with the geometric distribution of $\Sel_{\ell}$, for prime $\ell$. The psychology of the problem is as follows: one would like to ``understand'' the distributions by computing numerical invariants such as moments, but the distributions in question are not determined by their moments, since these moments grow too quickly. However, both distributions are the limit as a certain ``height'' parameter tends to infinity, and at finite height they are distributions on finite sets, hence obviously determined by their moments. We can then verify that the two limiting distributions agree by showing that the ``finite height'' distributions are very close, which we can then do by computing enough moments. 

The key point that makes this computation feasible is that \emph{the moments stabilize very quickly as the height grows}. It was already observed in \cite[Theorem 1.2]{landesman:geometric-average-selmer} that the \emph{first} moment (i.e., average) size of $\Sel_{\ell}$ for height $d$ elliptic curves (in the large $q$ limit) is already equal to its limiting value as soon as the height $d$ is at least $2$. In this section we go much further, computing the first $6d-2$ moments for the large $q$ limit of families of elliptic curves with height $d$ (in the large $q$ limit), and showing that they are all already equal to their limiting values. Even computing one fewer moment would be insufficient for our purposes, and it seems that computing one more moment in closed form would be quite difficult, as the next moment is \emph{not} equal to its limiting value! 

We caution, however, that the distribution at finite height depends quite delicately on the monodromy group; for example, the large $q$ limit does not literally exist because of small fluctuations among the monodromy groups, but the difference between its $\liminf_{q \rightarrow \infty}$ and $\limsup_{q \rightarrow \infty}$ will tend to 0 as the height tends to infinity. 

We now give an outline of the contents of this section.
In \autoref{subsection:prime-distribution}, we introduce the random kernel model, which is our model for Selmer groups that directly connects to points of the Selmer space.
This model will be defined in terms of kernels of random elements of subgroups of an orthogonal group, and so in
\autoref{subsec: dist 1-eigenspaces} we compute the probability distributions of the dimensions of these kernels.
In 
\autoref{subsection:proof-of-orbit-count}
we show how to determine compute the moments of the above mentioned random kernels, and then how to determine their distribution in terms of these moments,
which is used in
\autoref{subsection:tv-distance}
to bound the total variation distance between the random kernel model and the BKLPR model. We emphasize that these results a priori concern the random kernel model rather than $\Sel_n$, but later in \S \ref{section:proofs} it will be spelled out how to relate the two.

\subsection{The random kernel model}
\label{subsection:prime-distribution}

We introduce another probabilistic model which is closely related to the distribution of Selmer elements. We will continue to use the notation introduced earlier, especially from \autoref{subsubsection:quadratic-forms-base-change}.

\begin{definition}[Random 1-eigenspace for an element of $H$]\label{definition:random 1-eigenspace} Let $n$ and $d$ be positive integers. Let $H \subset \o(\qsel n d k)$ be a subset, where $ \o(\qsel n d k)$ is the orthogonal group for the quadratic form of \autoref{definition:selmer-space-quadratic-form}. We define $\rvcoset {V_{n}^d} H$ to be the random variable $\ker (g-\id)$, valued in isomorphism classes of $\bz/n\bz$-modules, for $g$ drawn uniformly at random from $H$.
\end{definition}

In this section, we will primarily be concerned with the
case of 
\autoref{definition:random 1-eigenspace} where $n = \ell$ is prime, but in \autoref{section:markov}, we will crucially use the case that $n = \ell^e$ is a prime power.
Now we will define the precise random variable that we end up relating to the distribution of ranks and Selmer groups of elliptic curves for our universal family. 

\begin{definition}[Random kernel model]\label{definition:kernel-distribution}
	For $n \in \bz_{\geq 1}, d \in \bz_{\geq 2}$ and $k$ a finite field of cardinality $q$ with $\gcd(q,2n) = 1$, 
	let $[q] \in \left(\bz/n\bz\right)^\times /\left( \left(\bz/n\bz\right)^\times \right)^2$ denote the class of $q$.
	Define
	\begin{align*}
	H_{n,k}^{d} := \left(\dickson{\qsel n d k}\right)^{-1}\left( \Delta_{\bz/2\bz}(\bz/2\bz)\right) \cap \left(\sp{\qsel n d k}\right)^{-1}([q^{d-1}]) \subset \o\left( \qsel n d k \right).
	\end{align*}
Define $\rvker n d k$ as the distribution on $\ab n$ given by
\begin{align*}
	\prob(\rvker n d k = G) := \frac{\#\{g \in H_{n,k}^d : \ker(g-\id) \simeq G\}}{\# H_{n,k}^d}.
\end{align*}
	Define $\flr n d k$ as the distribution on $\bz_{\geq 0} \times \ab n$ given by
\begin{align*}
	\prob(\flr n d k = (r, G) ) := 
	\begin{cases}
		\frac{\#\{g \in \so(\qsel n d k) \cap H_{n,k}^{d} : \ker(g-\id) \simeq G\}}{\# H_{n,k}^d} & \text{ if } r = 0, \\
		\frac{\#\{g \in \left( \o(\qsel n d k) - \so(\qsel n d k) \right) \cap H_{n,k}^{d} : \ker(g-\id) \simeq G\}}{\# H_{n,k}^d} & \text{ if } r = 1, \\
		0 & \text{ if } r \geq 2.\\
	\end{cases}
\end{align*}
\end{definition}

\autoref{theorem:monodromy}, adapted to the case of finite fields, gives: 

\begin{corollary}
	\label{corollary:finite-field-distribution}
	For $q$ ranging over all prime powers with $\gcd(q, 2n) = 1$ and $d \geq 2$ an integer,
	the distribution of $\im \mono n d {\bz[1/2n]}(\frob_x)$ ranging over $x \in \smestack d {\bz[1/2]} (\mathbb F_q)$,
	up to an error of $O_{n,d}(q^{-1/2}),$ agrees 
	with the distribution $\rvker n d {\mathbb F_q}$.
\end{corollary}

\begin{proof}
	First, by \autoref{corollary:selmer-equidistribution} to determine the distribution of Frobenius elements, it makes no difference whether we work with
	$\smestack d {\bz[1/2]}$ or $\smespace d {\bz[1/2]}$, so we choose to work with the latter.
	Observe that the monodromy agrees with the geometric monodromy (i.e., 
	$\im \mono n d {\mathbb F_q} = \im \mono n d {\ol{\mathbb F}_q}$)
	when $q$ is a square or $d$ is odd or $n \leq 2$, and has index $2$ in the geometric monodromy when $q$ is a square and $d$ is even and $n > 2$
	by \autoref{theorem:monodromy}.
	Therefore, in the former case, it is equidistributed in the monodromy group, which is $H_{n,k}^d$ in this case, 
	up to an error of $O_{n,d}(q^{-1/2})$
	by \autoref{proposition:chavdarov}.
	On the other hand, when $q$ is not a square and $d$ is even and $n > 2$, $\gamma_q$ as in \autoref{definition:mult} is nontrivial since the geometric
	monodromy is not equal to the monodromy. Hence, by \autoref{proposition:chavdarov}, $\frob_x$ is equidistributed in the nontrivial coset of 
	$\mono n d {\mathbb F_q} \subset \mono n d {\ol{\mathbb F}_q}$, which is precisely $\im \mono n d {\mathbb F_q} - \im \mono n d {\ol{\mathbb F}_q} = H_{n,k}^{d}$.
	
	The statement regarding the concrete characterization of the Dickson invariant and spinor norm is merely a restatement of \autoref{theorem:monodromy}.
\end{proof}

In \S \ref{section:proofs}, we will use the results from \S \ref{ssec: relation between selmer space and group} and \S \ref{subsection:equidistribution} to relate the random kernel model to the distribution of Selmer groups. For the rest of this section, we focus on analyzing the random kernel model. 

\subsection{Distribution of random 1-eigenspaces}\label{subsec: dist 1-eigenspaces}

We now focus on the case where $n = \ell$ is prime. 

\subsubsection{Some notation}

We will use \autoref{theorem:fulman-stanton} in conjunction with \autoref{lemma:same-orbits} to deduce the probability generating function for $\ker(g-\id)$ for $g$ drawn uniformly at random from a coset of $\Omega(\qsel \ell d k) \subset \o(\qsel \ell d k)$. Now we will take $H \subset \o(\qsel \ell d k)$ to be a coset of $\Omega(\qsel \ell d k)$ in $\o(\qsel \ell d k)$.
\begin{itemize}
	\item Note that when $\ell = 2$, the spinor norm is trivial on $\o(\qsel 2 d k)$ and hence $\Omega(\qsel 2 d k) = \so(\qsel 2 d k)$ and there are two
possibilities for the coset $H$, determined by the Dickson invariant.

\item When $\ell$ is odd, there are four cosets of $\Omega(\qsel \ell d k)$ given by the pair $(\sp {\qsel \ell d k}, \dickson {\qsel \ell d k})$. We label these cosets as in the following table.
\begin{center}
\begin{tabular}{|c|cc|}
\hline 
\diagbox{$\dickson {\qsel \ell d k}$}{$\sp {\qsel \ell d k}$}
       & trivial & non-trivial \\ 
       \hline
trivial  & $\Omega$ & $A$\\
non-trivial  & $B$  & $C$ \\
\hline 
\end{tabular}
\end{center}
\end{itemize}

For $Z$ a random variable valued in isomorphism classes of finite-dimensional $\F_{\ell}$-vector spaces, define the {\em probability generating function} of $Z$ to be the polynomial in $t$ given by
\[
G_Z(t) := \mathbb E(t^{\dim Z}) = \sum_{i \in \N} \prob(\dim Z = i) t^i.
\]
For a polynomial $f(t) = \sum_{i \in \N} a_i t^i$, introduce the notation $[f(t)]_r := a_r$ to denote the coefficient of $t^r$ in $f(t)$.

\subsubsection{The probability generating functions} We will now work towards the proof of:

\begin{theorem}\label{theorem: coset generating functions}
Let $\ell>2$ be an odd prime and $d \geq 1$ a positive integer. Then we have $G_{\rvcoset {V_{\ell}^d} B} = G_{\rvcoset {V_{\ell}^d} C}$ and 
\[
	G_{\rvcoset {V_{\ell}^d} \Omega} = G_{\rvcoset {V_{\ell}^d} A}  + \frac{1}{\# \Omega(\qsel \ell d k)}\prod_{i=0}^{6d-3} \left( t^2-\ell^{2i} \right).
\]
\end{theorem}

\subsubsection{Some lemmas} We begin with some preliminary results. For $(V, Q)$ a quadratic space and $k \in \Z_{\geq 0}$, we will abbreviate 
\[
V^{k} := \underbrace{V \times V \times \cdots \times V}_{s \text{ times}}
\]
and consider the diagonal action of $\o(Q)$ on $V^{k}$. This induces a diagonal action of the subgroup $\Omega(Q) \subset \o(Q)$ on $V^k$.

\begin{lemma}\label{lemma:same-orbits}
Let $m \in \bz_{\geq 0}$ and let $(V, Q)$ be a nondegenerate quadratic space over a finite field $L$ with $\dim_L V  = r$.
If $r \geq 2m+2$, then the orbits of $\o(Q)$ and $\Omega(Q)$ on $V^m$ coincide. Hence, the orbits of $\o(Q)$ on $V^m$ agree with the orbits of any subgroup $H \supset \Omega(Q)$ on $V^m$. 
\end{lemma}

\begin{proof}
	It suffices to show that $\Omega( Q)$ acts transitively on any orbit of $\o(Q)$. 
Fix an arbitrary tuple of vectors $(v_1, \ldots, v_m) \in V^m$. Let $W := \Span(v_1, \ldots, v_m)$.
We claim that if $\dim_L V \geq 2m + 2$, 
for every $a \in L,$ there is some $w \in W^\perp$ with $Q(w) = a$.

Assuming this claim, let us show that the orbits of $\o(Q)$ and $\Omega(Q)$ coincide.
First, we tackle the case $\chr(L) \neq 2$. 
In this case, it suffices to show that for each $(\alpha, \beta) \in \bz/2\bz \times \bz/2\bz$,
there is some $h \in \o(Q)$ fixing $(v_1, \ldots, v_m)$ with $\sp{Q}(h) = \alpha$ and $\det(h)= \beta$.
To see such an $h$ exists, let $w$ be an element in $W^\perp$ with $-Q(w)$ a square in $L$, and let $w'$ be an element with $-Q(w')$ a non-square in $L$.
Then the four elements $\id, r_w, r_{w'},r_w \circ r_{w'} \in \o(Q)$ attain all four possible values of $(\sp{Q}, \det)$ and fix $(v_1, \ldots, v_m)$.
This implies that $\Omega(Q)$ acts transitively on the $\o(Q)$-orbit of $(v_1, \ldots, v_m)$. 

The case $\chr(L) = 2$ is similar, but easier. 
To show $\Omega(Q)$ has the same orbits as $\o(Q)$, it suffices to exhibit an element of nontrivial Dickson invariant 
fixing $(v_1,\ldots, v_m)$.
Indeed, for any $v \in W^\perp$, $r_v$ is such an element.

We now conclude the proof by verifying the claim. If $(V, Q)$ is any nondegenerate quadratic space of dimension at least $2$ over a finite field $L$, then for every $a \in L$ there is some $v \in V$ with $Q(v) = a$.
Recall that the {\em rank} of a quadratic space $(V, Q)$ is defined to be $\rk(V, Q) := \dim V - \dim \rad(V, Q)$, where $\rad(V, Q)$ the {\em radical} of $(V, Q)$, i.e., the set of  $x \in V$ with $B_Q(x,y) = 0$ for all $y \in V$.
Therefore, it suffices to show that $\rk(Q|_{W^{\perp}}, W^\perp) \geq 2$. Note that $\rad(Q|_{W^\perp}, W^\perp)  = W \cap W^{\perp}$. Hence 
\begin{equation}\label{eq:rank equation}
	\rk(Q|_{W^{\perp}}, W^\perp)=   \dim W^\perp - \dim (W \cap W^{\perp}) .
\end{equation}
Since $\dim V \geq 2\dim W +2$, we have $\dim W^\perp - \dim (W \cap W^{\perp})  \geq \dim W^\perp - \dim W \geq 2$. 
\end{proof}

It will also be useful later to have a result on the case when $\dim V = 2m$. 

\begin{lemma}\label{lem: orbits edge case}
Let $(V, Q)$ be a nondegenerate quadratic space over a finite field $L$ with $\dim_L V  = r$. If $r = 2m$ is even, then the orbits of $\o(Q)$ and $\SO(Q)$ on $V^m$ agree except on $m$-tuples $(v_1, \ldots, v_m) \in V^m$ that span a maximal isotropic subspace of $V$. 
\end{lemma}

\begin{proof}
It suffices to exhibit an element of $\o(Q) - \SO(Q)$ that stabilizes $(v_1, \ldots, v_m)$. Let $W := \Span(v_1, \ldots, v_m)$ as in the proof of \autoref{lemma:same-orbits}. If we can find $w \in W^{\perp}$ such that $Q(w) \neq 0$, then $r_w$ does the job. 

To see that such $w$ exists, it suffices to show that $\rk(Q|_{W^{\perp}}, W^{\perp}) > 0$. But by \eqref{eq:rank equation}, this holds as long as $W$ is not maximal isotropic. 
\end{proof}

\begin{lemma}\label{lem: agree at positive points}
For $\ell$ a prime and $d \geq 1$, any coset $H \subset \o(\qsel \ell d i)$ of $\Omega(\qsel \ell d i)$, we have 
\[
G_{\rvcoset {V_{\ell}^d} H}(\ell^i) = G_{\rvcoset {V_{\ell}^d} \Omega}(\ell^i) \quad \text{ for } i = 0, 1, \ldots, 6d-3.
\]
\end{lemma}

\begin{proof}
For $g \in G$, let $V^{g=1}$ denote the $1$-eigenspace of $g$ acting on $V$.
Let $G' \subset G$ be a subgroup. 
By definition, we have 
\[
	G_{\rvcoset {V_{\ell}^d} {G'}} (t) = \frac{1}{\# G'}  \sum_{g \in G'} t^{\dim \ker(g-\id)} 
\]
so that 
\begin{equation}\label{eq: pgf 1}
	G_{\rvcoset {V_{\ell}^d} {G'}} (\ell^i) =  \frac{1}{\# G'}  \sum_{g \in G'} (\# V^{g=1})^i.
\end{equation}
Note that  $   (V^{g=1})^i = (V^i)^{g=1} $ where $g \in G$ acts diagonally on $V^i$, so that $(\# V^{g=1})^i = \# (V^i)^{g=1}$.  Putting this into \eqref{eq: pgf 1} gives 
\begin{equation}\label{eq: pgf 2}
G_{\rvcoset {V_{\ell}^d} {G'}} (\ell^i) =  \frac{1}{\# G'}  \sum_{g \in G'} (\# V^i)^{g=1}.
\end{equation}

By Burnside's Lemma, we have
\begin{equation}\label{eq: burnside}
\sum_{g \in G'} \# (V^i)^{g=1} = \# \{\text{orbits of $G'$ on $V^i$}\}. 
\end{equation}
By \autoref{lemma:same-orbits}, the right hand side of \eqref{eq: burnside} has the same value when we take $G'$ to be any of $\Omega(\qsel \ell d k)$, $\ker (\sp {\qsel \ell d k})$, $\ker( \dickson {\qsel \ell d k})$, and $\o(\qsel \ell d k)$ for $i \leq 6d-3$. Hence we have 
\[
\Scale[0.9]{G_{\rvcoset {V_{\ell}^d} \Omega}(\ell^i) = G_{\rvcoset {V_{\ell}^d} {\mrm{O}_-^*(V_{\ell}^d)}}(\ell^i) = G_{\rvcoset {V_{\ell}^d} {\mrm{SO}(V_{\ell}^d)}}(\ell^i) = G_{\rvcoset {V_{\ell}^d} {\mrm{O}(V_{\ell}^d)}}(\ell^i),  \quad i = 1 , \ldots, 6d-3.}
\]
We then obtain the result by noting that any coset can be expressed in terms of differences of the above subgroups.
For example, we can obtain the result for $H = B$ by writing
\[
G_{\rvcoset {V_{\ell}^d} {\mrm{SO}(V_{\ell}^d)}}(\ell^i) = \frac{1}{2} G_{\rvcoset {V_{\ell}^d} \Omega}(\ell^i)  + \frac{1}{2} G_{\rvcoset {V_{\ell}^d} B}(\ell^i).
\]
\end{proof}

\begin{proof}[Proof of \autoref{theorem: coset generating functions}]
Recall that the Dickson invariant of any element $g \in \o(\qsel n d k)$ agrees with the dimension of its $1$-eigenspace $\mod 2$.	Indeed, in general, the Dickson invariant of $g$ agrees with $\dim\im(1-g)$, 
	by
	\cite[p. 160]{taylor:the-geometry-of-the-classical-groups}, where the notation $[V,f]$ is used for $\im (1-f)$.
	Since $\dim \vsel n d k$ is even, it follows that $\dim\ker(1-g) \equiv \dim \im(1-g) \mod 2$.

Because of this, only \emph{odd} powers of $t$ can appear in $G_{\rvcoset {V_{\ell}^d} B}(t)$ and $G_{\rvcoset {V_{\ell}^d} C}(t)$. Furthermore, they have degree at most $12d-5$ since $\dim V = 12d-4$. 
By \autoref{lem: agree at positive points}, these functions agree at the $6d-2$ points $1, \ell, \ldots, \ell^{6d-3}$. Since they are both odd functions, they must agree as well at $0, -1, -\ell, \ldots, -\ell^{6d-3}$. But two polynomials of degree at most $12d-5$ agreeing at $12d-3$ points must be the same. 

Similarly, $G_{\rvcoset {V_{\ell}^d} \Omega}(t)$ and $G_{\rvcoset {V_{\ell}^d} A}$ are \emph{even} polynomials of degree at most $12d-4$, and they agree at the $12d-4$ points $\pm 1, \pm \ell, \ldots, \pm \ell^{6d-3}$. The difference $G_{\rvcoset {V_{\ell}^d} \Omega}(t)-G_{\rvcoset {V_{\ell}^d} A}(t)$ must therefore be proportional to $\prod_{i=1}^{6d-3}(t^2-\ell^{2i})$. To find the constant of proportionality, note that the coefficient of $t^{12d-4}$ in $G_{\rvcoset {V_{\ell}^d} H}(t)$ is the probability that $g \in H$ fixes all of $V$, i.e. is the identity. This happens with probability $\frac{1}{\# \Omega(\qsel \ell d k)}$ for $H = \Omega(\qsel \ell d k)$, and probability $0$ for any other coset. This completes the proof. 
\end{proof}

\subsubsection{Formulas for the generating functions}\label{sssec: generating function formulas}

Let 
$\o(12d-4, \mathbb F_\ell)$ denote the orthogonal group
associated to the standard quadratic form $\sum_{i=1}^{6d-2} x_{2i-1}x_{2i}$ on a $12d-4$ dimensional vector space over $\F_{\ell}$.  

\begin{lemma}\label{lem: same orthogonal group}
The group $\o(12d-4, \mathbb F_\ell)$ is isomorphic to $\o(\qsel \ell d k)$. 
\end{lemma}

\begin{proof}
	We begin by showing the quadratic form $\qsel n d k$ has discriminant $1$ over $\bz/n \bz$.
	Indeed, it is the reduction $\mod n$ of a quadratic form
	$\qsel {\bz} d k$ over $\bz$ which has discriminant $1$ over $\bz$ by \cite[Theorem 4.4 and Remark 4.5]{landesman:geometric-average-selmer}.
	Indeed, \cite[Remark 4.5]{landesman:geometric-average-selmer} explains that
	$\qsel{\bz} d k = U^{\oplus (2d -2)} \bigoplus (-E_8)^{\oplus d}$, where $U$ denotes the hyperbolic plane and $-E_8$ denotes the quadratic form associated to the $E_8$ lattice with negative
	its usual pairing. Since $U$ has discriminant $-1$ while $-E_8$ has discriminant $1$, the discriminant of $\qsel {\bz} d k$ is $(-1)^{2d-2} \cdot 1^d = 1$.
	We deduce that, $\o(\qsel \ell d k) = \o(12d-4, \mathbb F_\ell)$ has rank $12d-4$ and discriminant $1$. 
	When $\ell > 2$, there is a unique orthogonal group over $\mathbb F_\ell$ of discriminant $1$
	\cite[3.4.6]{wilson:the-finite-simple-groups}, 
	and so $\o(\qsel \ell d k) \simeq \o(12d-4, \mathbb F_\ell)$ in this case.
	When $\ell = 2$, there are two nonisomorphic quadratic forms of discriminant $1$ and rank $12d-4$, but $\o(12d-4, \mathbb F_\ell)$ is the unique hyperbolic such
	quadratic form, so we only need check $\o(\qsel \ell d k)$ is hyperbolic. To this end, it suffices to check the quadratic form associated to $E_8$ is hyperbolic when reduced modulo $2$.
	A nondegenerate even dimensional quadratic form over a field is hyperbolic if and only if it contains an isotropic subspace of half the dimension of the quadratic space \cite[III, Lemma 1.2]{MH73}.
	For the $E_8$ lattice, one can explicitly construct such a subspace, such as the space spanned by the first, third, sixth and eighth basis vectors, when the $E_8$ lattice
	is written as in
	\cite[Chapter 14, 0.3(iii)]{huybrechts:lectures-on-k3-surfaces}.
	\end{proof}
	
By \autoref{lem: same orthogonal group}, the generating function $\rvo \ell d$ agrees with the generating function $\rvcoset {V_{\ell}^d} H$ from \autoref{definition:random 1-eigenspace} with $H = \o(12d-4, \mathbb F_\ell)$ the full orthogonal group, so we may use these notations interchangeably. The following theorem, which completely characterizes $\rvo \ell d$, is proved in an unpublished manuscript of Rudvalis-Shinoda, cf. \cite{fulmanS:distribution-number-fixed}. 
We will give an independent proof of this theorem in \autoref{sssec: rudnalis-shinoda proof}.

For $Z$ a random variable we let $\mathbb E (Z^m)$ denote the {\bf $m$th moment} of $Z$, which is the expected value of the random variable $Z^m$.

\begin{theorem}[Rudvalis-Shinoda, \protect{\cite[Theorem 2.5 and 4.7]{fulmanS:distribution-number-fixed}}]\label{thm: rudvalis-shinoda}
	\label{theorem:fulman-stanton} 	We have 	
	\[\Scale[0.8]{
	\begin{aligned}
		\label{equation:p-orthogonal-distribution}		
	\prob(\dim \rvo \ell d = v)=
	\begin{cases}
			\frac{\ell^z}{2 \left | \gl_z(\mathbb F_{\ell^2}) \right | }\sum_{i=0}^{6d-2-z} \frac{(-1)^i}{\ell^{(2z-1)i}\left( \ell^{2i}-1 \right)\cdots \left( \ell^4-1 \right)\left( \ell^2-1 \right)} \\
			\hspace{1cm} + \frac{1}{2} \frac{\left( -1 \right)^{6d-2-z}}{\ell^{2z(6d-2-z)} \left | \gl_z(\mathbb F_{\ell^2}) \right | \left( \ell^{2(6d-2-z)} - 1 \right) \cdots \left( \ell^4 - 1 \right)\left( \ell^2 -1 \right)}
				& \text{ if } v=2z \\
			\frac{1}{2 \ell^z\left | \gl_z(\mathbb F_{\ell^2}) \right | }\sum_{i=0}^{6d-2-z} \frac{(-1)^i}{\ell^{i^2 + 2(z+1)i}\left( 1-q^{-2} \right) \left( 1-q^{-4} \right) \cdots \left( 1-q^{-2i} \right)} 
			& \text{ if } v= 2z+1.\\
		\end{cases}
		\end{aligned}}
		\]
	Furthermore, we have
	\begin{equation}
		\label{equation:limit-orthogonal-distribution}\Scale[0.9]{
		\lim_{d \ra \infty} \left( \prob(\dim \rvo \ell d = v) \right)
		= \prod_{j \geq 0}\left( 1 + \ell^{-j} \right)^{-1} \frac{1}{\ell^{(v^2-v)/2}\left( 1-\ell^{-1} \right)\left( 1-\ell^{-2} \right) \cdots \left( 1-\ell^{-v} \right) }.}
	\end{equation}
	Additionally, for $0 \leq m \leq 6d-2$, the moments of $\# \rvo \ell d$ are computed as 
	\begin{align*}
		\mathbb E (\# \rvo \ell d)^m = \prod_{i=1}^m \left( \ell^i + 1 \right).
	\end{align*}
\end{theorem}

From \autoref{theorem:fulman-stanton} and \autoref{theorem: coset generating functions}, it is fairly straightforward to deduce explicit formulas for the probability generating functions $G_{\rvcoset {V_{\ell}^d} \Omega}(t), G_{\rvcoset {V_{\ell}^d} A}(t), G_{\rvcoset {V_{\ell}^d} B}(t), G_{\rvcoset {V_{\ell}^d} C}(t)$. However, we omit the answers as we will not need them.

\subsection{Direct computation of the moments}\label{ssec: moment computations}

In this subsection we give an alternate computation of the moments of $\dim \ker(g-\id)$ for $g \in \o(Q)$, for $Q$ a quadratic form over $\mathbb F_\ell$
of sufficiently large rank
without using the unpublished results of Rudvalis and Shinoda. We will explain that this gives an alternate proof of \autoref{thm: rudvalis-shinoda}. In addition, the analysis here is used later to get better control on the convergence of the random kernel model. 

As already mentioned above, \cite{fulmanS:distribution-number-fixed} computed an explicit formula for the moments of $\dim \ker(g-\id)$ for $g \in \o(Q)$, using the probability distribution obtained in unpublished work of Rudvalis-Shinoda. The calculation of Rudvalis-Shinoda rests on intricate combinatorial analysis. We learned of this work after we had already found an independent computation of the probability distribution, which we will explain in this subsection. Our logic in this subsection runs in the opposite direction: we directly compute the moments, and deduce the probability distribution from it. (The advantage of this approach is that it also gives the distribution for $g$ drawn from subgroups of $\o(Q)$, such as $\Omega$.)

\begin{theorem}\label{thm: orbit count} 
	Fix $m \in\Z_{\geq 0}$, let $n$ be squarefree, and let $(V,Q)$ be a nondegenerate quadratic space over $\bz/n\bz$. 
	For $\rk_{\bz/n\bz} V \geq 2m + 2$, then:
\begin{enumerate}
\item The number of orbits of $\o(Q)$ acting diagonally on $V^{m}$ is 
\begin{align}
	\label{equation:average-selmer-value}
\prod_{\ell \,  \mrm{ prime } \mid n} (1+\ell)(1+\ell^2) \cdots (1+\ell^m).
\end{align}

\item The orbits of $\Omega(Q)$ acting diagonally on $V^{m}$ coincide with those of $\o(Q)$ acting diagonally on $V^m$. 
\end{enumerate}

For the next part (which is about getting slightly sharper results in the ``edge case'' $r=2m$), we let $n = \ell$ be prime and ask that $(V,Q)$ be a split\footnote{For the definition of this, see \cite[I, \S 6]{MH73}.} quadratic space of dimension $r$ over $\F_{\ell}$.
\begin{enumerate}
	\item[(3)] For $r = 2m$, the number of orbits of $\o(Q)$ acting diagonally on $V^m$ is also given by \eqref{equation:average-selmer-value}.
\item[(4)] For $r=2m$, 
\[
 \#\{\text{orbits of $\SO(Q)$ on $V^m$}\}  =  \#\{\text{orbits of $\o(Q)$ on $V^m$}\}  +  1.
\]
\end{enumerate}
\end{theorem}

\subsubsection{Proof of \autoref{thm: rudvalis-shinoda}, assuming \autoref{thm: orbit count}}\label{sssec: rudnalis-shinoda proof}
Let $\wt{G}(t)$ be the generating function of the distribution in \autoref{thm: rudvalis-shinoda}. This is a polynomial of degree $12d-4$; write 
\[
\wt{G}(t) = \wt{G}^{\mrm{odd}}(t) + \wt{G}^{\mrm{even}}(t)
\]
where $\wt{G}^{\mrm{odd}}(t)$ is an odd polynomial and $\wt{G}^{\mrm{even}}(t)$ is an even polynomial. The computation in \cite{fulmanS:distribution-number-fixed} shows that the moments of the even and odd parts of the distributions coincide, so that 
\[
\wt{G}^{\mrm{odd}}(\ell^m) = \wt{G}^{\mrm{even}}(\ell^m), \quad 0 \leq m \leq 6d-3.
\]

As explained Lemma \ref{lem: agree at positive points}, the orbit counts in \autoref{thm: orbit count} are the moments of $\# \rvo \ell d$, so \autoref{thm: orbit count} shows that the $m$th moment of $\# \rvo \ell d$ is as claimed in \autoref{thm: rudvalis-shinoda} for $0 \leq m \leq 6d-3$. Writing
\[
G_{\rvo \ell d}(t) = G_{\rvo \ell d}^{\mrm{odd}}(t) + G_{\rvo \ell d}^{\mrm{even}}(t)
\]
for the decomposition into odd and even parts, \autoref{lem: agree at positive points} implies also that 
\[
G_{\rvo \ell d}^{\mrm{odd}}(\ell^m) = G_{\rvo \ell d}^{\mrm{even}}(\ell^m), \quad \text{ for } 0 \leq m \leq 6d-3.
\]

Hence $\wt{G}^{\mrm{odd}}(\ell^m) = G_{\rvo \ell d}^{\mrm{odd}}(\ell^m) $ for $0 \leq m \leq 6d-2$. Since they are both odd polynomials, they also agree at $-\ell^m$ for $ 0 \leq m \leq 6d-3$. But since they both have degree at most $12d-5$, and they agree at $12d-4$ points, they must be equal. 

Similarly, $\wt{G}^{\mrm{even}}(\ell^m) = G_{\rvo \ell d}^{\mrm{even}}(\ell^m) $ for $0 \leq m \leq 6d-2$. Since they are both odd polynomials, they also agree at $-\ell^m$ for $ 0 \leq m \leq 6d-2$. Hence there difference is a polynomial of degree at most $12d-4$ vanishing at the $12d-4$ points $\pm \ell^m$ for $0 \leq m \leq 6d-3$, and must therefore a multiple of $\prod_{m=0}^{6d-2} (t^2 - \ell^{2m})$. But the coefficients of $t^{12d-4}$ in both $\wt{G}^{\mrm{even}}(t)$ and $G_{\rvo \ell d}^{\mrm{even}}(t)$ are both $\frac{2}{\# \mrm{O}(12d-4, \F_{\ell})}$, so the constant of proportionality must be $0$. 
\qed

The rest of this subsection is devoted towards proving \autoref{thm: orbit count}.

\subsubsection{Counting orbits of independent vectors}

Recall that a quadratic space is \emph{hyperbolic} if it has the form $W\oplus W^\vee$ 
with form $Q(w,\lambda)=\lambda(w)$; over a field, this is equivalent to the condition that it be \emph{metabolic}, i.e., that it is nondegenerate and contains an isotropic subspace of half the dimension \cite[III, Lemma 1.2]{MH73}.

\begin{lemma} 
	\label{lemma:isometric-embedding}	
	Let $(V, Q)$ be a metabolic quadratic space over a field. Then any (possibly degenerate) quadratic space $(W, Q')$ of dimension $\dim(W)\le \dim(V)/2$ embeds isometrically in $V$.
\end{lemma}

\begin{proof}  If $\dim(W)<\dim(V)/2$, we can always enlarge W by taking the direct sum with a trivial quadratic space of dimension $\dim(V)/2 - \dim(W)$, so we may as well assume that $\dim(W)=\dim(V)/2$. Let $Q''$ be the quadratic form on $W\oplus W^*$ given by $Q''(w,\lambda) = Q'(w)+\lambda(w)$. Then $(W, Q')$ embeds isometrically in the metabolic (thus hyperbolic) quadratic space $(W\oplus W^*,Q'')$. Since two hyperbolic quadratic spaces of the same dimension are isomorphic, there is an isometry $(W\oplus W^*,Q'')\cong (V,Q)$, and thus $(W,Q')$ embeds in $(V,Q)$ as required.
\end{proof}

\begin{corollary}\label{cor: can find subspace} Let $(V,Q)$ be a nondegenerate quadratic space over a finite field. Then any (possibly degenerate) quadratic space $(W,Q')$ of dimension $\dim(W)\le (\dim(V)-2)/2$ 
embeds isometrically in $(V,Q)$.\end{corollary}
\begin{proof}Any nondegenerate quadratic space over a finite field is isomorphic to the direct sum of a hyperbolic quadratic space and a nondegenerate quadratic space of dimension at most $2$, and \autoref{lemma:isometric-embedding} shows that $(W,Q')$ embeds in the former.
\end{proof}

The key technical ingredient in the proof of \autoref{thm: orbit count} is the following Proposition. 

\begin{proposition}\label{proposition:independent-orbits} 
	Fix $m \in\Z_{\geq 0}$ and let $(V, Q)$ be a nondegenerate quadratic space over $\F_\ell$ of dimension $r \geq 2m + 2$. Then, the number of orbits of $\o(Q)$ in $V^{ m}$ consisting of a tuple of independent vectors $(x_1, \ldots, x_m)$ is $\ell^{m(m+1)/2}$. More precisely, the orbits consisting of independent vectors are in bijection with $\F_{\ell}^{m(m+1)/2}$ via the map sending
\begin{equation}\label{eq: invariant map}
	(x_1, \ldots, x_m) \mapsto (Q(x_1), \ldots, Q(x_m), B_Q( x_i, x_j) \co {1 \leq i<j \leq m}).
\end{equation}
If $(V,Q)$ is metabolic, then the result still holds if $r = 2m$. 
\end{proposition}

\begin{proof}[Proof of \autoref{proposition:independent-orbits}]
First we argue that \eqref{eq: invariant map} is injective. If
$(x_1, \ldots, x_m)$ and $(x_1', \ldots, x_m')$ have the same image under \eqref{eq: invariant map}, $\spn ( x_1, \ldots, x_m )$ is
isomorphic as a quadratic subspace of $(V, Q)$ to $\spn (x_1', \ldots, x_m')$ by the map sending $x_i \mapsto x_i'$. Therefore,
by Witt's theorem 
\cite[I.4.1, p. 80]{chevalley:the-algebraic-theory-of-spinors},
there is an element of $\o(Q)$ sending $x_i \mapsto x_i'$.
Hence, if $(x_1, \ldots, x_m)$ and $(x_1', \ldots, x_m')$ have the same image under \eqref{eq: invariant map}, they lie in the same $\o(Q)$ orbit.

It remains to show that \eqref{eq: invariant map} is surjective. Suppose $(c_1, \ldots, c_m, c_{ij} \co 1 \leq i < j \leq m) \subset \F_{\ell}^{m(m+1)/2}$ are arbitrary. Let $(W,Q')$ be the quadratic space on basis vectors $(y_1, \ldots, y_m)$ with $Q'(y_i) = c_i$ and $B_{Q'}(y_i, y_j) = c_{ij}$. The surjectivity amounts to showing that we can find an embedding $(W,Q') \rightarrow (V,Q)$ which is an isometry onto its image. 
But this is exactly the content of \autoref{cor: can find subspace} if $r \geq 2m+2$, and \autoref{lemma:isometric-embedding} if $r \geq 2m$ and $(V,Q)$ is metabolic. 
\end{proof}

\subsubsection{Orbits of dependent vectors} 
We aim to explain how to determine the orbits of tuples of vectors that are linearly dependent inductively using \autoref{proposition:independent-orbits}. 
The following lemma is key to counting these dependent orbits.
\begin{lemma}
	\label{lemma:dependent-orbits}
	Let $(V, Q)$ be a nondegenerate quadratic space over $\F_\ell$ and let $\o(Q)$ act on $V^m$. Fix $(x_1, \ldots, x_{m-1}) \in V^{m-1}$ and let $W := \spn \left( x_1, \ldots, x_{m-1} \right)$. The number of orbits of vectors of the form $(x_1, \ldots, x_{m-1}, y) \in V^m$ under the action of $\o(Q)$ with $y \in \Span(x_1, \ldots, x_{m-1})$ is $\ell^{\dim W}$. 
\end{lemma}
\begin{proof}
	Suppose that $(x_{i_1}, \ldots, x_{i_t})$ is a basis for $W$, so $\dim W = t$. Then for any $g \in \o(Q)$, $g\cdot(x_1, \ldots, x_{m-1}, y)$ is uniquely determined by $g \cdot (x_{i_1}, \ldots, x_{i_t})$. 

To count the number of orbits, we can express $y$ uniquely as 
\[
y = \sum_{j=1}^t a_j x_{i_j}.
\]
Then the orbit of $(x_1, \ldots, x_{m-1}, y)$
is uniquely determined by the scalars $(a_i \in \F_{\ell})_{1 \leq i \leq t}$, and so there are $\ell^{\dim W}$ such orbits.
\end{proof}

\subsubsection{A recursive formula}

\begin{definition}
	\label{definition:}
	Fix a quadratic space $(V,Q)$ over a finite field $k$.
	Let $f(n,i)$ be the number of orbits of $V^n$ under the action of $\o(Q)$ such that $\dim_k \mrm{Span}(x_1, \ldots, x_n) = i$. 
\end{definition}

We next explain a recursive formula for the $f(n,i)$.
\begin{lemma}
	\label{lemma:recursion}
	The functions $f(n,i)$ satisfy the recursion
	\begin{equation}\label{eq: recursion}
		f(n,i) = f(n-1, i-1) \ell^i
		+ f(n-1, i) \ell^i.
\end{equation}
\end{lemma}
\begin{proof}
	Fix a tuple $(x_1, \ldots, x_{n-1}) \in V^{n-1}$. We will count the number of orbits of the form $(x_1, \ldots, x_{n-1}, y) \in V^{n}$, by conditioning on whether or not $y \in \spn\left( x_1, \ldots, x_{n-1} \right)$.
	\begin{itemize}
	\item If $y \in \spn\left( x_1, \ldots, x_{n-1} \right)$, each choice of $y$ yields a different orbit and there are $\ell^i$ possible such orbits by \autoref{lemma:dependent-orbits}.
	\item 
	If $y \notin \spn\left( x_1, \ldots, x_{n-1} \right)$, let $\left( x_{s_1}, \ldots, x_{s_{i-1}} \right)$ be a basis for $\spn\left( x_1, \ldots, x_{n-1} \right)$.
	\autoref{proposition:independent-orbits} shows that there are $\ell^{i(i+1)/2 - (i-1)i/2} = \ell^i$ orbits of the form $(x_1, \ldots, x_{n-1}, y)$, parameterized by the possible values of the pairings	
	\[
	B_Q(y, x_{s_1}), \ldots, B_Q(y, x_{s_{i-1}}) , Q(y, y).
	\]
	\end{itemize}
	Adding these two contributions over varying vectors $\left( x_1, \ldots, x_{n-1} \right) \in V^{n-1}$ yields the result.
\end{proof}

\begin{remark}
	\label{remark:}
	We have the initial condition $f(0,i)=1$ for all $i \geq 0$. This together with the recursion of \autoref{lemma:recursion} determine the $f(n,i)$ uniquely.
	We extend $f(n,i)$ by $0$ to a function on $\Z \times \Z$.
\end{remark}

\begin{definition}
	\label{definition:}
	For every $j \in \Z_{\geq 0}$, 
	define
\[
\Sigma^{(s)}(m) := \sum_{i \in \Z} f(m,i) \ell^{is}.
\]
\end{definition}
\begin{remark}
	\label{remark: number of orbits}
	From the definitions, it follows that the total number of orbits of $\o(Q)$ on $V^{ m}$ is 
$\Sigma^{(0)}(m) = \sum_{i \in \Z} f(m,i).$
Also observe that for any $j$, $\Sigma^{(j)}(0) = 1$ by definition, since $f(0,i) = 0$ unless $i = 0$.
\end{remark}
By \autoref{remark: number of orbits}, we want to calculate $\Sigma^{(0)}(m)$. The following lemma relates this to
$\Sigma^{(m)}(0)$.

\begin{lemma}\label{lem: sigma recursion}
For $m > 0$ and $s \geq 0$, We have 
\[
\Sigma^{(s)}(m) = (1+\ell^{s+1}) \Sigma^{(s+1)}(m-1).
\]
\end{lemma}
\begin{proof}
	By \autoref{lemma:recursion}, we have 
\begin{align*}
\Sigma^{(s)}(m) &=  \sum_{i \in \Z} f(m-1, i-1) \ell^{i+is} + \sum_{i \in \Z} f(m-1, i) \ell^{i+is} \\
&=\ell^{s+1} \sum_{i \in \Z} f(m-1, i-1) \ell^{(i-1)(s+1)} + \sum_{i \in \Z} f(m-1, i) \ell^{i(s+1)}\\
&=\ell^{s+1} \sum_{i \in \Z} f(m-1, i) \ell^{i(s+1)} + \sum_{i \in \Z} f(m-1, i) \ell^{i(s+1)}\\
& = (\ell^{s+1}+1) \Sigma^{(s+1)}(m-1).\qedhere
\end{align*}
\end{proof}

Using \autoref{lem: sigma recursion}, we can compute $\Sigma^{(0)}(m)$, and hence prove \autoref{thm: orbit count}.

\subsubsection{Proof of \autoref{thm: orbit count}}
\label{subsection:proof-of-orbit-count}
First we focus on the situation in parts (1) and (2), where $\rank_{\Z/n\Z} V \geq 2m+2$. Since $n$ is squarefree, we may reduce to the case $n = \ell$ is a prime by the Chinese remainder theorem.
	Once the statement for $\o(Q)$ is established, the statement for $\Omega(Q)$ follows from \autoref{lemma:same-orbits}.
By \autoref{remark: number of orbits}, we just need to show that
\[
\Sigma^{(0)}(m) = (1+\ell)(1+\ell^2) \cdots  (1+\ell^m).
\]
Indeed, using \autoref{lem: sigma recursion}, we find
\begin{align*}
\pushQED{\qed} 
	\Sigma^{(0)}(m) &= (1+\ell) \Sigma^{(1)}(m-1) \\
	&=  (1+\ell)(1+\ell^2) \Sigma^{(2)}(m-2)  \\
	&\hspace{.2cm}\vdots  \\
	&= (1+\ell)(1+\ell^2) \cdots  (1+\ell^m)\Sigma^{(m)}(0) \\
	&= (1+\ell)(1+\ell^2) \cdots  (1+\ell^m). \qedhere \popQED
\end{align*}
This completes the proof of parts (1) and (2). Now we move onto parts (3) and (4). The argument for part (3) is the same as for the proof of \autoref{thm: orbit count}. For Part (4), we note by \autoref{lem: orbits edge case} that the orbits coincide except on vectors $(x_1, \ldots, x_m) \in V^m$ that span a maximal isotropic subspace of $V$. In this case there is only one orbit of such vectors under $\o(Q)$, but two orbits under $\SO(Q)$ 
\cite[Corollary T.3.4]{conrad:algebraic-groups-ii}.

\subsection{Bounding the TV distance}
\label{subsection:tv-distance}

We use the moment computations in \autoref{ssec: moment computations} to obtain certain useful expressions for the probability generating functions. 

In this section, let $(V_{r}, Q_{r})$ be the \emph{split} orthogonal space over $\F_{\ell}$ of rank $r$ (hence discriminant $1$). We denote $\o_{r} = \o(V_r, Q_r)$, $\SO_r = \SO(V_r, Q_r)$, $\Omega_r = \Omega(V_r, Q_r)$, etc. 

Let $H_{2r} \subset \o_{2r}$ denote the kernel of the Dickson invariant, i.e., $H_{2r} = \SO_{2r}$ when $\ell$ is odd, and $H_{2r} = \Omega_{2r} $ when $\ell$ is even. For $j \geq 0$, let $M_j$ be the limit as $r \rightarrow \infty$ of the $j$th moment of $\rvcoset {V_r} \SO$, which by \autoref{thm: rudvalis-shinoda} is $\prod_{i=1}^j (\ell^i+1)$. 

\begin{lemma}\label{lem: various moments}  We have the following values for the moments of $\# \ker(g-1)$ for $g$ drawn from $H_{2r}$ and its complement: 
\begin{align*}
\EE_{g\in H_{2r}}( \# \ker(g-1)^j) &= M_j, 0\le j<r \\
\EE_{g\in H_{2r}} (\#\ker(g-1)^r) &= M_r + 1\\
\EE_{g\notin H_{2r}} (\#\ker(g-1)^j )&= M_j, 0\le j<r \\
\EE_{g\notin H_{2r}} (\#\ker(g-1)^r ) &= M_r-1.
\end{align*}
\end{lemma}

\begin{proof}
The claims for $j<r$ follow from \autoref{lemma:same-orbits} plus \autoref{thm: orbit count}. The claims for $j=r$ follow from \autoref{lem: orbits edge case} plus \autoref{thm: orbit count}
\end{proof}

Let $P_r(t)$ be the unique even polynomial of degree $2r$ such that $P_r(\ell^j)=M_j$ for all $0\le j\le r$, and let $P_r'(t)$ be the unique odd polynomial of degree $2r-1$ such that $P_r'(\ell^j)=M_j$ for $0\le j<r$ (not to be confused with the derivative of $P_r$). 

Define
\[
G_r(t)  := \EE_{g\in H_{2r}}[ t^{\dim\ker(g-1)}]
\]
to be the probability generating function for 1-eigenspaces of elements drawn randomly from $H_{2r}$, and 
\[
G_r'(t) := \EE_{g\in \o_{2r} - H_{2r}}[ t^{\dim\ker(g-1)}].
\]

\begin{lemma}\label{lem: G identities}
We have identities
\begin{equation}\label{eq: G r-1}
G_r(t) = P_{r-1}(t) + \frac{1}{\#H_{2r}} \prod_{0\le j<r} (t^2-\ell^{2j}),
\end{equation}
\begin{equation}\label{eq: G r}
G_r(t)  = 
P_r(t) + \prod_{0\le j<r} \frac{t^2-\ell^{2j}}{\ell^{2r}-\ell^{2j}},
\end{equation}
\begin{equation}\label{eq: G prime r-1}
	G'_{r+1}(t) = P_r'(t) + \ell^{-r}t\prod_{0\le j<r}\frac{t^2-\ell^{2j}}{\ell^{2r}-\ell^{2j}},
\end{equation}
\begin{equation}\label{eq: G prime r}
	G_{r+1}'(t)  = P_{r+1}'(t).
\end{equation}

\end{lemma}

\begin{proof}
	First, we check \eqref{eq: G r-1}.
By \autoref{lem: various moments}, $G_r(t) - P_{r-1}(t)$ vanishes at $t =  \pm \ell^j$ for $0 \leq j \leq r-1$, and is of degree $2r$, hence is proportional to $\prod_{0\le j<r} (t^2-\ell^{2j})$. Therefore, we can determine $G_r(t) $ completely by examining the coefficient of $t^{2r}$, which is $\# H_{2r}^{-1}$ because that is the probability of drawing the identity element. 

We next check \eqref{eq: G r}
Similarly, $G_r(t)  - P_r(t)$ is proportional to $\prod_{0\le j< r} (t^2-\ell^{2j})$, and it can be determined by evaluating at $\ell^r$, where the value is $1$
by \autoref{lem: various moments}.

Next, \eqref{eq: G prime r} holds because both $G'_{r+1}(t)$ and $P'_{r+1}(t)$ are polynomials of degree $2r+1$ vanishing at the $2r+3$ values $0, \pm 1, \pm \ell, \ldots, \pm \ell^r$.

Finally, we show \eqref{eq: G prime r-1}.
By \eqref{eq: G prime r} and \autoref{lem: various moments}, we see $P'_{r}(\ell^r) = M_r - 1$ while $G'_{r+1}(\ell^r) = M_r$.
Therefore, $G'_{r+1}(t) - P_r'(t)$ is a degree $2r+1$ polynomial vanishing at the $2r+1$ values $0, \pm 1, \pm \ell, \ldots \pm \ell^r$, and hence is determined up to a constant.
We can then determine its constant value by plugging in $t = \ell^r$, using $P'_{r}(\ell^r) = M_r - 1$ and $G'_{r+1}(\ell^r) = M_r$.
\end{proof}

\label{subsection:tv}
 Recall that the \emph{Total Variation distance} (TV) between two probability distributions $P$ and $P'$ is 
\[
\dTV(P,P') = \sup_{\text{events}\,A} |P(A) - P'(A)|.
\]
When $P$ and $P'$ are defined on a countable discrete probability space $X$, 
as shown in \cite[Proposition 4.2]{levinPW:markov-chains-and-mixing-times}
we can write this as 
\begin{equation}
\dTV(P, P') := \frac{1}{2} \sum_{x \in X} |P(x)-P'(x)|.
	\label{equation:tv-sum}
\end{equation}
\[
\]
In other words, conflating $P$ and $P'$ with functions on $X$, this is (up to the normalization factor $1/2$) the $L^1$-norm. Clearly, convergence in TV distance implies convergence as distributions (which is pointwise convergence in the case of distributions on a discrete space). We define the TV distance between two random variables to be the TV distance between their induced probability distributions. 

\begin{theorem}
	\label{theorem:p-selmer-distribution}
	For $\ell$ a prime, $d \geq 2$, and $q$ ranging over prime powers with $\gcd(q,2\ell)  =1$ We have 
	\[
	\limsup_{\substack{q \rightarrow \infty\\ \gcd(q,2n)=1}}
\dTV( \dim \rvker \ell d {\mathbb F_q} , \lim_{d \rightarrow \infty} \dim \rvo \ell d) = 	O(\ell^{-(6d-2)^2}).
\]
where the implicit constants are absolute in both cases.
\end{theorem}

\begin{proof}
We write the proof in the case where $\ell $ is odd; the case where $\ell = 2$ is even easier, as the analysis of the cosets simplifies because there are fewer cosets (cf. the discussion in \autoref{sssec: generating function formulas}).

We first compare the TV distance between $\dim \rvo \ell d $ and $\dim \rvker \ell d {\mathbb F_q}$. We have 
\[
 G_{\rvo \ell d}(t) = \frac{1}{4} G_{\rvcoset {V_{\ell}^d} \Omega} (t)+ \frac{1}{4} G_{\rvcoset {V_{\ell}^d} A}(t)  + \frac{1}{4} G_{\rvcoset {V_{\ell}^d} B} (t)+  \frac{1}{4}G_{\rvcoset {V_{\ell}^d} C}(t)
 \]
 and
 \[
G_{\rvker \ell d {\mathbb F_q}}(t) = \frac{1}{2} G_{\rvcoset {V_{\ell}^d} \Omega} (t) + \frac{1}{2} G_{\rvcoset {V_{\ell}^d} B}(t) \text{ or } \frac{1}{2} G_{\rvcoset {V_{\ell}^d} A}(t)  + \frac{1}{2} G_{\rvcoset {V_{\ell}^d} C}(t). 
\]
Note that the TV distance between random variables $Z$ and $Z'$ has a clean formulation in terms of the probability generating functions $G_Z(t)$ and $G_Z(t')$: it is half the sum of the absolute values of the differences of the coefficients, as follows from
\eqref{equation:tv-sum}.
Using this observation together with
\autoref{theorem: coset generating functions}, 
we have
\begin{align*}
\dTV(\dim \rvo \ell d, \dim \rvker \ell d {\mathbb F_q}) & \leq \frac{1}{4} \dTV(\dim \rvcoset {V_{\ell}^d} \Omega, \dim \rvcoset {V_{\ell}^d} A ) \\
& = \frac{1}{8} \cdot \frac{1}{\# \Omega(Q_{\ell}^d)} \prod_{i=0}^{6d-3}(1+\ell^{2i}).
\end{align*}
By examining the dimension of the orthogonal group, we find 
\[
	\# \Omega(Q_{\ell}^d) = \frac{1}{4} \#  O(Q_{\ell}^d) \asymp \ell^{(12d-4)(12d-5)/2}.
\]
On the other hand, we have
\[
 \prod_{i=0}^{6d-3}(1+\ell^{2i}) \asymp \ell^{(6d-2)(6d-3)}.
\]
Hence\footnote{The notation $A(d) \ll B(d)$ means $A(d) = O(B(d))$ as $d \rightarrow \infty$, where the implicit constant is absolute.}	 
\[
	\dTV(\dim \rvo \ell d, \dim \rvker \ell d {\mathbb F_q})  \ll  \ell^{-(6d-2)^2}.
\]

Next, we estimate $\dTV(\rvo \ell d, \lim_{r \ra \infty} \rvo \ell r)$. It suffices to show that 
\[
\dTV (\dim \rvcoset {V_{\ell}^{2r}} {\o_{2r}}, \dim \rvcoset {V_{\ell}^{2r+2}} {\o_{2r+2}} )  \ll \ell^{-r^2}.
\]
We compare the even and odd parts of their generating functions, using the computations of the preceding section. For the even part, using \autoref{lem: G identities} gives that the sum of the absolute values of the coefficients of $G_{r}(t) - G_{r-1}(t)$ is
\[
	\ll \ell^{-r}  \prod_{0 \leq j< r} \frac{1+\ell^{2j}}{\ell^{2r}-\ell^{2j}}
= 
\ell^{-r} \ell^{-r^2 + r}\prod_{0 \leq j < r}\frac{1+\ell^{-2j}}{1-\ell^{2j-2r}} \ll \ell^{-r^2}.
\]
This shows
\begin{align*}
	\limsup_{\substack{q \rightarrow \infty}}  \dTV( \dim \rvker \ell d {\mathbb F_q} ,  \lim_{d \rightarrow \infty} \rvo \ell d) = O(\ell^{-(6d-2)^2}).
\end{align*}

\end{proof}
\begin{corollary}\label{cor: TV estimate for prime case}
	Fix a prime $\ell$, an integer $d \geq 2$,
	and consider a sequence of prime powers $\{q_1, q_2, \ldots\}$ with
	$\gcd(q_i,2\ell)=1$, so that the $q_i$ lie in a fixed residue class mod $\ell$ if $\ell$ is odd, and lie in a fixed residue class mod $8$ if $\ell = 2$.
	Then, the TV distance between the BKLPR heuristic and
	$\lim_{i \to \infty} \dim \rvker \ell d {\mathbb F_{q_i}}$ is $O(\ell^{-(6d-2)^2})$.
\end{corollary}

\begin{proof}
First, we impose the assumption that the
the $q_i$ lie in a fixed residue class mod $\ell$ if $\ell$ is odd, and lie in a fixed residue class mod $8$ if $\ell = 2$, so that the distribution in \autoref{theorem:monodromy} is independent of the choice of $q_i$ in this sequence, since $\im \chi^{d-1}$ is independent of the choice of $q_i$.
Hence, $\lim_{i \to \infty} \dim \rvker \ell d {\mathbb F_{q_i}}$ exists.

Note that in the case where $\ell$ is prime, which we are currently considering, the ``BKLPR heuristic'' first appeared as the ``Poonen-Rains heuristic'' \cite{poonenR:random-maximal-isotropic-subspaces-and-selmer-groups}, whose explicit formula is given by \cite[Conjecture 1.1(a)]{poonenR:random-maximal-isotropic-subspaces-and-selmer-groups}. By inspection, this agrees with the distribution of $\lim_{d \rightarrow \infty} \rvo \ell d$ calculated in \autoref{thm: rudvalis-shinoda}. 
Hence the result follows from \autoref{theorem:p-selmer-distribution}.
\end{proof}

\section{Markov properties}
\label{section:markov}

In this section, we establish Markov properties satisfied by both the random kernel model and the BKLPR model, which will be used to identify their distributions for prime power order Selmer groups.
In \autoref{ssec: markov 1-eigenspace} we state the Markov property satisfied by the random kernel model, which we prove in
\autoref{subsection:proof-of-random-kernel-markov}.
We then recall the BKLPR model in \autoref{ssec: BKLPR}
and demonstrate the Markov property satisfied by the BKLPR model in
\autoref{subsection:markov-bklpr}.

\subsection{Markov property for random $1$-eigenspaces}\label{ssec: markov 1-eigenspace}

Let $(V,Q)$ be a nondegenerate quadratic space of rank $rm$ over $\Z/\ell^e \Z$. Recalling from \autoref{definition:random 1-eigenspace}, that for a subset $H \subset \o(V,Q)$ we let $\rvcoset V H$ be the random variable $\ker (g-\id)$, valued in isomorphism classes of finite abelian $\ell$-groups, for $g$ drawn uniformly at random from $H$. 
 
In this section only, we will use the notation $\o(V,Q), \Omega(V,Q),$ and $\so(V,Q)$ for various subgroups of orthogonal groups, because we will consider various coefficient changes and wish to emphasize this in the notation. Noting that $H$ acts on $V[\ell^j]$, we let $H_{j}$ be the image of $H$ in $\o(V[\ell^j], Q|_{V[\ell^j]})$.

\begin{theorem}\label{thm: 1-eigenspace markov}
Let $(V,Q)$ be a nondegenerate quadratic space of rank $2m$ over $\Z/\ell^e \Z$. For $j \leq e$, write $d_j(H) := \dim_{\F_{\ell}}( \ell^{j-1} \rvcoset  {V[\ell^{j}]} {H_{j}})$. 

If $H$ is a non-empty union of cosets of $\Omega(V,Q)$ in $\o(V,Q)$, then the sequence of random variables $d_1(H), d_2(H), \ldots, d_e(H)$ is Markov. 
If $\ell$ is odd or $d_i \ne 
2m$, then the distribution of $d_{i+1}(H)$ given $d_i(H)$ is the same as 
the dimension of the kernel of a uniform random alternating form on 
$\F_\ell^{d_i(H)}$. 
\end{theorem} 

\begin{corollary}\label{cor: 1-eigenspace markov}
For $n$ a prime power, $d \geq 2$ and $k$ a finite field,
the statement of \autoref{thm: 1-eigenspace markov} holds with $H := \im \mono n d k \cap \mult^{-1}(\mult \gamma_q)$.
\end{corollary}
\begin{proof} 
	By definition, $\im \mono n d k \cap \mult^{-1}(\mult \gamma_q)$ is a coset of the geometric monodromy group in the monodromy group. By \autoref{theorem:monodromy}, the geometric monodromy group contains $\Omega(\vsel n d k, \qsel n d k)$ and the monodromy group is contained in $\o(\vsel n d k, \qsel n d k)$. 
	Hence $(\im \mono n d k)^{\mult \gamma_q}$ is a union of cosets of $\Omega(\vsel n d k,\qsel n d k)$ in $\o(\vsel n d k, \qsel n d k)$, and we can apply \autoref{thm: 1-eigenspace markov} to each of the cosets.
\end{proof}
 
We next reduce \autoref{thm: 1-eigenspace markov} to \autoref{theorem: eigenspace markov} below.
For any $1 \leq j \leq e$, consider $\ell^{e-j} V = V[\ell^{j}]$, which is a nondegenerate quadratic space of rank $2m$ over $\Z/\ell^{j} \Z$. The action of $g \in \o(V,Q)$ on $V[\ell^{j}]$ factors through the quotient $\o(V,Q) \surj \o(V[\ell^{j}] , Q|_{V[\ell^{j}]})$. Let $H$ be any coset of $\Omega(V, Q)$. If $g$ is drawn uniformly at random in $\o(V,Q)$, its image in $\o(V[\ell^{j}], Q|_{V[\ell^{j}]})$ will also be uniform in a coset of $\Omega(V_{\Z/\ell^j \Z}, Q_{\Z/\ell^j \Z})$. 
We now naturally generalize \autoref{definition:random 1-eigenspace} to the setting of quadratic space over $\bz_\ell$.

\begin{definition}
	\label{definition:}
	Let $(V,Q)$ be a quadratic space over $\Z_\ell$, and let $H \subset \o(V,Q)$ be a subset which is a union of cosets of $\Omega(V,Q)$ in $\o(V,Q)$.
Define the random variable
$\rvcoset {V \otimes \Q_\ell/\Z_\ell} H$ to be given by 
$\ker (g-\id \mid_{V \otimes \Q_\ell/\Z_\ell})$ for
$g \in H$ drawn from the Haar measure (normalized to be a probability measure) of \autoref{lemma:subscheme-measure-0}.
\end{definition}

By the compatibility with reduction modulo $\ell^j$ discussed above, \autoref{thm: 1-eigenspace markov} then follows from: 

\begin{theorem}\label{theorem: eigenspace markov}
Let $(V,Q)$ be a nondegenerate quadratic space of rank $2m$ over $\Z_\ell$. Let $H \subset \o(V,Q)$ be a union of cosets of $\Omega(V,Q)$. Define the random variable
\[
d_{j}(H) :=\dim_{\F_{\ell}}( \ell^{j-1} \rvcoset {V \otimes \frac{\Q_\ell}{\Z_\ell}[\ell^{j}]} H).
\]
Then the sequence $d_1(H), d_2(H), \ldots$ is Markov, and for $\ell$ odd or $d_i \ne 
2m$, the distribution of $d_{i+1}(H)$ given $d_i(H)$ is the same as 
the dimension of the kernel of a uniform random alternating form on 
$\F_\ell^{d_i(H)}$. 
\end{theorem} 

We prove 
\autoref{theorem: eigenspace markov} in
\autoref{subsection:proof-of-random-kernel-markov}.

\begin{remark}
Another way to think about the numbers $d_j(H)$ is as follows. Decomposing 
\[
\rvcoset V H :=  (\Z/\ell \Z)^{r_1(H)} \oplus (\Z/\ell^2 \Z)^{r_2(H)} \oplus  (\Z/\ell^3\Z)^{r_3(H)} \oplus \ldots 
\]
where the $r_i(H)$ are random variables, we have 
\begin{align*}
d_1(H) & = r_1(H)  + r_2(H) + r_3 (H) + \ldots  \\
d_2(H) &= r_2(H) + r_3(H)+ \ldots  \\
d_3(H) &= r_3(H) + \ldots  \\
&\hspace{.2cm}\vdots 
\end{align*}
\end{remark}

\subsection{Proving \autoref{theorem: eigenspace markov}}
\label{subsection:proof-of-random-kernel-markov}

We now embark on the proof of \autoref{theorem: eigenspace markov}. The proof encompasses this entire subsection, and notation is built cumulatively throughout the section.

We begin by giving one more interpretation of the sequences $d_{j}(H)$. 
Referring to notation of \autoref{theorem: eigenspace markov}, let $V_{j}^H$ 
be the random variable\footnote{We apologize for the similarity to the notation $\vsel n d k$; at least, the latter notation will not appear in this section.}, valued in isomorphism classes of $\F_{\ell}$-vector spaces, given by 
\[
(\ker(g-\id)|_{V/\ell^jV}+\ell V)/\ell V \subset V \otimes \F_\ell,
\]
for $g$ drawn from the Haar measure on $H$. For a fixed $g \in \o(V,Q)$ we write 
\[
V_{j}^g := \ker ((g-\id)|_{V/\ell^jV}).
\]

\begin{lemma}\label{lem: d_l} For a fixed $g \in \o(V,Q)$, the isomorphism $V \otimes_{\Z_{\ell}} \F_\ell \xrightarrow{\sim}V \otimes_{\Z_{\ell}} \frac{\Q_\ell}{\Z_\ell}[\ell]$ identifies
\[
V_j^g 	 \xrightarrow{\sim} \ell^{j-1} \ker \left(g-\id \mid_{V \otimes_{\Z_{\ell}} \frac{\Q_\ell}{\Z_\ell}[\ell^j]} \right).
\]
Hence $\dim V_j^H$ coincides with the random variable $d_j(H)$. 
\end{lemma}
 
\begin{proof}
This is a straightforward verification which follows from commutativity of 
\begin{equation}
 \begin{tikzcd}
 (V \otimes \Q_\ell/\Z_\ell)[\ell^{j}] \ar[r, "\sim"', "\times \ell^j"] \ar[d, "\times \ell^{j-1}"] & V \otimes \Z/\ell^j \Z \ar[d, "\mod{\ell}"] \\
 (V \otimes \Q_\ell/\Z_\ell)[\ell] \ar[r, "\sim"', "\times \ell"] & V \otimes \F_\ell
 \end{tikzcd}
 \end{equation}
\end{proof}

We set $V_0^H := V\otimes_{\Z_{\ell}} \F_\ell$ by convention. We claim that the sequence $V_1^H,V_2^H,\ldots$ of random subspaces is Markov, 
and more precisely that if $\ell$ is odd or $V_j^H \ne V_0^H$, then $V_{j+1}^H$ is the 
kernel of a uniformly distributed alternating form on $V_j^H$. In view of \autoref{lem: d_l}, this will complete the proof of \autoref{theorem: eigenspace markov}.

\begin{lemma}\label{lem: orthogonal complement} The orthogonal complement of $V_j^g \subset V \otimes \F_\ell$ with respect to the quadratic form induced by $Q$ 
is $(\ell^{1-j}(\im(g-\id)\cap 
\ell^{j-1}V))/\ell V \subset V \otimes \F_{\ell}$.
\end{lemma}

\begin{proof} Inside $V/\ell^j V$, we have $\ker((g-\id)|_{V/\ell^j V})^\perp = 
\im((g-\id)|_{V/\ell^j V})$, hence
\[
(\im((g-\id)|_{V/\ell^j V})\cap \ell^{j-1}V/\ell^j V)^\perp
=
\ker((g-\id)|_{V/\ell^j V}) + \ell V.
\]
This immediately induces the claim about orthogonal complements inside 
$V\otimes \F_\ell$. 
\end{proof}

Given $j$ and $g$, for $v \in V_j^g$, we use $\widetilde{v}$ to denote any choice of lift to $V$.

\begin{lemma}\label{lem: conditions to be in subspace} Keep the notation of the preceding discussion. The following are equivalent: 
\begin{enumerate}[(i)]
\item $v \in V_{j+1}^g$, 
\item  $\ell^{-j}(g-\id)\tilde{v}\in (\ell^{1-j}(\im(g-\id)\cap 
\ell^{j-1}V))/\ell V  = (V_j^g)^\perp$,
\item $B(\ell^{-j}(g-\id)\tilde{v},w)=0$ for all $w\in V_j^g$, where $B$ is the bilinear form associated to the quadratic form $Q$ on $V$. 
\end{enumerate}

\end{lemma}

\begin{proof}

Given $v\in V_j^g$, we want to know when it is in $V_{j+1}^g$. The condition that $v \in V_j^g$ is 
equivalent to there being a lift $\tilde{v}$ of $v$ to $V$ such that 
$(g-\id)\tilde{v}\in \ell^jV$. Fixing such a
lift $\wt{v}$, the question is whether we can modify it to another lift $\wt{v}'$ such that $(g-\id)\tilde{v}'\in \ell^{j+1}V$. The freedom for modification is that we can replace $\wt{v}$ by $\wt{v} + \ell \delta$ for some $\delta \in V$. So we want to know if $\delta$ can be chosen so that 
\[
(g-\id) (\wt{v} + \ell\delta) \in \ell^{j+1} V,
\]
or equivalently, so that 
\[
(g-\id) \wt{v} \equiv \ell (g-\id) \delta \mod{\ell^{j+1} V}.
\]
Since we know that $(g-\id) \wt{v} \in \ell^j V$ by assumption, we can rewrite this as 
\[
\ell^{-j} (g-\id) \wt{v} = \ell^{1-j} (g-\id) \delta \in V \otimes \F_\ell
\]
for $\delta$ such that $(g-\id) \delta \in \ell^{j-1}V$. This establishes the equivalence of (i) and (ii). 

The equivalence of (ii) and (iii) then follows from \autoref{lem: orthogonal complement}.
\end{proof}

The $\F_\ell$-linear functional $w \mapsto B(\ell^{-j}(g-\id)\tilde{v},w)$ on $V_j^g$ depends only on $v$, and expresses
$V_{j+1}^g$ as the kernel of a linear transformation $V_j^g\to (V_j^g)^\vee$,
or 
equivalently as the radical of a bilinear form. 

\begin{lemma}
	\label{lemma:alternating-and-radical}
Keep the notation of the preceding discussion. Define the bilinear form on $V_j^g$: 
\[
\langle v, w \rangle_j := B(\ell^{-j}(g-\id)\tilde{v},w).
\]
Then 
\begin{enumerate}[(i)]
\item $V_{j+1}^g$ is the radical of $\langle \cdot, \cdot \rangle_j$. 
\item $\langle \cdot, \cdot \rangle_j$ is alternating.
\end{enumerate}
\end{lemma}

\begin{proof} Part (i) follows from \autoref{lem: conditions to be in subspace}. For (ii), we need to show that 
\[
B((g-\id)\tilde{v},\tilde{v})\in \ell^{j+1}\Z_\ell.
\]
But this follows by observing:
\begin{align*}
B((g-\id)\tilde{v},\tilde{v}) &= Q(g\tilde{v}) - Q((g-\id)\tilde{v}) - Q(\tilde{v}) \\
& =-Q((g-\id)\tilde{v}) \\
& = -\ell^{2j} Q(\ell^{-j}(g-\id)\tilde{v})
\in
\ell^{2j}\Z_\ell.
\qedhere
\end{align*}
\end{proof}

 We thus find that $V_{j+1}^g$ is the kernel of an alternating form on $V_j^g$, 
so it remains only to show that as $g$ varies over elements with fixed 
sequence $(V_1^g,\dots,V_j^g)$, this alternating form is uniformly 
distributed. It suffices to show this when $g$ merely varies 
over elements of a fixed coset of $\Omega(V,Q) \subset \o(V,Q)$. 
Let $\Omega_j \subset \Omega(V,Q)$ be the subgroup consisting of elements which are $1 \mod{\ell^j}$. We will show that the uniform distribution holds already when drawing uniformly from the coset $H = \Omega_j g$. For fixed $v$, changing $g \mapsto hg$ with $h\in  \Omega_j$ changes the linear functional by
\[
w\mapsto B(\ell^{-j}(h-1)g\tilde{v},w)
     = B(\delta_h gv,w)
     = B(\delta_h v,g^{-1}w),
\]
where $\delta_h = \ell^{-j}(h-1)$. We view its reduction modulo as an element of the Lie algebra of the special fiber of $\o(V,Q)$: $\ol{\delta_h} \in \Lie(\o(V,Q)_{\F_{\ell}})$. 
To get equidistribution, it suffices for the induced homomorphism
from $\Omega_j/\Omega_{j+1}$ to the space $\wedge^2(V_j^g)^{\vee}$ of alternating forms 
on $V_j^g$, sending $h$ to the restriction of $\ol{\delta}_h$, to be surjective.

\subsubsection{The case $\ell>2$} If $\ell$ is odd, then $\Omega_1$ is a pro-$\ell$-group, and thus the spinor 
\label{subsubsection:ell>2}
norm vanishes on $\Omega_1$.  It immediately follows that the logarithm induces an isomorphism 
$\Omega_j/\Omega_{j+1} \xrightarrow{\sim} \Lie(\o(V,Q)_{\F_{\ell}}) \cong \wedge^2(V \otimes \F_\ell)^{\vee}$, hence the further projection map to $\wedge^2 (V_j^g)^{\vee}$ is surjective.

\subsubsection{The case $\ell=2$}
\label{subsubsection:ell=2}
For $\ell=2$, it may not be the case that 
$\Omega_j $ surjects on $\Lie \o(V,Q)$.  However, $\Omega(V,Q)$ contains the commutator subgroup of $\o(V,Q)$, and the image of the commutator subgroup in $\Lie(\o(V,Q)_{\F_{\ell}})$ contains the image of $\Ad g - \Id$ for all $g \in \o(V,Q)$. In particular, the image of $\Omega_j$ contains 
\[
(\Ad g - \Id) \cdot \alpha = \alpha\mapsto g \alpha g^t-\alpha
\]
for any $g\in \o(V,Q)$ and any alternating form $\alpha \in \wedge^2 (V \otimes \F_{\ell})^{\vee}$.

Take $g$ to be any lift of the 
reflection in a nonisotropic vector $v \in V_{\F_{\ell}}$  (i.e., a vector with $Q(v) \neq 0$). Denoting $v^* =B(v, \bullet) \in V^{\vee}$, $g \in V_{\F_{\ell}}^{\vee}  \otimes V_{\F_{\ell}} $ can be represented by $\Id + \frac{v^*}{Q(v)} v$ (the unusual expression because we are in characteristic $2$). Then 
\[
g\alpha g^t  - \alpha = \frac{1}{Q(v)}( v^*\otimes v \cdot  \alpha  + \alpha \cdot v^* \otimes v) - \frac{1}{Q(v)^2} (v^* \otimes v) \alpha (v^* \otimes v).
\]
A computation shows all $w^* \otimes v^*$ with $B(v,w) = 0$ are in the space generated by such expressions
\footnote{We spell out this computation in more detail. Let $x$ be such that $B(x,v)= 1$. Take $\alpha$ to be represented by $x^* \otimes w \in V_{\F_{\ell}}^* \otimes V_{\F_{\ell}}$, where we have used $B$ to identify $V$ with $V^*$. Then $g \alpha g^t - \alpha$ is represented by 
\[
\underbrace{(v^* \otimes v )(x^* \otimes w)}_{v^* \otimes w} + \underbrace{(x^* \otimes w)(v^* \otimes v)}_{0} + \underbrace{(v^* \otimes v) (x^* \otimes w) (v^* \otimes v)}_0.
\]
}

 Since for any $w$, $\langle w \rangle^\perp$ is spanned by nonisotropic vectors, the space $\log(\Omega_j)$ in fact contains 
\begin{equation}\label{eq: trace 0 alt forms}
\{ v^* \wedge w^* \co B(v,w)=0\},
\end{equation}
and thus has codimension at most $1$. The full Lie 
algebra $\Lie \o(V,Q)$ is generated over this space by any single element $v^*
\wedge w^*$ with $B(v,w)\ne 0$. If $W$ is any proper subspace of $V$, then we can pick $v \in W^{\perp}$ and $w \in V$ such that $B(v,w) \neq 0$. The image of $v^* \wedge w^*$ in $\wedge^2 W^{\vee}$ is zero, hence the restriction map from \eqref{eq: trace 0 alt forms} to $\wedge^2(W^\vee)$
is surjective for any proper subspace $W \subset V$. Thus the only case in which the 
alternating form may not be equidistributed is when $V_j=V_0$.
This completes the proof of \autoref{theorem: eigenspace markov}. \qed 

\subsection{The BKLPR heuristic}\label{ssec: BKLPR}

We summarize the model for the Selmer group described in \cite[\S 1.2]{bhargavaKLPR:modeling-the-distribution-of-ranks-selmer-groups}. 
\subsubsection{The $\ell^\infty$ rank and Selmer distribution from BKLPR}
\label{subsubsection:ell-infty-bklpr}
Let $m \in \Z$ and $V = \Z_\ell^{2m}$, with the quadratic form $Q \co V \rightarrow \Z_\ell$ given by 
\[
Q(x_1, \ldots, x_m, y_1, \ldots, y_m) = \sum_{i=1}^m x_i y_i.
\]
A $\Z_\ell$-submodule $Z \subset V$ is called \emph{isotropic} if $Q|_{Z}  = 0$. Let $\ogr_{(V,Q)}(\Z_\ell)$ be the set of maximal isotropic \emph{summands} of $V$, hence each $Z \in \ogr_{(V,Q)}(\Z_\ell)$ is a free $\bz_\ell$-module of rank $m$. 

There is a probability measure on $\ogr_{(V,Q)}(\Z_\ell)$ such that the distribution of $Z/\ell^eZ$ in $V/\ell^e V$ for each $e \geq 1$ is uniform
\cite[\S1.2, \S2, \S4]{bhargavaKLPR:modeling-the-distribution-of-ranks-selmer-groups}.
We define $\scr{Q}_{2m,\ell}$ 
(notated in \cite{bhargavaKLPR:modeling-the-distribution-of-ranks-selmer-groups} as $\scr{Q}_{2m}$)
to be the distribution associated to the random variable $S$, valued in isomorphism classes of
abelian groups, where $S$ obtained by drawing $Z$ and $W$ from $\ogr_{(V,Q)}(\Z_\ell)$ independently from this measure, and forming
\[
	S := \left(Z \otimes \frac{\Q_\ell}{\Z_\ell}  \right) \cap  \left(W \otimes \frac{\Q_\ell}{\Z_\ell} \right).
\]

\begin{remark}
	\label{remark:}
	In \cite{bhargavaKLPR:modeling-the-distribution-of-ranks-selmer-groups}, $\scr{Q}_{2m,\ell}$ and related distributions were defined on symplectic abelian groups,
	which are abelian groups together with a nondegenerate alternating pairing to $\bq/\bz$.
	Since two symplectic abelian groups are isomorphic \emph{if and only if} their underlying abelian groups are isomorphic
	\cite[\S3.2]{bhargavaKLPR:modeling-the-distribution-of-ranks-selmer-groups},	
	their distribution
can be regarded as a distribution on abelian groups (which takes probability
$0$ on any abelian group not admitting a symplectic structure). 
\end{remark}

As $m \rightarrow \infty$ the distributions $\scr{Q}_{2m,\ell}$ converge to a discrete probability distribution $\scr{Q}_\ell$ \cite[Theorem 1.2]{bhargavaKLPR:modeling-the-distribution-of-ranks-selmer-groups}, which is conjectured in \cite[Conjecture 1.3]{bhargavaKLPR:modeling-the-distribution-of-ranks-selmer-groups} to determine the asymptotic distribution of $\ell^\infty$-Selmer groups of elliptic curves ordered by height. 

Furthermore, $S$ fits naturally into a short exact sequence
\[
0 \rightarrow R \rightarrow S \rightarrow T \rightarrow 0
\]
where $R := (Z \cap W) \otimes \frac{\Q_\ell}{\Z_\ell}$ and $T$ is torsion. It is further conjectured that the joint distribution of $(R,S,T)$ models the
joint distribution of 
the rank of the elliptic curve (i.e., $R = (\Q_\ell/\Z_\ell)^r$ for $r$ modeling the rank),
the $\ell^\infty$-fSelmer group, and the $\ell$-primary part of the Tate-Shafarevich group, respectively \cite[Conjecture 1.3]{bhargavaKLPR:modeling-the-distribution-of-ranks-selmer-groups}. For example, the following proposition expresses the compatibility of these predictions with the Katz-Sarnak philosophy \cite{KS99}
that 50\% of elliptic curves should have rank $0$ and 50\% should have rank $1$. 

\begin{proposition}[{\cite[Proposition 5.6]{bhargavaKLPR:modeling-the-distribution-of-ranks-selmer-groups}}]
Let notation be as above. Fix $W \in \ogr_{(V,Q)}(\Z_\ell)$. If $Z$ is chosen randomly from $\ogr_{(V,Q)}(\Z_\ell)$ (according to the above measure), then $Z \cap W$ has rank $0$ with probability 1/2 and rank $1$ with probability 1/2. 
\end{proposition}

\subsubsection{The $\ell^\infty$ Selmer distribution from BKLPR conditioned on rank}
\label{subsubsection:conditioned-bklpr}

Let $\scr{T}_{2m,r, \ell}$ be the distribution on finite abelian $\ell$-groups, 
(notated in \cite{bhargavaKLPR:modeling-the-distribution-of-ranks-selmer-groups} as $\scr{T}_{2m,r}$)
given by the above process in \autoref{subsubsection:ell-infty-bklpr} \emph{conditioned} on the assumption $\rank(Z \cap W) = r$. By \cite[Theorem 1.6]{bhargavaKLPR:modeling-the-distribution-of-ranks-selmer-groups}, these distributions converge as $m \rightarrow \infty$ to a discrete distribution $\scr{T}_{r, \ell}$, 
(notated in \cite{bhargavaKLPR:modeling-the-distribution-of-ranks-selmer-groups} as $\scr{T}_{r}$)
which agrees with Delaunay's conjecture for the distribution of $\Sha[\ell^{\infty}]$ of rank $r$ elliptic curves over $\Q$ \cite[p. 278]{bhargavaKLPR:modeling-the-distribution-of-ranks-selmer-groups}.

There is another characterization of the distribution $\mathscr T_{r,\ell}$. 
For non-negative integers $m,r$ with $m - r \in 2 \Z_{\geq 0}$, let $A$ be drawn randomly from the Haar probability measure on the set of \emph{alternating} $m \times m$-matrices over $\Z_{\ell}$ having rank $m-r$, and $\scr{A}_{m,r,\ell}$ be the distribution of $(\coker A)_{\tors}$. According to \cite[Theorem 1.10]{bhargavaKLPR:modeling-the-distribution-of-ranks-selmer-groups}, as $m \rightarrow \infty$ through integers with $m-r \in 2 \Z_{\geq 0}$, the distributions $\scr{A}_{m,r, \ell}$ converge to a limit $\scr{A}_{r,\ell}$, which coincides with $\scr{T}_{r,\ell}$. 

Finally, \cite[\S 5.6]{bhargavaKLPR:modeling-the-distribution-of-ranks-selmer-groups} predicts that, conditioned on elliptic curves having rank $r$, $\Sha$ is distributed as the direct sum over all primes $\ell$ of a finite abelian group drawn from $\scr{T}_{r, \ell}$. 

\subsubsection{The BKLPR $n$-Selmer distribution}
\label{subsubsection:bklpr-n-selmer}
We next review the model for $n$-Selmer elements described at the beginning of \cite[\S 5.7]{bhargavaKLPR:modeling-the-distribution-of-ranks-selmer-groups}.
Let $\mathscr T_{r,\ell}$ denote the random variable defined on isomorphism classes of finite abelian $\ell$ groups (notated $\mathscr T_r$ in \cite{bhargavaKLPR:modeling-the-distribution-of-ranks-selmer-groups})
defined in \cite[Theorem 1.6]{bhargavaKLPR:modeling-the-distribution-of-ranks-selmer-groups} and reviewed in \autoref{subsubsection:conditioned-bklpr}.
For $G$ an abelian group, we let $G[n]$ denote the $n$ torsion of $G$.
For $n \in \bz_{\geq 1}$ with prime factorization $n = \prod_{\ell \mid n} \ell^{a_\ell}$, 
define a distribution $\mathscr T_{r,\bz/n\bz}$ on finitely generated $\bz/n\bz$ modules
by choosing a collection of abelian groups $\{ T_\ell\}_{\ell \mid n}$, 
with $T_\ell$ drawn from $\mathscr T_{r,\ell}$, and defining the probability $\mathscr T_{r,\bz/n\bz} = G$ to be the probability that $\oplus_{\ell \mid n} T_\ell[n] \simeq G$.

Given the above predicted distribution for the $n$-Selmer group of elliptic curves of rank $r$,
the heuristic that $50\%$ of elliptic curves have rank $0$ and $50\%$ have rank $1$ leads to
the following predicted joint distribution of the $n$-Selmer group and rank:
\begin{definition}
	\label{definition:bklpr-rank-selmer}
	Let $\bklpr n$ be the joint distribution on $\bz_{\geq 0}  \times \ab n$ defined by
	\begin{align*}
		\prob(\bklpr n = (r,G)) =
		\begin{cases}
			\frac{1}{2} \mathscr T_{r, \bz/n\bz} & \text{ if } r \leq 1 \\
			0 & \text{ if } r \geq 2. \\
		\end{cases}
	\end{align*}
\end{definition}

\subsection{Markov property for the BKLPR model}
\label{subsection:markov-bklpr}

Fix $Z,W \in \ogr_{(V,Q)}(\Z_\ell)$ and set $S = (Z \otimes \frac{\Q_\ell}{\Z_\ell})  \cap (W \otimes \frac{\Q_\ell}{\Z_\ell})$. Define 
\begin{equation}\label{eq: S_j}
S_j := \left(\underbrace{W/\ell^j \cap Z/\ell^j}_{\subset V/\ell^j} + \frac{\ell V}{\ell^j V}\right)/\ell V,
\end{equation}
which are the analogues of the $V_j$ in \autoref{lem: d_l}. Although $S_j$ depends on $Z$ and $W$, and will be viewed as a random variable in the future, we suppress this dependence for notational convenience. 
The main result of this subsection is the following \autoref{thm: random intersection markov}, and the proof encompasses the remainder of this subsection.

\begin{theorem}\label{thm: random intersection markov}
Let $V,Z,$ and $W$ be as in \autoref{ssec: BKLPR}. Define random variables, valued in isomorphism classes of finite-dimensional $\F_{\ell}$-vector spaces, by $S_0 := V \otimes \F_\ell$, and $S_1, S_2, \ldots, S_j, \ldots$ as in \eqref{eq: S_j}. 
Then, the sequence $S_1, S_2, \ldots$ is Markov, and the distribution of $\dim S_{i+1}$ given $S_i$ coincides with the distribution of the dimension of the kernel of a uniformly random alternating form on 
$S_i$. 
\end{theorem} 

We omit the proof of the following lemma, which is similar to that of \autoref{lem: d_l}.
\begin{lemma}
Keep the notation above. Under the identification 
\[
\left(V \otimes \frac{\Q_\ell}{\Z_\ell } \right)[\ell] \xrightarrow{\sim} V \otimes \F_\ell,
\]
we have 
\[
\ell^{j-1} \cdot S[\ell^j]   \xrightarrow{\sim} S_j.
\] 
\end{lemma}

The non-degenerate bilinear form $B$ on $V$ induces a non-degenerate bilinear form on $V \otimes \F_\ell$, that we denote by $\ol{B}$. We may sometimes abbreviate notation by using $\ol{B}(v,x)$, with $v \in V$ and $x \in V \otimes \F_\ell$, to denote $\ol{B}(v \pmod{\ell}, x)$. 

We will construct the sequence of alternating forms (one for each $S_j$, whose radical is $S_{j+1}$) referenced in \autoref{thm: random intersection markov}.

\begin{lemma}
	\label{lemma: ortho complement S_j}
Identifying $\ell^{1-j} (	\ell^{j-1}V/\ell^j V ) \xrightarrow{\sim} V \otimes \F_{\ell}$, the orthogonal complement of $S_j$ in $V \otimes \mathbb F_\ell$ is $\ell^{1-j}\left( (Z/\ell^j + W/\ell^j) \cap \ell^{j-1}V/\ell^j V \right)$.
\end{lemma}
\begin{proof}
	Inside $V/\ell^j V$, we have 
	\begin{align*}
	\left( Z/\ell^j \cap W/\ell^j \right)^\perp = Z^\perp/\ell^j + W^\perp/\ell^j = Z/\ell^j + W/\ell^j
	\end{align*}
	using that $Z$ and $W$ are maximal isotropic.
	Therefore,
\begin{align*}
	\left( (Z/\ell^j \cap W/\ell^j) + \ell V/\ell^j \right)^\perp = (Z/\ell^j \cap W/\ell^j)^\perp \cap (\ell V/\ell^j)^\perp = (Z/\ell^j + W/\ell^j) \cap \ell^{j-1}V/\ell^{j}.
\end{align*}
The result then follows by tensoring with $\F_\ell$. 
\end{proof}

Next, given $v \in S_j$, we seek to characterize when $v \in S_{j+1}$. By definition, $v \in S_j$ is equivalent to the existence of a representative $\wt{v} \in W/\ell^j \cap  Z/\ell^j$ reducing to $v$ mod $\ell$, and lifts $w_v$ of $\wt{v}$ to $W$ and $z_v$ of $\wt{v}$ to $Z$ such that $w_v \equiv z_v \pmod{\ell^{j} V}$. Hence $w_v-z_v = \ell^{j} \epsilon$ for some $\epsilon \in V$. 

\begin{lemma}
	\label{lemma:equivalent-perp}
	With notation above, $v \in S_j$ lies in $S_{j+1}$ if and only if
	the associated $\epsilon$ as above 
	satisfies $\epsilon \in \ell^{1-j}\left( (Z/\ell^j + W/\ell^j) \cap \ell^{j-1}V/\ell^j V \right).$
\end{lemma}
\begin{proof}
	For $v \in S_{j+1},$ if we can find other lifts $\wt{v}'$, $w_v'$, $z_v'$ satisfying the same conditions, but such that $w_v' \equiv z_v' \pmod{\ell^{j+1}}$. Such modifications are exactly of the form $w_v' = w_v + \ell \delta_W$ with $\delta_W \in W$ and $z_v' = z_v + \ell \delta_Z$ with $\delta_Z \in Z$. Hence $v \in S_{j+1}$ if and only if we can choose $\delta_W, \delta_Z$ such that  
\[
w_v + \ell \delta_W \stackrel{?}= z_v + \ell \delta_Z  + \ell^{j+1} \epsilon'.
\]
Since $w_v = z_v + \ell^{j} \epsilon$, this is equivalent to solving 
\[
\ell^{j-1} \epsilon \equiv \delta_W  - \delta_Z  \pmod{\ell^j} \quad \text{for some $\delta_W \in W/\ell^j, \delta_Z \in Z/\ell^j$}.
\]
which is equivalent to 
\[
\epsilon \in \ell^{1-j}\left( (Z/\ell^j + W/\ell^j) \cap \ell^{j-1}V/\ell^j V \right).
\qedhere
\]
\end{proof}

\begin{lemma}
	\label{lemma:}
	There is a well defined bilinear form 
	\[
	A_j : S_j \times S_j \rightarrow \Q_\ell/\Z_\ell
	\]
	given by 
	\begin{equation}\label{eq: bklpr form}
	A_j(v,x) := \ol{B}(\epsilon, x) = \ol{B} (\ell^{-j}(w_v - z_v), x ).
	\end{equation}
	
\end{lemma}
\begin{proof}
We need to check that the value
	\begin{equation}\label{eq: proposed form}
	\ol{B}(\epsilon, x ) = \ol{B}(\ell^{-j} (w_v-z_v), x) \mod{\ell}.
	\end{equation}
	is independent of the choices of $\wt{v}$, $w_v$, and $z_v$. Indeed, any other allowable $w_v'$ differs from $w_v$ by an element of $\ell^j W$, say $\ell^j \delta$ with $\delta \in W$. But since $W/\ell$ is isotropic and $x$ lies in $S_j \subset W/\ell \subset V/\ell$, we have $B(\delta, x) \equiv 0 \pmod{\ell}$. Similarly, replacing $z_v$ with any other allowable $z_v'$ will not alter \eqref{eq: proposed form}.
\end{proof}

\begin{lemma}
\label{lemma:alternating-and-radical-bklpr}	
	Keep the notation of the preceding discussion. 
\begin{enumerate}[(i)]
\item The radical of $A_j$ is $S_{j+1}$. 
\item $A_j$ is alternating. 
\end{enumerate}
\end{lemma}

\begin{proof}
	By definition, $v \in S_j$ is in the radical of $A_j$ if and only if (following the notation above) $\epsilon_v := \ell^{-j}(w_v - z_v)$ lies in $S_j^{\perp}$. But by \autoref{lemma: ortho complement S_j}, $\epsilon \in S_j^{\perp}$ if and only if $\epsilon_v \in \ell^{1-j}\left( (Z/\ell^j + W/\ell^j) \cap \ell^{j-1}V / \ell^j V\right)$, which, as we proved in \autoref{lemma:equivalent-perp}, occurs if and only if $\epsilon \in S_{j+1}$. 

For (ii), since we can take $z_v$ as a lift of $v$ to $V$, it suffices to check $B(w_v-z_v  , z_v ) \in \ell^{j+1} \bz_\ell$.
	For this, write $w_v-z_v = \ell^j \epsilon$ and observe that $Z$ and $W$ are isotropic for $Q$, we have
	\begin{align*}
		B( w_v-z_v , z_v ) &=
		Q(w_v) - Q(w_v- z_v) - Q(z_v) \\
		& = Q(w_v - z_v) \\
		&= Q(\ell^{j} \epsilon)  \\
		& = \ell^{2j}Q(\epsilon) \in \ell^{j+1} \bz_\ell. \qedhere
	\end{align*}
\end{proof}

As in \autoref{ssec: markov 1-eigenspace}, it suffices to show that as $Z$ and $W$ are drawn from the canonical measure on $\ogr_{(V,Q)}(\Z_\ell)$, the alternating form $A_j$ is uniformly distributed.

\begin{lemma}\label{lem: trans action}
$\o(V,Q)$ acts transitively on $\ogr_{(V,Q)}(\Z_\ell)$. 
\end{lemma}

\begin{proof}
Fix $W,Z \in \ogr_{(V,Q)}(\Z_\ell)$. Then we have a scheme 
\[
\mrm{Isom}(W,Z) = \{ g \in  \o(V,Q) \co g W = Z\}  \subset \o(V,Q)
\]
over $\Z_\ell$. This is evidently a torsor for the parabolic subgroup $\Isom(W,W) \subset \o(V,Q)$. Moreover, Witt's theorem implies that $\mrm{Isom}(W,Z)$ has a point over $\F_\ell$, which lifts to a $\Z_\ell$-point because $\mrm{Isom}(W,Z)$ is smooth (being a torsor for a smooth group scheme). 
\end{proof}

It will suffice to show that conditioning on a fixed $W$, the distribution of $A_j$ is already uniform. The distribution of $Z$ conditioned on a fixed $W$ coincides with the orbit measure on $\ogr_{(V,Q)}(\Z_\ell)$ 
induced by the Haar measure on $\o(V,Q)$, since $\o(V,Q)$ acts transitively on $\ogr_{(V,Q)}(\Z_\ell)$ by \autoref{lem: trans action}. As in \autoref{ssec: markov 1-eigenspace}, it suffices to show that the distribution of $A_j$ is already uniform as $Z$ varies over an orbit of a coset of the principal congruence subgroup
\[
\Gamma(\ell^j) := \{ g \in \o(V,Q) \co g \equiv \Id \pmod{\ell^j} \}.
\]
For fixed $Z^0$, which induces the alternating form 
\[
A_j(v, x) = \ol{B}(\ell^{-j} (w_v-z_v^0), x),
\]
the alternating form associated to $\gamma Z^0$ for $\gamma \in \Gamma(\ell^j)$ is 
\[
\ol{B}(\ell^{-j}(w_v-\gamma z_v^0), x)
\]
which changes the functional by 
\[
x \mapsto \ol{B}(\ell^{-j} (1-\gamma) z_v^0, x).
\]
Now, since the map $\gamma \mapsto 1 - \gamma$ induces an isomorphism $\Gamma(\ell^j) / \Gamma(\ell^{j+1}) \xrightarrow{\sim} \Lie \o(V_{\F_\ell}, Q)$,
the resulting alternating form $A_j$ is uniformly distributed,
so we are done.
\qed
\begin{remark}
	\label{remark:}
	Note that unlike in the case of the random kernel model, where we had additional complications to deal with associated to $\ell = 2$ in \autoref{subsubsection:ell=2},
	there are no additional complications here for $\ell=2$ in the proof of \autoref{thm: random intersection markov}, because here we are working with the full congruence subgroup $\Gamma(\ell^j),$ instead of a subgroup which may have index $2$,
as was the case in \autoref{subsection:proof-of-random-kernel-markov}.
\end{remark}

\section{Proofs of the main theorems}
\label{section:proofs}

We conclude the paper by proving our main theorems.
In \autoref{subsection:selmer-to-random-kernel}
we connect the actual Selmer distribution to the random kernel model,
while in
\autoref{subsection:random-kernel-to-bklpr}
we connect the random kernel model to the BKLPR distribution.
Combining these gives us a proof of our main theorem, \autoref{theorem:main}.
Finally, in
\autoref{subsection:remaining-results}
we prove \autoref{corollary:squarefree-moments}
and
\autoref{corollary:prime-moments}.

\subsection{Comparing the Selmer distribution with the random kernel model}
\label{subsection:selmer-to-random-kernel}

To start, we state one of our main theorems, which compares the distribution of Selmer groups of elliptic curves to the random kernel model.
We prove this at the end of the subsection.

\begin{theorem}
	\label{theorem:large-q-distribution}
	Fix integers $d \geq 2$ and $n \geq 1$.
	For $q$ ranging over prime powers, with $\gcd(q,2n) = 1$ and $(r,G) \in \bz_{\geq 0} \times\ab n$, we have
	\begin{equation}
	\label{equation:selmer-error-estimate}
		\begin{aligned}
			\prob(\rvsel n d {\mathbb F_q} \simeq G) &= \prob(\rvker n d {\mathbb F_q} = G) + O_{n,d}(q^{-1/2})
		\end{aligned}
\end{equation}
and
	\begin{equation}
	\label{equation:error-estimate}
		\begin{aligned}
		\prob(\algsel n d {\mathbb F_q} = (r,G)) &=
		\prob(\ansel n d {\mathbb F_q} = (r,G)) + O_{n,d}(q^{\error d}) \\
	&=	\prob(\flr n d {\mathbb F_q} = (r,G)) + O_{n,d}(q^{\error d}).
\end{aligned}
\end{equation}
In particular,
	\begin{align}
		\label{equation:lim-sup}
		\limsup\limits_{\substack{q \ra \infty\\ \gcd(q,2n)=1}} \ansel n d {\mathbb F_q} &=
		\limsup\limits_{\substack{q \ra \infty\\ \gcd(q,2n)=1}} \algsel n d {\mathbb F_q} =
		\limsup\limits_{\substack{q \ra \infty}} \flr n d {\mathbb F_q} \\
		\label{equation:lim-inf}
		\liminf\limits_{\substack{q \ra \infty\\ \gcd(q,2n)=1}} \ansel n d {\mathbb F_q} &=
		\liminf\limits_{\substack{q \ra \infty\\ \gcd(q,2n)=1}} \algsel n d {\mathbb F_q} =
		\liminf\limits_{\substack{q \ra \infty}} \flr n d {\mathbb F_q},
	\end{align}
The values of \eqref{equation:lim-sup} and \eqref{equation:lim-inf} agree when $d$ is odd or $n \leq 2$, but differ when $d$ is even and $n > 2$.
		\end{theorem}
		
We are nearly ready to prove \autoref{theorem:large-q-distribution},
but first we will need to establish two preliminary results.
The first preliminary result relates the Selmer group of an elliptic curve to the $1$-eigenspace of Frobenius.
\begin{lemma}
	\label{lemma:selmer-as-frob-fixed}
	For $n \geq 1, d \geq 2$ and $[E_x] = x \in \smespace d {\bz[1/2n]}(\F_q)$, we have 
	\[
	\Sel_n(E_x) = \ker\left( \mono n d {\bz[1/2n]}(\frob_x) - \id |_{\left(\smestack d k\right)_{x}}\right).
	\]
\end{lemma}
\begin{proof}
	Notate the geometric fiber of $\smestack d k$ over $x$ by $\left(\smsstack n d k\right)_{\ol x}$, and the fiber by $\left(\smsstack n d k\right)_{x}$.
	Since $\left(\smestack d k\right)_{x}$ is a finite \'{e}tale $\F_q$-scheme, we have 
\[
	\left(\smsstack n d k\right)_{x}(\mathbb F_q) = \ker\left( \mono n d {\bz[1/2n]}(\frob_x) - \id |_{\left(\smestack d k\right)_{\ol{x}}}\right).
\]
	Hence, combining this with \autoref{lemma:selmer-fiber-estimate}, we obtain that for $[E_x] = x \in \smestack d k$,
	\[
		\ker\left( \mono n d {\bz[1/2n]}(\frob_x) - \id |_{\left(\smestack d k\right)_{\ol{x}}}\right) = \Sel_n(E_x).
	\]
	Here we are using that there is an isomorphism $(\smsstack n d k)_x \simeq (\smsspace n d k)_{x'}$ for $x' \in \smespace d k$ mapping to $x$,
	coming from the definition of $\smsstack n d k$ and $\smestack d k$ as quotients of $\smsspace n d k$ and $\smespace d k$ by a compatible group action.
\end{proof}

Our second preliminary result relates the rank of an elliptic curve
$[E_x] \in \smespace d k(\F_q)$ to the Dickson invariant of $\mono {\bz_\ell} d k(\frob_x)$.

Recall from \autoref{definition:selmer-space-quadratic-form}
	that $(\qsel \bz d k, \vsel \bz d k)$ denotes the quadratic space over $\bz$, whose reduction $\mod n$ is $(\qsel n d k, \vsel n d k)$
on which the monodromy representation
$\mono n d k$ acts.
Let $(\qsel {\bz_\ell} d k, \vsel {\bz_\ell} d k) := (\qsel \bz d k \otimes_{\bz} \bz_\ell, \vsel \bz d k \otimes_\bz \bz_\ell)$ denote the base change to $\bz_\ell$.

\begin{proposition}
	\label{proposition:component-and-rank}
	Let $d \geq 2$, and let $\ell$ be a prime. For $q$ a prime power with $\gcd(q, 2\ell) = 1$,
define	
\begin{align*}
	\smallrank \ell d q := 
	\left\{[E_x] = x \in \sfespace d {\bz[1/2\ell]}(\mathbb F_q) :
	\anrk(E_x) \leq 1  \right\}.
\end{align*}
	\begin{enumerate}
		\item For $q$ ranging over prime powers with $\gcd(q,2\ell) = 1$, we have
\begin{align*}
\frac{\# \smallrank \ell d q }
{ \# \sfespace d {\bz[1/2\ell]}(\mathbb F_q)} = 1 + O_{d}\left(q^{\error d} \right).
\end{align*}
	\item For all $x \in \smallrank \ell d q \subset \sfespace d {\bz[1/2\ell]}(\mathbb F_q)$, we have 
	\[
	\rk \ker\left( \mono {\bz_\ell} d {\bz[1/2\ell]}(\frob_x) - \id\right) = \begin{cases} 0 & \iff \mono {\bz_\ell} d {\bz[1/2\ell]}(\frob_x) \in \so(\qsel {\bz_\ell} d {\bz[1/2\ell]}) \\
	1 & \iff \mono {\bz_\ell} d {\bz[1/2\ell]}(\frob_x) \notin \so(\qsel {\bz_\ell} d {\bz[1/2\ell]}).\end{cases}
\]
	\item The above statements are true with analytic rank replaced by algebraic rank.
	\end{enumerate}
\end{proposition}
\begin{proof}
	To start, observe that $(2)$ follows directly from 
	\autoref{lemma:generalized-eigenspace-is-rank} and
	\autoref{proposition:eigenvalues}

	We next demonstrate (1). By \autoref{lemma:generalized-eigenspace-is-rank}, whenever $x \in \sfespace d \bz[1/2\ell]$,
	the analytic rank of $E_x$ is equal to the rank of the $1$-generalized eigenspace of
	$\mono {\bz_\ell} d {\bz[1/2\ell]}(\frob_x) - \id$.

	By \autoref{proposition:eigenvalues}, whenever $x \notin \smallrank \ell d q$, there is a particular Zariski closed hypersurface
	$Z$ in the algebraic group $\o(\qsel {\bz_\ell} d {\bz[1/2\ell]})$,
	i.e., the hypersurface parameterizing elements with a two or more dimensional generalized $1$-eigenspace,
	such that $\mono {\bz_\ell} d {\bz[1/2\ell]}(\frob_x) \in Z(\bz_\ell)$.
	By \autoref{lemma:subscheme-measure-0}, for any positive integer $e$, we have 
	\[
	\im ( Z(\bz/\ell^e \bz) \ra \o(\qsel {\bz_\ell} d {\bz[1/2\ell]})(\Z/\ell^e \Z)) = O_{Z}\left(\ell^{e(\dim \o(\qsel {\bz_\ell} d {\bz[1/2\ell]})-1)} \right) = O_{\ell, d}\left(\ell^{e(\dim \o(\qsel {\bz_\ell} d {\bz[1/2\ell]})-1)} \right).
	\]
	By \autoref{theorem:monodromy}, we know $\im \mono {\ell^e} d {\bz[1/2\ell]}$ has index at most $2$ in $\o(\qsel {\bz_\ell} d {\bz[1/2\ell]})$,
	and hence has size within a constant factor of $\ell^{e\dim \o(\qsel {\bz_\ell} d {\bz[1/2\ell]})}$.
	Therefore, it follows from \autoref{proposition:chavdarov}
	that
\begin{equation}
	\label{equation:rank-2-bound}\Scale[0.9]{
\begin{aligned}
	\frac{\#  \left( \sfespace d {\bz[1/2\ell]}(\mathbb F_q)- \smallrank \ell d q \right) }
	{ \# \sfespace d {\bz[1/2\ell]}(\mathbb F_q)} 
	&=
	\frac
{\# \im ( Z(\bz/\ell^e \bz) \ra \o(\qsel {\bz_\ell} d {\bz[1/2\ell]}))}
	{\# \im \mono {\ell^e} d {\bz[1/2\ell]}} \\
	&\hspace{1cm}+
	O_{d}\left( \# \im \mono {\ell^e} d {\bz[1/2\ell]} \sqrt{ \frac{ \# \im ( Z(\bz/\ell^e \bz) \ra \o(\qsel {\bz_\ell} d {\bz[1/2\ell]}))}{q}} \right) \\
	&=
	O_{\ell, d} \left( \frac{\ell^{e(\dim \o(\qsel {\bz_\ell} d {\bz[1/2\ell]})-1)} }{\ell^{e \dim \o(\qsel {\bz_\ell} d {\bz[1/2\ell]})}}+ q^{-1/2} \ell^{e \dim \o(\qsel {\bz_\ell} d {\bz[1/2\ell]})}
	{\ell^{\frac{1}{2} e(\dim \o(\qsel {\bz_\ell} d {\bz[1/2\ell]})-1)} } \right) \\
	&=
	O_{\ell,d} \left( \ell^{-e} + q^{-1/2} (\ell^{e})^{(\frac{3}{2}  \dim \o(\qsel {\bz_\ell} d {\bz[1/2\ell]}) - \frac 1 2)}\right).
\end{aligned} } \end{equation}
Crucially, the above constant does not depend on $e$, and so we may freely choose $e$ to minimize the above error term.
Indeed, we may take $e$ to be the least positive integer so that $q \leq (\ell^e)^{(1 + 3 \dim \o(\qsel {\bz_\ell} d {\bz[1/2\ell]}))}$,
or equivalently $q^{\frac{1}{1 + 3 \dim \o(\qsel {\bz_\ell} d {\bz[1/2\ell]})}} \leq \ell^e$. 
Then, so long as $q > \ell$, replacing $q$ by $(\ell^e)^{(1 + 3 \dim \o(\qsel {\bz_\ell} d {\bz[1/2\ell]}))}$
will introduce at most a factor of $\ell$, and so
\begin{equation}
	\label{equation:error-for-rank-2}\Scale[0.9]{
\begin{aligned}
	O_{\ell, d}(\ell^{-e}) &= O_{\ell, d}(q^{\frac{-1}{1 + 3 \dim \o(\qsel {\bz_\ell} d {\bz[1/2\ell]}))}} ) \\
	O_{\ell, d}(q^{-1/2} (\ell^{e})^{(\frac 3 2 \dim \o(\qsel {\bz_\ell} d {\bz[1/2\ell]}) - \frac 1 2)}) &= 
O_{\ell, d}\left(q^{-\frac{1}{2} + \frac{\frac{3}{2}  \dim \o(\qsel {\bz_\ell} d {\bz[1/2\ell]}) - \frac 1 2}{1 + 3 \dim \o(\qsel {\bz_\ell} d {\bz[1/2\ell]})}} \right)
=O_{\ell, d}\left(q^{\frac{-1}{1 + 3 \dim \o(\qsel {\bz_\ell} d {\bz[1/2\ell]})}} \right).
\end{aligned}}
\end{equation}
Further, for the finitely many $q < \ell$, we can adjust the constants so that the above still holds with no dependence on $q$.

Combining \eqref{equation:rank-2-bound} and \eqref{equation:error-for-rank-2},
we find
\begin{align*}
	\frac{\#  \left( \sfespace d {\bz[1/2\ell]}(\mathbb F_q)- \smallrank \ell d q \right) }
	{ \# \sfespace d {\bz[1/2\ell]}(\mathbb F_q)} 
	&=
O_{\ell, d}\left(q^{\frac{-1}{1 + 3 \dim \o(\qsel {\bz_\ell} d {\bz[1/2\ell]}))}} \right).
\end{align*}
Further, the constant above does not depend on $\ell$ because the analytic rank, and hence the subset
$\smallrank \ell d q \subset \sfespace d {\bz[1/2\ell]}(\mathbb F_q)$
is independent of the auxiliary choice of $\ell$.
Now, (1) follows because 
\begin{align*}
\frac{-1}{1 + 3\dim \o(\qsel {\bz_\ell} d {\bz[1/2\ell]}))} = \frac{-1}{1 + \frac{3(12d-4)(12d-5)}{2}} 
= \frac{-1}{1 + 3(6d-2)(12d-5)}
= \frac{-1}{216d^2 -162d+31}.
\end{align*}

	Part (3) follows from the proceeding ones and 
	fact that, for elliptic curves of rank at most $1$ over $\F_q$ of characteristic $\geq 3$, we know on a full density (as $q \rightarrow \infty$) subset that algebraic rank equals analytic rank. For $\chr \F_q >3$ the statement holds for every elliptic curve of rank at most $1$, as explained in \cite[\S 3.8]{ulmer:elliptic-curves-and-analogies}, using the analogue of the Gross-Zagier formula in \cite[Theorem 1.2]{ulmer:geometric-non-vanishing}. If $\chr \F_q = 3$, it follows by combining \cite[\S 3.8]{ulmer:elliptic-curves-and-analogies} with the Gross-Zagier formula for everywhere semistable elliptic curves in \cite[Remark 1.5]{YZ19}. Note that there is an open subscheme $\sfespace d B \subset \smespace d B$ parameterizing those elliptic surfaces which have squarefree discriminant, so are everywhere semistable. This is fiberwise dense over $B$ by \cite[Lemma 3.14]{landesman:geometric-average-selmer}, so that in the large $q$ limit, a density $1$ subset of $\espace d B(\mathbb F_q)$
	corresponds to elliptic curves with everywhere semistable reduction.
\end{proof}

\begin{proof}[Proof of \autoref{theorem:large-q-distribution}]
	We will explain how the distribution of 
	$\ansel n d {\mathbb F_q}$ and $\algsel n d {\mathbb F_q},$ up to an error of $O_{n,d}(q^{\error d})$, are determined by the distributions of $\frob_x$ for $x \in  \smallrank \ell d q \subset \sfestack d {\mathbb F_q}(\mathbb F_q)$,
	as defined in \autoref{proposition:component-and-rank}.
	By definition, these distributions are determined by $\frob_x$ for
	$x \in \estack d {\mathbb F_q}(\mathbb F_q)$, so we only
	need justify why there are 
	$O_{n,d}(q^{\error d})$ points in 
$\sfestack d {\mathbb F_q}(\mathbb F_q)- \smallrank \ell d q$,

	To start, we explain why
	$\ansel n d {\mathbb F_q}$ and $\algsel n d {\mathbb F_q}$ agree with their restrictions 
	from $\estack d {\mathbb F_q}(\mathbb F_q)$	
	to $\sfestack d {\mathbb F_q}(\mathbb F_q)$, up to an error of $O_{n,d}(q^{-1/2})$.
	The argument here is analogous to that in \autoref{remark:distribution-definition-variations}.
	Indeed, the closed substack $\sfestack d {\mathbb F_q} - \estack d {\mathbb F_q} \subset \estack d {\mathbb F_q}$ has positive codimension. Hence, contributes at most $O_{n,d}(q^{-1/2})$
	to the distributions $\ansel n d {\mathbb F_q}$ and $\algsel n d {\mathbb F_q}$, as can be deduced from the Lang-Weil estimate and \cite[Lemma 5.3]{landesman:geometric-average-selmer}.

	We next explain how to relate the distribution of $\mono n d {\bz[1/2]}(\frob_x)$ over $x \in \sfestack d {\mathbb F_q}(\mathbb F_q)$ to $\ansel n d {\mathbb F_q}$ and $\algsel n d {\mathbb F_q}$.
	The key will be the following two results shown above.
	\begin{enumerate}[(i)]
		\item By \autoref{lemma:selmer-as-frob-fixed}, we have $\Sel_n(E_x) = \ker\left( \mono n d {\bz[1/2n]}(\frob_x) - \id |_{\left(\smestack d {\mathbb F_q} \right)_{x}}\right)$.
		\item By \autoref{proposition:component-and-rank}, there is a subset $\smallrank \ell d q \subset \sfestack d {\mathbb F_q}(\mathbb F_q)$ whose density is $1 + O_{d}(q^{\error d})$ for
			$q$ ranging over prime powers with $\gcd(q,2\ell)  = 1$
		such that
		\[
			\rk(E_x) = \anrk(E_x) = \delta_{\mono n d {\bz[1/2n]}(\frob_x) \notin \so(\qsel n d {\mathbb F_q})},
		\]
		where $\delta_{a \notin B} = 1$ if $a \notin B$ and $0$ if $a \in B$.
	\end{enumerate}
	The observation (i) then establishes \eqref{equation:selmer-error-estimate}.
	Combining (i) and (ii) with the preceding discussion, we have explained how the distribution of Frobenius elements determines the joint distributions
		$\ansel n d {\mathbb F_q}$ and $\algsel n d {\mathbb F_q}$,
		up to an error of $O_{n,d}(q^{\error d})$.
	By 
	\autoref{corollary:finite-field-distribution} and \autoref{corollary:selmer-equidistribution},
	up to an error of $O_{n,d}(q^{-1/2})$, the elements $\mono n d {\bz[1/2n]}(\frob_x)$ 
are equidistributed between the two cosets of $\Omega(\qsel n d k)$ given by 
\begin{align*}
\left( \dickson{\qsel n d  {\bz[1/2n]}}, \sp{\qsel n d {\bz[1/2n]}} \right) \in \left\{ \left( (0, \ldots, 0), [q^{d-1}] \right) , \left( (1, \ldots, 1), [q^{d-1}] \right)\right\}.
\end{align*}
This describes the distribution $\flr n d {\mathbb F_q}$ and hence yields 
\eqref{equation:error-estimate}, \eqref{equation:lim-sup} and \eqref{equation:lim-inf}.

	To conclude the proof we need justify the values of \eqref{equation:lim-sup} and \eqref{equation:lim-inf} agree when $d$ is odd or $n \leq 2$ but differ when $d$ is even and $n > 2$.
	Because these limits approach $\flr n d {\mathbb F_q}$, 
	it suffices to show $\flr n d {\mathbb F_q}$ is independent of $q$ when $d$ is odd or $n \leq 2$ but depends on $q$ when $d$ is even.
	When $d$ is odd, this follows from \autoref{definition:kernel-distribution} because the square class of $q^{d-1}$ is always trivial, hence independent of $q$.
	Also, when $n \leq 2$, this holds again by \autoref{definition:kernel-distribution} because the spinor norm is trivial.
	However, when $d$ is even and $n > 2$, the spinor norm is nontrivial, and $\flr n d {\mathbb F_q}$ will change depending on whether $q$ is a square or nonsquare.
	Indeed, when $q$ is a square, $\prob(\rvker n d {\mathbb F_q} = \left( \bz/n\bz \right)^{12d-4}) > 0$,
	corresponding to the case that $g = \id$ in \autoref{definition:kernel-distribution},
	while when $q$ is not a square, $\prob(\rvker n d {\mathbb F_q} = \left( \bz/n\bz \right)^{12d-4}) = 0$.
\end{proof}

\subsection{Comparing the random kernel model with the BKLPR heuristic}
\label{subsection:random-kernel-to-bklpr}

We now prove: 

\begin{theorem}\label{theorem:kernel vs bklpr}
The TV distance between the BKLPR heuristic and $\limsup\limits_{\substack{q \ra \infty}} \flr n d {\mathbb F_q}$ is $O(2^{-(6d-2)^2})$, where the implicit constant is absolute, and similarly for the TV distance between the BKLPR heuristic and $\liminf\limits_{\substack{q \ra \infty}} \flr n d {\mathbb F_q}$

In particular, we have
\begin{align*}
	\bklpr {n} & = \lim_{d \rightarrow \infty} \limsup\limits_{\substack{q \ra \infty}} \flr n d {\mathbb F_q} \\
	&= \lim_{d \rightarrow \infty}  \liminf\limits_{\substack{q \ra \infty}} \flr n d {\mathbb F_q}.
\end{align*}
\end{theorem}

\begin{proof}
By \autoref{definition:kernel-distribution}, with probability one the rank is $0$ or $1$, and determined by whether the random $g$ in the random kernel model has Dickson invariant $0$ or $1$, respectively. Hence the rank component of these distributions is completely determined by the Selmer component, we can focus our attention on the Selmer component. 

Thanks to \autoref{cor: TV estimate for prime case}, we know that the TV distance between $\liminf_{q \to \infty} \dim \rvker \ell d {\mathbb F_q}$ and the BKLPR heuristic for $\Sel_{\ell}$ is $O(\ell^{-(6d-2)^2})$, and similarly for $\limsup_{q \to \infty}$ in place of $\liminf_{q \to \infty}$. The Markov properties \autoref{thm: 1-eigenspace markov} and \autoref{cor: 1-eigenspace markov} and \autoref{thm: random intersection markov} imply that for $\ell>2$, the two distributions for $\Sel_{\ell^e}$ agree conditioned upon them agreeing for $\Sel_{\ell}$. For $\ell=2$, the same is true as long as $d_1<12d-4$ where the notation $d_1$ is as in \autoref{thm: 1-eigenspace markov}, which only fails if $g$ reduces to the identity element in $\mrm{O}(12d-4, \F_{\ell})$. This happens with probability $1/\# \mrm{O}(12d-4, \F_{\ell})$, which is negligible compared to the error term we seek. We conclude that the TV distance between the two distributions for $\Sel_{\ell^e}$ is also $O(\ell^{-(6d-2)^2})$. 

Finally, we consider general $n$. For $n = \prod \ell^{a_{\ell}}$, the prime factorization of $n$, we have 
\[
\Sel_n \cong \oplus_{\ell} \Sel_{\ell^{a_{\ell}}}.
\]
The BKLPR heuristic predicts that the distributions of the $\Sel_{\ell^{a_{\ell}}}$ are independent after conditioning on the rank. If $(V,Q)$ is a quadratic form over $\bz/n\bz$ then note that $\Omega(Q) \simeq \prod_{\text{prime }\ell \mid n} \Omega(Q|_{\bz/\ell^{a_\ell} \bz}).$
Therefore, conditioned on each coset of $\Omega$ in $H_{\ell,k}^{d,i}$ the distributions $(\mrm{RSel}_{\ell^{a_{\ell}}}^{\mrm{kernel}})_{\F_q}^d$ are independent. 

Since the TV distance of two product distributions is the sum of the TV distance of the factors, the TV distance between the BKLPR heuristic and $\limsup\limits_{\substack{q \ra \infty}} \flr n d {\mathbb F_q}$ is
\[
\ll \sum_{\text{prime } \ell \mid n} \ell^{-(6d-2)^2} \ll \zeta((6d-2)^2)-1 \ll 2^{-(6d-2)^2}.\qedhere
\]
\end{proof}

We can now complete the proof of \autoref{theorem:main}. 

\begin{proof}[Proof of \autoref{theorem:main}]
This follows immediately from combining \autoref{theorem:large-q-distribution} and  \autoref{theorem:kernel vs bklpr}.
\end{proof}

\subsection{Remaining results}
\label{subsection:remaining-results}
We conclude by proving two remaining results, promised in the introduction.
First, we prove \autoref{corollary: minimalist-precise}, which is a version of \autoref{corollary: minimalist} with more precise error terms,
and then we prove \autoref{corollary:squarefree-moments-precise} which is a version of \autoref{corollary:squarefree-moments} with more precise error terms.

\begin{corollary}[Large $q$ analog of \protect{\cite[Conjecture 1.2]{poonenR:random-maximal-isotropic-subspaces-and-selmer-groups}}]
	\label{corollary: minimalist-precise}
	For fixed integers $d \geq 2$ and $n \geq 1$, and $q$ ranging over prime powers with $\gcd(q, 2n) = 1$, 
	we have
\begin{align}
	\prob(\rvrk n d {\mathbb F_q} = r) &=
		\begin{cases}
			1/2 + O_d(q^{\error d}) & \text{ if } r \leq 1, \\
			O_d(q^{\error d})  & \text{ if } r \geq 2.
		\end{cases} 
	\end{align}
		Furthermore, 
		\[
			\EE[\rvrk n d {\mathbb F_q}] = 1/2 +  O_d(q^{\error d}).
\]
\end{corollary}
\begin{proof}	
	The first statement follows immediately from \eqref{equation:error-estimate}
	by summing over the set of possible groups $G$ which can appear.
	For the statement regarding average rank, we also need to know that there is a uniform bound
	on the rank of elliptic curves of height $d$ over $\mathbb F_q(t)$, only depending on $d$.
	This holds because the rank is bounded by the size of the Selmer group, which is uniformly bounded in $q$ among all elliptic curves of height $d$,
	as follows from \cite[Corollary 3.27]{landesman:geometric-average-selmer}, since the Selmer space $\sspace n d {\mathbb F_q}$ is quasi-compact and quasi-finite over
	$\espace d {\mathbb F_q}$
	and hence has uniformly bounded fiber degree.
\end{proof}

\begin{theorem}[Large $q$ analog of \protect{\cite[Conjecture 1.4]{poonenR:random-maximal-isotropic-subspaces-and-selmer-groups}}]
	\label{corollary:squarefree-moments-precise}
	Let $n$ be a squarefree positive integer, $d \geq 2$, and $\omega(n)$ be the number of prime factors of $n$.
	\begin{enumerate}
		\item Fix $c_\ell \in \bz_{\geq 0}$ for each prime $\ell \mid n$. Then
\begin{equation}
	\label{equation:poonen-rains-squarefree-prediction-precise}\Scale[0.8]{
\begin{aligned}
	&\lim_{d \ra \infty}\limsup_{\substack{q \ra \infty \\ \gcd(q,2n)=1}} \prob\left( \rvsel n d {\mathbb F_q} \simeq \prod_{\ell \mid n} \left( \bz/\ell \bz \right)^{c_\ell} \right) =\lim_{d \ra \infty}\liminf_{\substack{q \ra \infty \\ \gcd(q,2n)=1}} \prob\left( \rvsel n d {\mathbb F_q} \simeq \prod_{\ell \mid n} \left( \bz/\ell \bz \right)^{c_\ell} \right) \\
	&= 
	\begin{cases}
	2^{\omega(n)-1} \prod_{\ell \mid n} \left( \left( \prod_{j \geq 0} \left( 1-\ell^{-j} \right)^{-1} \right) \left( \prod_{j=1}^{c_\ell} \frac{\ell}{\ell^j-1} \right) \right)
		& \text{ if all $c_\ell$ have the same parity},  \\
		0 & \text{otherwise.}  \\
	\end{cases}
\end{aligned}}
\end{equation}
		\item For $q$ ranging over prime powers with $\gcd(q, 2n) = 1$, we have
		\[
			\EE[ \# \rvsel n d {\mathbb F_q}] = \sigma(n) +O_{n,d}(q^{-1/2}) :=\sum_{s \mid n} s + O_{n,d}(q^{-1/2}).
		\]
		\item For $m \leq 6d-3$ the $m$th moment of $\rvsel n d {\mathbb F_q}$ is 
			\[
				\EE[(\# \rvsel n d {\mathbb F_q})^m] = \prod_{\text{prime }\ell \mid n} \prod_{i=1}^m \left( \ell^i +1 \right) + O_{n,d,m}(q^{-1/2}) .
			\]
	\end{enumerate}
\end{theorem}

\begin{proof}
	The first part follows from \autoref{theorem:main} once we establish that 
	$\mathrm{Sel}_{n}^{\mathrm{BKLPR}}$ has distribution as predicted in the bottom line of
	\eqref{equation:poonen-rains-squarefree-prediction-precise}.
	To see this, note that, by definition, the model $\mathrm{Sel}_{n}^{\mathrm{BKLPR}}$ is determined by the models for
	$\mathrm{Sel}_{\ell}^{\mathrm{BKLPR}}$ with $\ell \mid n$ which are independent, except for the constraint that the parities of their $\bz/\ell\bz$ ranks are all equal.
	Hence, it suffices to establish the first part in the case $n= \ell$ is prime. 	Note that the model
	$\mathrm{Sel}_{\ell}^{\mathrm{BKLPR}}$ agrees with the model for $\ell$-Selmer groups
	defined in \cite[Definition 2.9]{poonenR:random-maximal-isotropic-subspaces-and-selmer-groups}
	by \cite[Theorem 2.19(f)]{poonenR:random-maximal-isotropic-subspaces-and-selmer-groups}.
	Therefore, in the case $n = \ell$ is prime, $\mathrm{Sel}_{\ell}^{\mathrm{BKLPR}}$ has distribution as predicted in the bottom line of
	\eqref{equation:poonen-rains-squarefree-prediction-precise} by 
	\cite[Proposition 2.6(d) and (f)]{poonenR:random-maximal-isotropic-subspaces-and-selmer-groups}.

	Note that (2) is the special case of (3) with $m = 1$, so it suffices to prove (3).
	To simplify notation in the ensuing proof, we use $\smsspace n {d,m} {\mathbb F_q}$ to denote $\underbrace{\smsspace n d {\mathbb F_q} \times_{\smespace d {\mathbb F_q}} \cdots \times_{\smespace d {\mathbb F_q}} \smsspace n d {\mathbb F_q}}_{m \text{ times}}$ and
$\sspace n {d,m} {\mathbb F_q}$ to denote $\underbrace{\sspace n d {\mathbb F_q} \times_{\espace d {\mathbb F_q}} \cdots \times_{\espace d {\mathbb F_q}} \sspace n d {\mathbb F_q}}_{m \text{ times}}$
	To establish parts (2) and (3), 
	we claim it is equivalent to show 
	$\lim_{\substack{q \ra \infty \\ \gcd(q,2n)=1}}  \frac{\# \smsspace n {d,m} {\mathbb F_q} (\mathbb F_q)}{\#\smespace d {\mathbb F_q}(\mathbb F_q)}$
	has values as given by the right hand sides of (2) and (3).
	To show this is the case, it is enough to show that both $\#\smespace d {\mathbb F_q}(\mathbb F_q)$ is within a factor of $1 + O_{n,d,m}(q^{-1/2})$ of the total number of height $d$ elliptic curves and
$\# \smsspace n {d,m} {\mathbb F_q}(\mathbb F_q)$
is within a factor of $1 + O_{n,d,m}(q^{-1/2})$ of the sum of $\#\sel_n(E)^m$ over all height $d$ elliptic curves.
First, $\#\smespace d {\mathbb F_q}(\mathbb F_q)$ certainly furnishes a lower bound for the size of the set of all elliptic curves of height $d$,
while $\espace d {\mathbb F_q}(\mathbb F_q)$ furnishes an upper bound (it is only an upper bound because it includes non-minimal smooth elliptic curves).
Next, using \autoref{lemma:selmer-fiber-estimate} to compare $\# \sel_n(E)$ to $\# H^1(\mathbb P^1, \mathscr E^0[n])$, for $E$ with smooth Weierstrass model, we find that 
$\# \smsspace n {d,m}{\mathbb F_q} (\mathbb F_q)$
indeed furnishes a lower bound for the sum of $\#\sel_n(E)^m$ over all height $d$ elliptic curves. 

Finally, 
to reduce to computing $\lim_{\substack{q \ra \infty \\ \gcd(q,2n)=1}}  \frac{\# \smsspace n {d,m} {\mathbb F_q} (\mathbb F_q)}{\#\smespace d {\mathbb F_q}(\mathbb F_q)}$
for (3),
we wish to show that up to a factor of $1 + O_{n,d,m}(q^{-1/2})$, 
$\# \smsspace n {d,m}{\mathbb F_q} (\mathbb F_q)$ also furnishes an upper bound for 
the sum of $\#\sel_n(E)^m$ over all height $d$ elliptic curves. 
Since $\sspace n d {\mathbb F_q} \to \espace d {\mathbb F_q}$ is \'etale and quasi-finite,
and $\smsspace n d {\mathbb F_q}$ constitutes a dense open in $\sspace n d {\mathbb F_q}$, it follows that 
$\smsspace n {d,m} {\mathbb F_q}$ 
constitutes a dense open in the maximal dimensional components of
$\sspace n {d,m} {\mathbb F_q}$.
Therefore, 
$\#\sspace n {d,m} {\mathbb F_q}(\mathbb F_q) - \#\smsspace n {d,m} {\mathbb F_q}$
is bounded by $O_{n,d,m}(q^{-1/2})$, using the Lang-Weil estimates.
The difference
$\#\sspace n {d,m} {\mathbb F_q}(\mathbb F_q) - \#\smsspace n {d,m} {\mathbb F_q}$
is not necessarily an upper bound for the sum of 
$\#\sel_n(E)^m$. However, as shown in
\cite[Corollary 3.27]{landesman:geometric-average-selmer},
it is an upper bound for the sum over all height $d$ elliptic curves of $\#(n^2 \cdot \sel_n(E))^m$.

To conclude, it remains to determine $\lim_{\substack{q \ra \infty \\ \gcd(q,2n)=1}}  \frac{\# \smsspace n {d,m} {\mathbb F_q} (\mathbb F_q)}{\#\smespace d {\mathbb F_q}(\mathbb F_q)}$.
%
Using the Lang-Weil estimates as in \cite[Lemma 5.1]{landesman:geometric-average-selmer}, it is enough to compute the number of geometrically irreducible components of 
$\# \smsspace n {d,m} {\mathbb F_q}$.
	Now, Part (3) follows from Burnside's lemma for the action of $\im \mono n d k$ acting diagonally on $(\vsel n d k)^m$, which we claim has a total of
	$\prod_{\ell \mid n}\prod_{i=1}^m \left( \ell^i +1 \right)$
	orbits.
	Note that $\Omega(\qsel n d k) \subset \im \mono n d k \subset \o(\qsel n d k)$,
	so it suffices to show both
	$\Omega(\qsel n d k)$ and $\o(\qsel n d k)$ have 
$\prod_{\ell \mid n}\prod_{i=1}^m \left( \ell^i +1 \right)$ orbits on $(\vsel n d k)^m$.
This follows from \autoref{theorem:fulman-stanton} and \autoref{lemma:same-orbits}, together with the Chinese remainder theorem to bootstrap this latter result from primes to squarefree integers.
\end{proof}

\subsection*{Conflict of interest}
On behalf of all authors, the corresponding author states that there is no conflict of interest.
\subsection*{Data availability}
Data sharing not applicable to this article as no datasets were generated or analysed during the current study.
		
\bibliographystyle{alpha}
\bibliography{bibliography}

\end{document}